\documentclass[a4paper, 9pt]{article}
\usepackage{graphics}
\usepackage[dvips]{color}

\usepackage{lscape}
\usepackage{latexsym}
\usepackage{amsmath,verbatim}
\usepackage{graphics}
\usepackage{amsthm}
\usepackage{amssymb}
\usepackage{makeidx}
\usepackage{stmaryrd}
\usepackage{multicol}
\usepackage{helvet}
\usepackage{bm}
\usepackage[all]{xy}
\newfont{\Fr}{eufm10}
\newfont{\Sc}{eusm10}
\newfont{\Bb}{msbm10}
\newfont{\Am}{msam10}
\newfont{\am}{msam7}
\numberwithin{equation}{section}
\newtheorem{theorem}{Theorem}[section]
\newtheorem{proposition}[theorem]{Proposition}
\newtheorem{lemma}[theorem]{Lemma}
\newtheorem{corollary}[theorem]{Corollary}
\newtheorem{claim}{Claim}{\bf}{\it}

\newtheorem{ftheorem}{Theorem}{\bf}{\it}
{\bf}{\it}
\theoremstyle{definition}
\newtheorem{definition}[theorem]{Definition}

{\bf}{\rm}

\theoremstyle{remark}

\newtheorem{remark}[theorem]{Remark}
{\bf}{\it}
\newtheorem{cond}[theorem]{Condition}

\newtheorem{definition and corollary}[theorem]{Definition and Corollary}

\newtheorem{fexample}[ftheorem]{Example}{\it}{\rm}

\newcommand{\h}{\mathfrac{\h}}

\title{An algebraic study of extension algebras\footnote{This is a revised and expanded version of the first part of the paper ``PBW bases and KLR algebras" arXiv:1203.5254. We divided it into two pieces since we learned that this part were invisible in the previous version of the above mentioned paper.}}
\author{Syu \textsc{Kato} \footnote{Department of Mathematics, Kyoto University, Oiwake Kita-Shirakawa Sakyo Kyoto 606-8502, Japan. \tt{E-mail:syuchan@math.kyoto-u.ac.jp}} \footnote{Research supported in part by JSPS Grant-in-Aid for Young Scientists (B) 23-740014.}}

\begin{document}
\maketitle

\begin{abstract}
We present simple conditions which guarantee a geometric extension algebra to behave like a variant of quasi-hereditary algebras. In particular, standard modules of affine Hecke algebras of type $\mathsf{BC}$, and the quiver Schur algebras are shown to satisfy the Brauer-Humphreys type reciprocity and the semi-orthogonality property. In addition, we present a new criterion of purity of weights in the geometric side. This yields a proof of Shoji's conjecture on limit symbols of type $\mathsf{B}$ [Shoji, Adv. Stud. Pure Math. 40 (2004)], and the purity of the exotic Springer fibers [K, Duke Math. 148 (2009)]. Using this, we describe the leading terms of the $C^{\infty}$-realization of a solution of the Lieb-McGuire system in the appendix. In [K, arXiv:1203.5254], we apply the results of this paper to the KLR algebras of type $\mathsf{ADE}$ to establish Kashwara's problem and Lusztig's conjecture.
\end{abstract}

\section*{Introduction}
In representation theory of an algebra associated to a root datum, study of a geometric extension algebra plays a major r\^ole. Introduced by Ginzburg \cite{Gi} in his study of affine Hecke algebras, it appeared in the study of affine Hecke algebras \cite{L-CG1,A,CG,K1,VV-B}, the BGG categories \cite{So,BGS,ABG,Sc}, the Springer correspondence \cite{CG,AcG,K4,Ri}, quantum loop algebras \cite{Na}, the Khovanov-Lauda-Rouquier algebras \cite{Z,W,VV}, the quiver Schur algebras \cite{VV2,WS}, and so on. Once appeared, it produces deep results in the spirit of the Kazhdan-Lusztig conjecture \cite{BB,BK} and the Koszul-Langlands duality \cite{BGS,So2,ABG}.

However, not much is known about the standard modules of these geometric extension algebras. From the viewpoint of highest weight category \cite{CPS}, we have several natural expectations on standard modules. For example, they should be indecomposable with simple heads, and filter projective modules. These are not a part of the general theory of geometric extension algebras \cite{CG} \S 8, and even its indecomposablity is usually guaranteed by rather ad-hoc induction arguments.

Actually, several criteria which guarantee nice behavior of geometric extension algebras are known \cite{MV,BGS}. The problem is that many of the geometric extension algebras arising from representation theory fail to satisfy such criteria. In particular, the resulting geometric extension algebras rarely give rise to highest weight categories.

The goal of the present paper is to supply some geometric conditions $(\clubsuit)$ which cover some algebras that are not covered by previous results, and to deduce their representation-theoretic consequences. Such an analysis, together with its algebraic interpretations, yields a proof of Shoji's conjecture in this paper. In addition, it serves as a basis of our proofs of Kashiwara's problem and Lusztig's conjecture in \cite{K6}.

Let $G$ be a connected algebraic group acting on a variety $\mathfrak X$ over $\mathbb C$ with finitely many orbits $\{ \mathbb O _{\lambda} \} _{\lambda \in \Lambda}$ labeled by $\Lambda$. Let $\mathsf{IC} _{\lambda}$ be the minimal extension of the constant sheaf on $\mathbb O _{\lambda}$ (see e.g. \cite{BBD}). We assume the following three conditions unless stated otherwise (here we employ an extra condition $(\spadesuit)'$ in order to simplify the statements in this introduction):

\begin{itemize}
\item[$(\spadesuit)'$] For each $\lambda \in \Lambda$, the $G$-orbit $\mathbb O _{\lambda}$ has a connected stabilizer $G_{\lambda}$;
\item[$(\clubsuit)_1$] The algebra $A _{(G,\mathfrak X)}$ defined in the below is pure of weight zero;
\item[$(\clubsuit)_2$] The sheaf $\mathsf{IC}_{\lambda}$ is pointwise pure (of some weight) for every $\lambda \in \Lambda$.
\end{itemize}

Here we count weights by considering its positive characteristic analogue. For each $\lambda \in \Lambda$, we have a natural morphism $\psi_{\lambda} : H ^{\bullet} _{G} ( \{ \mathrm{pt} \} ) \rightarrow H ^{\bullet} _{G} ( \mathbb{O} _{\lambda} )$ (between equivariant cohomologies; see e.g. \cite{BL}) of graded algebras. The condition $(\clubsuit)_2$ is usually rather difficult to verify. For this, our analysis shows:

\begin{ftheorem}[$=$ Theorem \ref{critpure}]\label{fcritpure}
We assume $(\clubsuit)_1$, but not $(\clubsuit)_2$. We have $(\clubsuit)_2$ as a consequence of the following two conditions:
\begin{enumerate}
\item For each $\lambda, \mu \in \Lambda$, the stalk of $\mathsf{IC}_{\lambda}$ along $\mathbb O_{\mu}$ satisfies the parity vanishing;
\item We have $\ker \, \psi _{\lambda} \not\subset \ker \, \psi _{\mu}$ for every $\mu \prec \lambda \in \Lambda$.
\end{enumerate}
\end{ftheorem}
The geometric extension algebra associated to such a pair $(G,\mathfrak X)$ is:
$$A = A_{(G,\mathfrak X)} := \bigoplus _{\lambda,\mu} \mathrm{Ext} ^{\bullet} _{D^b_G(\mathfrak X)} ( L _{\lambda} \boxtimes \mathsf{IC} _{\lambda} [\dim \mathbb O_{\lambda} ], L _{\mu} \boxtimes \mathsf{IC} _{\mu} [\dim \mathbb O_{\lambda} ] ),$$
where $L_{\lambda}$ is a self-dual non-zero graded vector space for each $\lambda \in \Lambda$. This makes $\{ L_{\lambda} \} _{\lambda}$ a complete collection of self-dual simple graded modules of $A$. Geometric extension algebras arising from affine Hecke algebras of type $\mathsf{A}$ \cite{CG}, type $\mathsf{BC}$ \cite{K1}, the KLR algebras of type $\mathsf{ADE}$ \cite{VV,K6}, the quiver Schur algebras \cite{Lu, VV2}, and the algebra which governs the BGG category \cite{So,BGS} satisfy our condition. In addition, our condition is stable under the restriction to a locally closed subset (cf. Corollary \ref{localKsc} and Lemma \ref{cc}). Hence, one obtains many intermediate varieties and algebras (with some representation-theoretic significance) that are connected to each other (see \cite{KP} for example).

\begin{ftheorem}[$=$ Theorem 3.5]\label{fabsKas}
The global dimension of $A$ is finite.
\end{ftheorem}

This is new for the quiver Schur algebras and the KLR algebras. For more precise explanation on the latter, see \cite{K6} (cf. \cite{BKM}). For affine Hecke algebras, this is standard by taking filtered quotients (cf. \cite{NO} D.VII). In this case, an explicit bi-resolution of $A$ is constructed by Opdam-Solleveld \cite{OS}.

Since we deal with a variety $\mathfrak X$, we have a closure ordering among $G$-orbits that we denote by $\prec$. Let $A \mathchar`-\mathsf{gmod}$ be the category of finitely generated graded $A$-modules. Let $P_{\lambda}$ be the projective cover of $L _{\lambda}$. For $M \in A \mathchar`-\mathsf{gmod}$ and $j \in \mathbb Z$, we denote by $M \left< j \right>$ the grade $j$ shift of $M$. We define
$$\widetilde{K} _{\lambda} := P _{\lambda} / \bigl( \sum _{\mu \prec \lambda, f \in \mathrm{hom} _A ( P_{\mu}, P _{\lambda} )} \mathrm{Im} f \bigr) \hskip 3mm \text{ and } \hskip 3mm K_{\lambda} := \widetilde{K} _{\lambda} / \bigl( \sum _{f \in \mathrm{hom} _A ( P_{\lambda}, \widetilde{K} _{\lambda} ) ^{>0}} \mathrm{Im} f \bigr),$$
where $\mathrm{hom} _A (M,N)$ is the direct sum of the space of degree $j$ homomorphisms $\mathrm{hom} _A (M,N)^j := \mathrm{Hom} _{A \mathchar`-\mathsf{gmod}} ( M \left< j \right>, N )$ (for each $M, N \in A \mathchar`-\mathsf{gmod}$).

For $M \in A \mathchar`-\mathsf{gmod}$, we define its graded character as:
$$\mathsf{gch} \, M := \sum _{\lambda \in \Lambda, j \in \mathbb Z} t ^j [ M : L_{\lambda} \left< j \right>]_0 [ \lambda ] \in \bigoplus _{\lambda \in \Lambda} \mathbb Z (\!(t)\!) [\lambda],$$
where $[ M : L_{\lambda} \left< j \right>]_0 \in \mathbb Z_{\ge 0}$ is the multiplicity of $L_{\lambda} \left< j \right>$ in $M$. We define $[ M : L_{\lambda} ] := \sum _{j \in \mathbb Z} [ M : L_{\lambda} \left< j \right>]_0 t^j \in \mathbb Z(\!( t )\!)$ and define $\{ [ M : \widetilde{K} _{\lambda} ] \} _{\lambda}$ by
$$\mathsf{gch} \, M = \sum _{\lambda \in \Lambda} [ M : \widetilde{K} _{\lambda} ] \, \mathsf{gch} \, \widetilde{K} _{\lambda}.$$

\begin{ftheorem}[$=$ Theorems \ref{KS}, \ref{absLS}, and \ref{wfilt} + Corollary \ref{absBH}]\label{fcore} Under the above setting, we have:
\begin{enumerate}
\item Each $\widetilde{K} _{\lambda}$ admits a separable decreasing filtration whose associated graded is a direct sum of grading shifts of $K_{\lambda}$. In addition, we have
$$[\widetilde{K} _{\lambda} : L _{\lambda} ] = \sum _{i \ge 0} t^{i} \dim \, H ^{i} _{\mathsf{Stab} _G ( x_{\lambda} )} ( \{ \mathrm{pt} \} );$$
\item For each $\lambda, \mu \in \Lambda$, we have
\begin{align*}
& \mathrm{ext} _A ^{\bullet} ( \widetilde{K} _{\lambda}, \widetilde{K} _{\mu} ) = \{ 0 \} \hskip 3mm \text{ if } \hskip 3mm \lambda \not\succsim \mu \hskip 3mm \text{ and }\\
& \mathrm{ext} _A ^i ( \widetilde{K} _{\lambda}, K _{\mu} ^* ) = \begin{cases} \mathbb C & ( \lambda = \mu, i = 0)\\ \{ 0 \} & (\text{otherwise}) \end{cases},
\end{align*}
where $K_{\mu}^*$ is the graded dual of $K_{\mu}$ regarded as an $A$-module;
\item We have
$$[ P_{\lambda} : \widetilde{K}_{\mu} ] = [K_{\mu} : L _{\lambda}] \hskip 3mm \text{ for every } \hskip 3mm \lambda, \mu \in \Lambda;$$
\item Each $P_{\lambda}$ admits a finite filtration whose associated graded is a direct sum of grading shifts of $\{ \widetilde{K} _{\mu} \} _{\mu}$.
\end{enumerate}
\end{ftheorem}

Note that the non-zero map in Theorem \ref{fcore} 2) has its image $L_{\lambda}$ since essentially the LHS is concentrated in degree $\ge 0$, and the RHS is concentrated in degree $\le 0$. Also, Theorem \ref{fcore} 1) yields the Cartan determinant formula of $A$ (Corollary \ref{cartan}) and a variant of the Lusztig-Shoji algorithm (Remark \ref{Rcartan} {\bf 2)}).

For the proof of Theorem \ref{fcore}, we make a link between the mixed geometry of $\mathfrak X$ and the homological algebra of $A$ in \S 2, which might be of independent interest. One important consequence there is a description of the minimal projective resolution of $\widetilde{K}_{\lambda}$ in terms of $\mathbb C _{\lambda}$ (Corollary \ref{kf}).

\begin{fexample}\label{fA3}
Let $G=\mathop{PGL} ( 3, \mathbb C )$ and let $\mathfrak X = \mathcal N$ be the nilpotent cone of $\mathfrak{sl}_3$. Let $\mathfrak t$ be the Cartan subalgebra of $\mathfrak{sl}_3$. We have $A = \mathbb C \mathfrak S_3 \ltimes \mathbb C [ \mathfrak t ]$ with $A^0 = \mathbb C \mathfrak S_3$ and $\deg \, \mathfrak t^* = 2$. We have $\Lambda = \{ \mathsf{triv} \succ \mathsf{ref} \succ \mathsf{sgn} \}$ and
$$
\mathsf{gch} \, K _{\mathsf{triv}} = [ \mathsf{triv} ], \mathsf{gch} \, K _{\mathsf{ref}} = [ \mathsf{ref} ] + t^2 [ \mathsf{triv} ], \mathsf{gch} \, K _{\mathsf{sgn}} = [ \mathsf{sgn} ] + ( t^2 + t^4 ) [ \mathsf{ref} ] + t^6 [ \mathsf{triv} ].
$$
Then, the matrix $[P:L] := ( [P_{\lambda} : L _{\gamma}] ) _{\lambda, \gamma}$ is presented as:
$$[P:L] = [P:\widetilde{K}][\widetilde{K}:K][K:L] = {}^{\mathtt t} [K:L][\widetilde{K}:K][K:L],$$
where $[\widetilde{K}:K]$ is the graded character transition matrix. This reads as:{\footnotesize
\begin{align*}
\frac{1}{(1-t^4)(1-t^6)} & \left( \begin{matrix} 1 & t^2 + t^4 & t^6\\ t^2+t^4 & 1 + t^2 + t^4 + t^6 & t^2 + t^4\\ t^6 & t ^2 + t^4 & 1 \end{matrix} \right)\\
&= \left( \begin{matrix}1 & t^2 & t^6\\ 0 & 1 & t^2 + t^4\\ 0 & 0 & 1 \end{matrix}\right) \left( \begin{matrix} 1 & 0 & 0\\ 0& \frac{1}{1-t^2} & 0 \\ 0 & 0 & \frac{1}{(1-t^4)(1-t^6)} \end{matrix} \right) \left( \begin{matrix}1 & 0 & 0\\ t^2 & 1 & 0\\ t^6 & t^2 + t^4 & 1 \end{matrix}\right).
\end{align*}}
\end{fexample}

In \cite{Sh4} 3.13, Shoji conjectured that his limit Green function of type $\mathsf{B}$ gives the graded character of a coinvariant ring of a Weyl group of type $\mathsf{B}$. Combining the results of Achar-Henderson \cite{AH} with Theorem \ref{fcore} and our previous results, we prove:

\begin{ftheorem}[Shoji's conjecture for type $\mathsf B$ $=$ Corollaries \ref{SC}, \ref{SS}, and \ref{espure}]\label{fmain}
Let $G= \mathop{Sp} ( 2n, \mathbb C )$ and let $\mathfrak X = \mathfrak N$ be its exotic nilpotent cone $\cite{K1}$. Then, the modules $K_{\lambda}$ arising from $A$ are coinvariant algebras of type $\mathsf B$. In particular, their graded characters are calculated by the Lusztig-Shoji algorithm. Moreover, every exotic Springer fiber has a pure homology.
\end{ftheorem}

We note that Shoji-Sorlin \cite{SS} independently proved the last part of Theorem \ref{fmain} by a completely different method. Also, Theorem \ref{fmain} implies \cite{AH} Conjecture 6.4.

As a bonus of Theorem \ref{fmain}, we prove that the graded modules appearing from the Lieb-McGuire integrable systems \cite{HO} are exactly Shoji's coinvariant algebras in Appendix A.

The organization of this paper is as follows: In the first section, we formulate two conditions $(\spadesuit)$ and $(\clubsuit)$ and our main results (Theorem \ref{fcritpure} and the main part of Theorem \ref{fcore}). In the second section, we translate an iteration of distinguished triangles to a complex of projective modules to prove the main part of Theorem \ref{fcore}. In the third section, we prove Theorem \ref{fabsKas} and numerical consequences of Theorem \ref{fcore} including the Brauer-Humphreys type reciprocity and the Cartan determinant formula. In the fourth section, we complete the proof of Theorem \ref{fcore}, and show that our conditions are invariant under restrictions. Finally, we prove Shoji's conjecture in the fifth section. The appendices are devoted to the analysis of the Lieb-McGuire system and a proof of Theorem \ref{fcritpure}.

Theorem \ref{fcore} 3) shows that we have two versions of standard modules and they are deeply incorporated into the formulation. In addition, two versions of standard modules actually differ in Example \ref{fA3}. Therefore, our result mainly points different direction from that of Cline-Parshall-Scott \cite{CPS}. Also, there is a notion of affine cellularity due to K\"onig-Xi \cite{KX}, which provides a framework for algebraic results similar to ours. However, their algebraic conditions seem rather difficult to verify compared with our geometric conditions\footnote{For example, the combination of \cite{K1} and \S 1 yields homological/categorical consequences parallel to the theory of affine cellular algebras for the affine Hecke algebras of type $\mathsf{BC}$ with arbitrary rank and arbitrary real parameters (while the affine cellularity of the affine Hecke algebras of type $\mathsf{BC}$ is known only in rank two case with generic real parameters \cite{GM}).}. In addition, our approach gives some more precise information on $A$ than affine cellularity (in a sense) as our approach naturally gives a variant of the Kazhdan-Lusztig polynomials given geometrically.

\begin{flushleft}
{\small
{\bf Acknowledgement:} The author is indebted to Masaki Kashiwara for helpful discussions and comments. He is also indebted to Yoshiyuki Kimura for helpful comments and pointing out some errors. He is also grateful to Ryosuke Kodera for his question which leads him to Theorem \ref{wfilt}, Eric Opdam for suggesting him to study the material presented in Appendix A, Julia Sauter for giving me some comments, and Wolfgang Soergel for suggesting him to study extension algebras.}
\end{flushleft}

\begin{flushleft}
{\small {\bf Convention}}
\end{flushleft}

An algebra $R$ is a (not necessarily commutative) unital $\mathbb C$-algebra. A variety $\mathfrak X$ is a separated reduced scheme $\mathfrak X_0$ of finite type over some localization $\mathbb Z_S$ of $\mathbb Z$ specialized to an algebraically closed field $\Bbbk$ that we specify. It is called a $G$-variety if we have an action of a connected affine algebraic group scheme $G$ flat over $\mathbb Z_S$ on $\mathfrak X_0$ specialized to $\Bbbk$.

We fix some prime $\ell$ and fix an {\it identification} $\overline{\mathbb Q}_{\ell} \cong \mathbb C$ once for all. Let us denote by $D^b ( \mathfrak X )$ (resp. $D^+ ( \mathfrak X )$) the bounded (resp. bounded from the below) derived category of the category of constructible sheaves on $\mathfrak X$, and denote by $D^+ _G ( \mathfrak X )$ the $G$-equivariant derived category of $\mathfrak X$. We have a natural forgetful functor $D^+ _G ( \mathfrak X ) \to D^+ ( \mathfrak X )$, whose preimage of $D^b ( \mathfrak X )$ is denoted by $D ^b _G ( \mathfrak X )$. For an object of $D^b_G ( \mathfrak X )$, we may denote its image in $D ^b ( \mathfrak X )$ by the same letter.

Let $\mathsf{vec}$ be the category of $\mathbb Z$-graded vector spaces (over $\mathbb C$) bounded from the below so that its objects have finite-dimensional graded pieces. In particular, for $V = \oplus _{i \gg - \infty} V^i \in \mathsf{vec}$, its graded dimension $\mathsf{gdim} \, V := \sum _{i} t ^i \dim V ^i \in \mathbb Z (\!( t )\!)$ makes sense (with $t$ being indeterminant). We define $V \left<m\right>$ by setting $( V \left<m\right> ) ^i := V ^{i - m}$.

In this paper, a graded algebra $A$ is always a $\mathbb C$-algebra whose underlying space is in $\mathsf{vec}$. Let $A \mathchar`-\mathsf{gmod}$ be the category of finitely generated graded $A$-modules. For $E, F \in A \mathchar`-\mathsf{gmod}$, we define $\hom_{A} ( E, F )$ to be the direct sum of graded $A$-module homomorphisms $\hom _{A} ( E, F )^j$ of degree $j$ ($= \mathrm{Hom} _{A \mathchar`-\mathsf{gmod}} ( E \left< j \right>, F )$). We employ the same notation for extensions (i.e. $\mathrm{ext}_{A} ^i ( E, F ) = \oplus_{j\in \mathbb Z} \mathrm{ext}^{i}_{A} (E,F)^j$). We denote by $\mathsf{Irr} \, A$ the set of isomorphism classes of graded simple modules of $A$, and denote by $\mathsf{Irr} _0 \, A$ the set of isomorphism classes of graded simple modules of $A$ up to grading shifts. Two graded algebras are said to be (graded) Morita equivalent if their (graded) module categories are equivalent. For a graded $A$-module $E$, we denote its head by $\mathsf{hd} \, E$, and its socle by $\mathsf{soc} \, E$.

For $Q (t) \in \mathbb Q (t)$, we set $\overline{Q ( t )} := Q ( t^{-1} )$. We denote the set of isomorphism classes of a finite group $H$ by $\mathsf{Irr} \, H$. For each $\chi \in \mathsf{Irr} \, H$, we denote its dual representation by $\chi^{\vee}$. For derived functors $\mathbb R F$ or $\mathbb L F$ of some functor $F$, we represent its arbitrary graded piece (of its homology complex) by $\mathbb R ^{*} F$ or $\mathbb L ^{*} F$, and the direct sum of whole graded pieces by $\mathbb R ^{\bullet} F$ or $\mathbb L ^{\bullet} F$. For example, $\mathbb R ^{*} F \cong \mathbb R ^{*} G$ means that $\mathbb R ^{i} F \cong \mathbb R ^{i} G$ for every $i \in \mathbb Z$, while $\mathbb R ^{\bullet} F \cong \mathbb R ^{\bullet} G$ means that $\bigoplus _{i} \mathbb R ^{i} F \cong \bigoplus _{i} \mathbb R ^{i} G$.

When working on some sort of derived category, we suppress $\mathbb R$ or $\mathbb L$, or the category from the notation for the sake of simplicity in case there is only small risk of confusion.

\section{The conditions $(\spadesuit), (\clubsuit)$ and a main result}\label{general_Ext}
For a while, we fix our ground field $\Bbbk$ to be either $\mathbb C$ or the algebraic closure of a finite field. Let $G$ be a connected (affine) algebraic group. Let $\mathfrak X$ be a $G$-variety. Let $\underline{\Lambda}$ be the labeling set of $G$-orbits of $\mathfrak X$. We denote by $\mathbb O_{\underline{\lambda}}$ the $G$-orbit corresponding to $\underline{\lambda} \in \underline{\Lambda}$. Let $C_{\underline{\lambda}}$ be the component group of the stabilizer $G_{\underline{\lambda}}$ of a closed point of $\mathbb O_{\underline{\lambda}}$. It is always a finite group. Let $\Lambda$ be the set of conjugacy classes of pairs $(\underline{\lambda},\xi)$ with $\underline{\lambda} \in \underline{\Lambda}$ and $\xi \in \mathsf{Irr} \, C_{\lambda}$. For $\lambda = ( \underline{\lambda}, \xi ) \in \Lambda$, we set $\mathbb O _{\lambda} := \mathbb O _{\underline{\lambda}}$, $C _{\lambda} := C _{\underline{\lambda}}$, $G_{\lambda} := G_{\underline{\lambda}}$, and $\lambda^{\vee} := ( \underline{\lambda}, \xi^{\vee} )$. For $\lambda, \mu \in \Lambda$, we write $\lambda \precsim \mu$ if $\mathbb O _{\lambda} \subset \overline{\mathbb O _{\mu}}$, and write $\lambda \sim \mu$ if $\mathbb O _{\lambda} = \mathbb O _{\mu}$. We assume the following condition $(\spadesuit)$ in the below:

\begin{cond}[Condition $(\spadesuit)$]\label{spadesuit}
The set $\underline{\Lambda}$ (or equivalently $\Lambda$) is finite. For each $\lambda \in \Lambda$, we fix a $\Bbbk$-valued point $x_{\lambda} \in \mathbb O _{\lambda}$ such that $x_{\lambda} = x _{\mu}$ if $\lambda \sim \mu$.
\end{cond}

We have a (relative) dualizing complex $\omega_{\mathfrak X} := p^! \underline{\mathbb C} \in D^b _{G} ( \mathfrak X )$, where $p : \mathfrak X \to \{\mathrm{pt}\}$ is the $G$-equivariant structure map. We have a dualizing functor
$$\mathbb D : D ^b _G ( \mathfrak X ) ^{op} \ni C^{\bullet} \mapsto \mathcal{H}om ^{\bullet} ( C ^{\bullet}, \omega_{\mathfrak X} ) \in D^b _G ( \mathfrak X ).$$

We have a $\mathbb D$-autodual $t$-structure of $D^b _G ( \mathfrak X )$ whose truncation functor and perverse cohomology functor are denoted by $\tau$ and ${}^p H$, respectively. In particular, $\mathbb D$ induces an equivalence of categories $\tau _{\ge 0} D^b _G ( \mathfrak X ) ^{op} \cong \tau _{\le 0} D^b _G ( \mathfrak X )$ and each $\mathcal E \in D^b _G ( \mathfrak X )$ admits a distinguished triangle
$$\tau _{< m} \mathcal E \to \mathcal E \to \tau _{\ge m} \mathcal E \stackrel{+1}{\longrightarrow} \tau _{< m} \mathcal E [1]$$
for every $m \in \mathbb Z$.

For each $\lambda := ( \underline{\lambda}, \xi ) \in \Lambda$, we have a $G$-equivariant irreducible local system $\underline{\xi}$ on $\mathbb O _{\lambda}$ induced from $\xi$. We have inclusions $i_{\lambda} : \{ x_{\lambda} \} \hookrightarrow \mathfrak X$ and $j _{\lambda} : \mathbb O_{\lambda} \hookrightarrow \mathfrak X$. Let
\begin{equation}
\mathbb C _{\lambda} := ( j _{\lambda} ) _! \underline{\xi} \, [ \dim \mathbb O _{\lambda} ] \hskip 3mm \text{ and } \hskip 3mm \mathsf{IC} _{\lambda} := ( j _{\lambda} ) _{!*} \underline{\xi} \, [ \dim \mathbb O _{\lambda} ] \label{ICdef}
\end{equation}
be the extension by zero and the minimal extension, which we regard as objects of $D^b _G ( \mathfrak X )$. We denote by
\begin{align*}
\mathrm{Ext} ^{\bullet} _G ( \bullet, \bullet ) & : D^b _G ( \mathfrak X ) ^{op} \times D^b _G ( \mathfrak X ) \longrightarrow D^+ ( \{\mathrm{pt}\} )\\
\mathrm{Ext} ^{\bullet} ( \bullet, \bullet ) & : D ^b ( \mathfrak X ) ^{op} \times D ^b ( \mathfrak X ) \longrightarrow D ^b ( \{\mathrm{pt}\} )
\end{align*}
the Ext (as bifunctors) of $D^b_G ( \mathfrak X )$ and $D^b ( \mathfrak X )$, respectively.

For each $\lambda \in \Lambda$, we fix $L _{\lambda} \in D ^b ( \{ \mathrm{pt} \} )$ as a non-zero graded vector space with a trivial differential which satisfies the duality condition $L _{\lambda ^{\vee}} \cong L _{\lambda} ^*$. We set
$$\mathcal L := \bigoplus _{\lambda \in \Lambda} L_{\lambda} \boxtimes \mathsf{IC} _{\lambda} \in D _G ^b ( \mathfrak X ).$$
By construction, we find an isomorphism $\mathcal L \cong \mathbb D \mathcal L$.

We form a graded Yoneda algebra
$$A _{(G, \mathfrak X)} = \bigoplus _{i \in \mathbb Z} A ^i _{(G, \mathfrak X)} := \bigoplus _{i \in \mathbb Z} \mathrm{Ext} _G ^i ( \mathcal L, \mathcal L )$$
whose degree is the cohomological degree. We denote by $B _{(G, \mathfrak X)}$ the algebra $A_{(G, \mathfrak X)}$ obtained by taking $\mathcal L = \bigoplus _{\lambda \in \Lambda} \mathsf{IC} _{\lambda}$ (and call it the basic ring of $A_{(G, \mathfrak X)}$). The algebra $B_{(G, \mathfrak X)}$ is graded Morita equivalent to $A _{(G, \mathfrak X)}$ (see Lemma \ref{gMe} in the below), and hence all the arguments in the below are independent of the choice of $\mathcal L$, that we suppress for simplicity. We also drop $(G, \mathfrak X)$ in case the meaning is clear from the context. It is standard that $B _{(G, \mathfrak X)}$ is non-negatively graded (as $\mathrm{Ext} ^{< 0} _G ( \mathsf{IC} _{\lambda}, \mathsf{IC} _{\mu} ) \equiv \{ 0 \}$) and $\{L_{\lambda}\}_{\lambda \in \Lambda}$ forms a complete collection of graded simple $A$-modules up to grading shifts (as $\bigoplus _{\lambda \in \Lambda} \mathrm{end} _{\mathbb C} (L_{\lambda})$ is the maximal graded semisimple quotient of $A$ by $\dim \mathrm{Ext} ^{0} _G ( \mathsf{IC} _{\lambda}, \mathsf{IC} _{\mu} ) = \delta _{\lambda, \mu}$). We have $\mathcal{E}xt ^{\bullet} ( \mathcal L, \mathcal L ) \in D ^b ( \mathfrak X)$, which implies $\dim \mathrm{Ext} ^{\bullet} ( \mathcal L, \mathcal L ) < \infty$. By the Serre spectral sequence
\begin{equation}
H ^{\bullet} _G ( \{ \mathrm{pt} \} ) \otimes _{\mathbb C} \mathrm{Ext} ^{\bullet} ( \mathcal L, \mathcal L ) \Rightarrow \mathrm{Ext} ^{\bullet}_G ( \mathcal L, \mathcal L ) \cong A,\label{SS1}
\end{equation}
we conclude that $A \in \mathsf{vec}$.

\begin{lemma}[cf. Joshua \cite{J} \S 6]\label{symm}
We have an isomorphism $A \cong A^{op}$ of graded algebras induced by an isomorphism $\mathcal L \cong \mathbb D \mathcal L$. In addition, the graded dual of $L_{\lambda}$ is naturally isomorphic to $L_{\lambda ^{\vee}}$ as a graded $A$-module for each $\lambda \in \Lambda$.
\end{lemma}

\begin{proof}
The isomorphism $\mathbb D \mathcal L \cong \mathcal L$ induces an isomorphism $\mathbb D ( L_{\lambda} \boxtimes \mathsf{IC} _{\lambda}) \cong L _{\lambda ^{\vee}} \boxtimes \mathsf{IC} _{\lambda ^{\vee}}$ for each $\lambda \in \Lambda$. This determines the both of $L_{\lambda} ^* \cong L _{\lambda ^{\vee}}$ and $\mathbb D \mathsf{IC} _{\lambda} \cong \mathsf{IC} _{\lambda ^{\vee}}$ by the simplicity of $\mathsf{IC} _{\lambda}$ (up to a scalar).

For each $\lambda, \mu \in \Lambda$, we have a natural identification
$$\mathrm{hom} _{\mathbb C} ( L_{\lambda}, L_{\mu} ) \boxtimes \mathrm{Ext} _G^{\bullet} ( \mathsf{IC}_{\lambda}, \mathsf{IC} _{\mu} ) = \mathrm{Ext} _G^{\bullet} ( L_{\lambda} \boxtimes \mathsf{IC}_{\lambda}, L_{\mu} \boxtimes \mathsf{IC} _{\mu} ) \subset A.$$
For each $\lambda, \mu, \gamma \in \Lambda$, this identification factors the multiplication map
\begin{equation}
\mathrm{Ext} _G^{\bullet} ( L_{\lambda} \boxtimes \mathsf{IC}_{\lambda}, L_{\mu} \boxtimes \mathsf{IC} _{\mu} ) \times \mathrm{Ext} _G^{\bullet} ( L_{\mu} \boxtimes \mathsf{IC}_{\mu}, L_{\gamma} \boxtimes \mathsf{IC} _{\gamma} ) \rightarrow \mathrm{Ext} _G^{\bullet} ( L_{\lambda} \boxtimes \mathsf{IC}_{\lambda}, L_{\gamma} \boxtimes \mathsf{IC} _{\gamma} )\label{assT}
\end{equation}
inside $A$ into the external tensor products of
\begin{align}
& \mathrm{hom} _{\mathbb C} ( L_{\lambda}, L_{\mu} ) \times \mathrm{hom} _{\mathbb C} ( L_{\mu}, L_{\gamma} ) \to \mathrm{hom} _{\mathbb C} ( L_{\lambda}, L_{\gamma} ), \text{ and}\label{assL}\\
& \mathrm{Ext} _G^{\bullet} ( \mathsf{IC}_{\lambda}, \mathsf{IC}_{\mu} ) \times \mathrm{Ext} _G^{\bullet} ( \mathsf{IC}_{\mu}, \mathsf{IC}_{\gamma} ) \to \mathrm{Ext} _G^{\bullet} ( \mathsf{IC}_{\lambda}, \mathsf{IC}_{\gamma} ).\label{assI}
\end{align}
The associativity of (\ref{assL}) and (\ref{assI}) implies that of (\ref{assT}). By applying $\mathbb D$, we have identifications
\begin{eqnarray*}
& \mathrm{hom} _{\mathbb C} ( L_{\lambda}, L_{\mu} ) \cong \mathrm{hom} _{\mathbb C} ( L_{\mu} ^*, L_{\lambda} ^* ) \cong \mathrm{hom} _{\mathbb C} ( L_{\mu^{\vee}}, L_{\lambda ^{\vee}} ), \text{ and }\\
&\mathrm{Ext} _G^{*} ( \mathsf{IC}_{\lambda}, \mathsf{IC} _{\mu} ) \cong \mathrm{Ext} _G^{*} ( \mathbb D \mathsf{IC}_{\mu}, \mathbb D \mathsf{IC} _{\lambda} ) \cong \mathrm{Ext} _G^{*} ( \mathsf{IC}_{\mu ^{\vee}}, \mathsf{IC} _{\lambda ^{\vee}} )
\end{eqnarray*}
for each $\lambda, \mu \in \Lambda$. These identifications commute with the compositions (\ref{assL}) and (\ref{assI}). Taking direct sums, we find an algebra isomorphism $\psi: A \cong A ^{op}$.

Since $\mathbb D$ exchanges $L_{\lambda} \boxtimes \mathsf{IC} _{\lambda}$ with $L _{\lambda ^{\vee}} \boxtimes \mathsf{IC} _{\lambda ^{\vee}}$, we have
\begin{align*}
A \supset \, & \mathrm{Ext} _G ^{\bullet} ( L _{\lambda} \boxtimes \mathsf{IC} _{\lambda}, L _{\lambda} \boxtimes \mathsf{IC} _{\lambda} ) \supset \mathrm{end} _{\mathbb C} ( L_{\lambda} )\\
& \stackrel{\psi}{\longrightarrow} \mathrm{end} _{\mathbb C} ( L_{\lambda ^{\vee}} ) ^{op} \subset \mathrm{Ext} _G ^{\bullet} ( L _{\lambda^{\vee}} \boxtimes \mathsf{IC} _{\lambda^{\vee}}, L _{\lambda^{\vee}} \boxtimes \mathsf{IC} _{\lambda^{\vee}} ) \subset A ^{op}.
\end{align*}
Therefore, $L_{\lambda} ^*$ naturally carries the structure of a graded $A$-module $L_{\lambda ^{\vee}}$ (via $\psi$) as required.
\end{proof}

\begin{lemma}\label{NA}
The algebra $A$ is left and right Noetherian.
\end{lemma}

\begin{proof}
Let $C$ be the image of the natural map $H ^{\bullet} _G ( \mathrm{pt} ) \longrightarrow A$. Since $H ^{\bullet} _G ( \mathrm{pt} )$ is a polynomial algebra, we deduce that $C$ is a Noetherian algebra. By (\ref{SS1}), $A$ is a finitely generated module over $C$. Hence, every $C$-submodule of $A$ is again finitely generated. This particularly applies to every left or right $A$-submodule of $A$ (i.e. ideals of $A$), and hence the result.
\end{proof}

For each $\lambda \in \Lambda$, we set
$$P_{\lambda} := \mathrm{Ext} _G ^{\bullet} ( \mathsf{IC} _{\lambda}, \mathcal L ) = \bigoplus _{i \in \mathbb Z} \mathrm{Ext} _G ^i ( \mathsf{IC} _{\lambda}, \mathcal L ).$$
Each $P_{\lambda}$ is a graded projective left $A$-module. By construction, we have
$$A \cong \bigoplus _{\lambda \in \Lambda} L _{\lambda} ^* \boxtimes \mathrm{Ext} _G ^{\bullet} ( \mathsf{IC} _{\lambda}, \mathcal L ) = \bigoplus _{\lambda \in \Lambda} P _{\lambda} \boxtimes L _{\lambda} ^*$$
as left $A$-modules. It is standard that $P_{\lambda}$ is an indecomposable $A$-module whose head is isomorphic to $L _{\lambda}$ (cf. \cite{CG} \S 8.7). We have an idempotent $e_{\lambda} \in A$ so that $P_{\lambda} \cong A e _{\lambda}$ as left graded $A$-modules (up to a grading shift).

For each $\lambda = ( \underline{\lambda}, \xi ) \in \Lambda$, we set
$$\widetilde{K} _{\lambda} := \mathrm{Ext} _G ^{\bullet} ( \mathbb C _{\lambda}, \mathcal L ) \text{ and } K _{\lambda} := \mathrm{Hom} _{C_{\lambda}} ( \xi, H ^{\bullet} i_{\lambda} ^! \mathcal L[ \dim \mathbb O _{\lambda} ] ).$$
We call $K _{\lambda}$ a standard module, and $\widetilde{K} _{\lambda}$ a dual standard module of $A$. Here the Serre-type spectral sequence takes the form:
\begin{align}\nonumber
E_2 & \cong \mathrm{Ext} ^{\bullet} _{G} ( \underline{\xi} \, [\dim \mathbb O_{\lambda}], j_{\lambda} ^! \mathcal L )\\\nonumber
& \cong \mathrm{Ext} ^{\bullet} _{G _{\lambda}} ( \xi \, [- \dim \mathbb O_{\lambda}], i_{\lambda} ^! \mathcal L )\\\nonumber
& \cong \mathrm{Ext} ^{\bullet + \dim \mathbb O_{\lambda}} _{G _{\lambda} ^{\circ}} ( \xi, i_{\lambda} ^! \mathcal L ) ^{C_{\lambda}}\\
& \cong \bigoplus _{\mu = ( \underline{\lambda}, \zeta ) \in \Lambda} \mathrm{Hom} _{C_{\lambda}} ( \xi, \zeta \otimes _{\mathbb C} H _{G_{\lambda}^{\circ}} ^{\bullet} ( \{ x_{\lambda} \} ) ) \boxtimes K _{\mu} \Rightarrow \widetilde{K} _{\lambda}.\label{leray}
\end{align}
Note that the first isomorphism is adjunction, the second is the induction equivalence (cf. \cite{BL} 2.6.2), the third is the Hochshild-Serre spectral sequence, and the fourth is expanding $i_{\lambda} ^! \mathcal L$ with the action of $C_{\lambda}$ recorded and the usual Serre spectral sequence.

We now specialize to the case where our base field $\Bbbk$ is the algebraic closure of a finite field $\mathbb F _q$ of cardinality $q$. We regard each $\mathsf{IC} _{\lambda}$ as a simple mixed perverse sheaf (of some weight) in the category of mixed sheaves on $\mathfrak X$ via \cite{BBD} \S 5, and each $L_{\lambda}$ as a mixed (complex of) vector space of weight zero. I.e. each $L_{\lambda} ^i$ is pure of weight $i$ in the sense that the geometric Frobenius acts by $q ^{i/2} \mathsf{id}$. It follows that the algebra $A$ and its standard modules $\{ K_{\lambda} \} _{\lambda \in \Lambda}$ acquire (mixed) weight structures.

In addition, we also equip $\mathbb C_{\lambda}$ (for $\lambda = ( \underline{\lambda}, \xi) \in \Lambda$) with a mixed weight structure. In particular, dual standard modules $\{ \widetilde{K}_{\lambda} \} _{\lambda \in \Lambda}$ also acquire some (mixed) weight structures.

\begin{cond}[Condition $(\clubsuit)$]\label{clubsuit}
The condition $(\clubsuit)$ consists of two subconditions:
\begin{itemize}
\item[$(\clubsuit)_1$] The algebra $A$ is pure of weight zero;
\item[$(\clubsuit)_2$] For each $\lambda \in \Lambda$, the sheaf $\mathsf{IC} _{\lambda}$ $($whose weight is normalized to be zero$)$ is pointwise pure of weight zero.
\end{itemize}
\end{cond}

\begin{remark}
{\bf 1)} The conditions $(\spadesuit)$, and $(\clubsuit)$ are closed under the restriction to a closed $G$-subvariety (cf. Lemma \ref{cc}); {\bf 2)} Thanks to Lusztig \cite{Lu} and Varagnolo-Vasserot \cite{VV2}, the quiver Schur algebras of type $\mathsf{A}$ also satisfy the conditions $(\spadesuit)$ and $(\clubsuit)$ by taking $\mathfrak X$ as the spaces of nilpotent representations of cyclic quivers.
\end{remark}

\begin{theorem}\label{critpure}
We assume $(\spadesuit)$ and $(\clubsuit)_1$, but not $(\clubsuit)_2$. We have $(\clubsuit)_2$ as a consequence of the following two conditions:
\begin{enumerate}
\item[$a)$] The spectral sequence $(\ref{leray})$ is $E_2$-degenerate for each $\lambda \in \Lambda$;
\item[$b)$] For each $\lambda \in \Lambda$, we have a natural morphism $\psi_{\lambda} : H ^{\bullet} _{G} ( \{ \mathrm{pt} \} ) \rightarrow H ^{\bullet} _{G} ( \mathbb{O} _{\lambda} )$ of graded algebras. We have
$$\ker \, \psi _{\lambda} \not\subset \ker \, \psi _{\gamma} \hskip 5mm \text{ for every } \hskip 5mm \gamma \prec \lambda \in \Lambda;$$
\end{enumerate}
Here the condition $a)$ can be replaced by its variant:
\begin{enumerate}
\item[$a)'$] For each $\lambda, \mu \in \Lambda$, the stalk of $\mathsf{IC}_{\lambda}$ along $\mathbb O_{\mu}$ satisfies the parity vanishing.
\end{enumerate}
\end{theorem}

The proof of Theorem \ref{critpure} is given in Appendix B.

\begin{remark}\label{rcp}
{\bf 1)} The condition $(\clubsuit)_2$ implies the condition a) of Theorem \ref{critpure} since $H^{\bullet} _M ( \{ \mathrm{pt} \} )$ is pure for an affine algebraic group $M$; {\bf 2)} Theorem \ref{critpure} presents new proofs of pointwise purity of $($equivariant$)$ intersection cohomology complexes on some varieties from the structure of the affine Hecke algebras of type $\mathsf{BC}$ with $2$-parameters \cite{K1}, and the quiver Schur algebras \cite{Lu,VV2}.
\end{remark}

For $M \in A \mathchar`-\mathsf{gmod}$ and $i \in \mathbb Z$, we define
\begin{align*}
& [ M : L_{\lambda} \left< i \right> ]_0 := \dim \, \mathrm{Hom} _{A\mathchar`-\mathsf{gmod}} ( P_{\lambda} \left< i \right>, M ) \in \mathbb Z \hskip 2mm \text{ and }\\
& [ M : L_{\lambda} ] := \mathsf{gdim} \, \mathrm{hom} _A ( P_{\lambda}, M ) \in \mathbb Z (\!(t)\!).
\end{align*}

We have $[M : L_{\lambda}] = \sum _{i \in \mathbb Z} [ M : L_{\lambda} \left< i \right> ]_0 t^i \in \mathbb Z (\!(t)\!)$. This is a graded version of the composition multiplicity of $M$, that is existent as $A$ is finite-dimensional modulo its graded Jacobson radical (cf. Lemma \ref{hd} in the below).

We define $A \mathchar`-\mathsf{gmod}^{pf}$ to be the full subcategory of $A \mathchar`-\mathsf{gmod}$ consisting of objects which admit finite resolutions by finitely generated graded projective $A$-modules (this is an additive category). For $M \in A \mathchar`-\mathsf{gmod} ^{pf}$ and $N \in A \mathchar`-\mathsf{gmod}$, we define its graded Euler-Poincar\'e characteristic as:
\begin{equation}
\left< M, N \right> _{\mathsf{gEP}} := \sum _{i \ge 0} (-1)^i \mathsf{gdim} \, \mathrm{ext} ^i _A ( M, N ) \in \mathbb Z (\!( t )\!).\label{defnEP}
\end{equation}

\begin{theorem}\label{KS} Assume the conditions $(\spadesuit)$ and $(\clubsuit)$:
\begin{enumerate}
\item We have
$$[\widetilde{K} _{\lambda} : L _{\mu} ] = 0 = [K _{\lambda} : L _{\mu} ] \hskip 2mm \text{ for } \lambda \not\precsim \mu \hskip 2mm \text{ and } \hskip 3mm [K _{\lambda} : L _{\mu} ] = \delta _{\lambda,\mu} \hskip 2mm \text{ for } \lambda \sim \mu ;$$
\item We have
$$\mathrm{ext} _A ^{\bullet} ( \widetilde{K} _{\lambda}, \widetilde{K} _{\mu} ) = \{ 0 \} \hskip 1mm\text{ for each } \mu \not\precsim \lambda \hskip 1mm \text{ and } \hskip 1mm \mathrm{ext} _A ^{\bullet} ( \widetilde{K} _{\lambda}, K _{\mu} ) = \{ 0 \} \hskip 1mm \text{ for each } \mu \not\preceq \lambda;$$
\item For each $\lambda \in \Lambda$, we have
$$\widetilde{K} _{\lambda} \cong P _{\lambda} / \Bigl(  \sum _{\mu \prec \lambda} A e_{\mu} P _{\lambda} \Bigr);$$
\item Each $\widetilde{K} _{\lambda}$ admits a separable decreasing filtration whose associated graded is a direct sum of grading shifts of $K_{\mu}$ with $\mu \sim \lambda$. In addition, we have
$$[\widetilde{K} _{\lambda} : L _{\mu} ] = \mathsf{gdim} \, \mathrm{hom} _{C_{\lambda}} ( \xi, \zeta \otimes _{\mathbb C} H ^{\bullet} _{G _{\lambda} ^{\circ}} ( \{ \mathrm{pt} \} ) )$$
for every $\lambda = ( \underline{\lambda}, \xi ) \sim ( \underline{\mu}, \zeta ) = \mu$.
\end{enumerate}
\end{theorem}

\begin{remark}
Theorem \ref{KS} still hold if some of $L_{\lambda}$ is zero, but we need to impose the following condition:
\begin{itemize}
\item For each $\lambda \in \Lambda$ with $L_{\lambda} \neq \{ 0 \}$, every restriction of $\mathsf{IC} _{\lambda}$ to a $G$-orbit is a direct sum of local systems corresponding to $\mu \in \Lambda$ with $L _{\mu} \neq \{ 0 \}$.
\end{itemize}
This is the situation we encounter in \cite{K4} (cf. \cite{AcK}), and all the results in sections three and four still hold with straight-forward modifications.
\end{remark}

The proof of Theorem \ref{KS} is given at the end of section two.

\begin{corollary}[of Theorem \ref{KS}]\label{utK}
The matrices $K := ( [ K_{\lambda} : L_{\mu} ] )_{\lambda,\mu \in \Lambda}$ and $\widetilde{K} := ( [ \widetilde{K} _{\lambda} : L _{\mu}] )_{\lambda,\mu \in \Lambda}$ are blockwise upper-triangular and invertible in $\mathbb Q (\!(t)\!)$.
\end{corollary}
\begin{proof}
Thanks to Theorem \ref{KS} 1), $K$ satisfies the desired property. Here we also deduce that $\widetilde{K}$ is blockwise upper-triangular. The block-diagonal entries of $\widetilde{K}$ are invertible since they are identity modulo $t \mathbb Z [\![t]\!]$ by Theorem \ref{KS} 4). Therefore, we can invert the matrix $\widetilde{K}$ as desired.
\end{proof}

\begin{lemma}\label{gMe}
The graded algebra $A _{(G, \mathfrak X)}$ is graded Morita equivalent to $B_{(G, \mathfrak X)}$.
\end{lemma}
\begin{proof}
By construction, we see that
$$[P_{\lambda} : L_{\mu}] = \mathsf{gdim} \, \mathrm{Ext} ^{\bullet} _G ( \mathsf{IC} _{\lambda}, \mathsf{IC} _{\mu} ) \hskip 2mm \text{ for each } \hskip 2mm \lambda, \mu \in \Lambda.$$
In particular, the graded vector space $\mathrm{hom} _A ( P _{\mu}, P_{\lambda} )$ does not depend on the choice of $\mathcal L$. Since the (graded) algebra structure of $A$ arises from the Yoneda compositions, the composition map
$$\mathrm{hom} _A ( P _{\gamma}, P_{\mu} ) \times \mathrm{hom} _A ( P _{\mu}, P_{\lambda} ) \longrightarrow \mathrm{hom} _A ( P _{\gamma}, P_{\lambda} ) \hskip 2mm \text{ for each } \hskip 2mm \lambda, \mu, \gamma \in \Lambda$$
is identified with the composition map
$$\mathrm{Ext} ^{\bullet} _G ( \mathsf{IC} _{\lambda}, \mathsf{IC} _{\mu} ) \times \mathrm{Ext} ^{\bullet} _G ( \mathsf{IC} _{\mu}, \mathsf{IC} _{\gamma} ) \longrightarrow \mathrm{Ext} ^{\bullet} _G ( \mathsf{IC} _{\lambda}, \mathsf{IC} _{\gamma} ).$$
Therefore, we deduce $\mathrm{end} _A ( \bigoplus _{\lambda \in \Lambda} P_{\lambda} ) ^{op} \cong B_{(G, \mathfrak X)}$. Hence, $\mathrm{hom} _A ( \bigoplus _{\lambda \in \Lambda} P_{\lambda}, \bullet )$ yields the desired graded Morita equivalence through Lemma \ref{symm}.
\end{proof}

\begin{lemma}\label{hd}
A graded $A$-module $M$ with finite dimensional graded pieces is generated by its head if the grading of $M$ is bounded from the below.
\end{lemma}

\begin{proof}
Let $P$ be the projective cover of $\mathsf{hd} \, M$, that is the maximal graded semisimple quotient of $M$. Let us denote the lifting map $P \rightarrow M$ by $\phi$. We set $N := \mathsf{coker} \, \phi$. Since a simple quotient of $N$ is a simple quotient of $M$, we conclude that $\mathsf{hd} \, N = \{ 0 \}$. As a quotient of $M$, the grading of $N$ is bounded from the below. Since $\mathrm{Im} \, \phi$ is precisely the submodule of $M$ spanned by $\mathsf{hd} \, M$, we assume that $N \neq \{ 0 \}$ to deduce contradiction.

Let $c := \min \{ k \in \mathbb Z \mid A^k \neq \{ 0 \}\}$ and let $b := \min \{ k \in \mathbb Z \mid N^k \neq \{ 0 \}\}$. Consider the graded $A$-submodule $N_c$ of $N$ generated by $\{ N ^k \} _{k > b - c}$. Then, we have $N_c ^b = \{ 0 \}$ by a degree counting. In particular, we have $N / N_c \neq \{ 0 \}$. Since $N^{k} = N_c ^{k}$ for $k > b - c$ and each graded piece of $M$ is finite-dimensional, we deduce that $N / N_c$ is finite-dimensional.

A finite-dimensional graded $A$-module admits a simple graded quotient since the action of $A$ factors through a finite-dimensional algebra, which is Artin. As a simple quotient of $N / N_c$ is a simple quotient of $N$, we have a contradiction and hence $\mathrm{Im} \, \phi = M$ as required.
\end{proof}

\begin{corollary}\label{mphd}
Suppose that $A ^i = \{ 0 \}$ for $i < c$. Let $M$ be a graded $A$-module with finite dimensional graded pieces so that $M^i = \{ 0 \}$ for $i < b$. Then, $M$ admits a minimal graded projective $A$-resolution $($that are not necessarily of finite length$)$ so that its $k$-th term $P_k$ satisfies $P_k ^i = \{ 0 \}$ for $i < b + ( k + 1 ) c$ and has finite-dimensional graded pieces.
\end{corollary}

\begin{proof}
We have the projective cover $P_0$ of $M$ by lifting $\mathsf{hd} \, M$, whose grading ranges the sum of that of $A$ and $M$. Hence, we have $P_0 ^i = \{ 0 \}$ for $i < b + c$. In addition, we have
$$\dim P_0 ^i \le \sum _{j \in \mathbb Z} \dim M^j \cdot \dim A^{i-j} < \infty \hskip 3mm \text{ for each } i \in \mathbb Z.$$

Assume that we have a minimal projective resolution of $M$ with the desired property up to the $k$-th term:
$$P_k \stackrel{d_k}{\longrightarrow} P_{k-1}  \stackrel{d_{k-1}}{\longrightarrow} P_{k-2} \cdots P_1 \stackrel{d_1}{\longrightarrow} P_{0} \rightarrow M.$$
We set $P_{k+1}$ to be the projective cover of $\mathsf{hd} \, \ker d_{k}$ and $d_{k+1}$ to be its lift to $P_{k}$ to construct a projective resolution up to the $(k+1)$-th term. Since $P_k ^i = \{ 0 \}$ for $i < b + ( k + 1 ) c$, we deduce $( \ker d _k ) ^i = \{ 0 \}$ for $i < b + ( k + 1 ) c$ and $P_{k+1} ^i = \{ 0 \}$ for $i < b + ( k + 2 ) c$. In addition, Lemma \ref{hd} asserts that $P_{k+1}$ surjects onto $\ker d _k$, and hence $P_{k+1}$ gives the $(k+1)$-th term of a projective resolution of $M$. Since $\ker d_{k}$ and $A$ have finite-dimensional graded pieces and their gradings are bounded from the below, we deduce that each graded piece of $P_{k+1}$ is also finite-dimensional. As the minimality assumption yields that the induced map $\mathsf{hd} \, \ker d_{k} \to P_k \to \mathsf{hd} \, P _k$ is zero, our complex must be minimal. This proceeds the induction (which might not terminate), and hence the result.
\end{proof}

\section{Standard complexes and a proof of Theorem \ref{KS}}\label{wf}
In this section, we work over an algebraic closure of a finite field. We begin by general results on weight filtration, for which some part seems to follow from \cite{Bo,P,Sc2}.

\begin{definition}[Perverse class]
Let $\mathfrak X$ be a variety with an action of an affine algebraic group $G$. For $\mathcal E \in D ^b _G ( \mathfrak X )$, we define
$$[ \mathcal E ] := \sum _{m \in \mathbb Z} [ {}^p H ^m ( \mathcal E ) ] \in K ( \mathsf{Perv} _G \, \mathfrak X ),$$
where $\mathsf{Perv} _G \, \mathfrak X$ is the category of $G$-equivariant perverse sheaves on $\mathfrak X$ (cf. \cite{BL} 5.1), $K ( \mathsf{Perv} _G \, \mathfrak X )$ is its Grothendieck group, and $[ \mathcal F ]$ in the RHS denote the class of $\mathcal F \in \mathsf{Perv} _G \, \mathfrak X$.
\end{definition}

It is straight-forward to see $[ \mathcal E ] = [ \tau _{< m} \mathcal E ] + [\tau _{\ge m} \mathcal E ]$ for each $m \in \mathbb Z$.

\begin{lemma}\label{bei}
Let $\mathfrak X$ be a variety with an action of an affine algebraic group $G$. For each $\mathcal E \in D ^b _G ( \mathfrak X )$, we have $\# \{ m \in \mathbb Z \mid {}^p H ^m ( \mathcal E ) \neq \{0\} \} < \infty$. In particular, the perverse classes are well-defined elements of $K ( \mathsf{Perv} _G \, \mathfrak X )$.
\end{lemma}

\begin{proof}
Since $\mathcal E$ is constructible, its stalk and costalk are bounded complexes at each point of $\mathcal X$. Thus, the characterization of objects of $\tau _{\ge 0} D ^b ( \mathfrak X )$ and $\tau _{\le 0} D ^b ( \mathfrak X )$ in terms of stalks and costalks \cite{BBD} 2.2.2 implies the result (as there are finitely many stratum where the stalks/costalks of $\mathcal E$ are constant).
\end{proof}

\begin{lemma}\label{weight-trunc}
Let $\mathfrak X$ be a variety with an action of an affine algebraic group $G$. Let $\mathcal E \in D ^b _G ( \mathfrak X )$ be a complex with a mixed structure. For each $a \in \mathbb Z$, there exist canonically constructed mixed complexes $F_{< a} \mathcal E, F_{\ge a} \mathcal E \in D ^b _G ( \mathfrak X )$ with the following properties:
\begin{enumerate}
\item $F _{< a} \mathcal E$ has weight $< a$, and $F _{\ge a} \mathcal E$ has weight $\ge a$;
\item We have a distinguished triangle
\begin{equation}
F _{< a} \mathcal E \to \mathcal E \to F _{\ge a}\mathcal E  \stackrel{+1}{\longrightarrow} \hskip 3mm \text{ in } D^b _G ( \mathfrak X );\label{wtr}
\end{equation}
\item We have $[ \mathcal E ] = [ F_{< a} \mathcal E ] + [ F_{\ge a} \mathcal E ] \in K ( \mathsf{Perv} _G \, \mathfrak X )$.
\end{enumerate}
\end{lemma}

\begin{proof}
For each $m \in \mathbb Z$, the weight $< (a+m)$-part ${}^p H ^m ( \mathcal E ) _{< a}$ of ${}^p H ^m ( \mathcal E )$ is a subobject by \cite{BBD} 5.3.5. By induction on $m$, we construct distinguished triangles
\begin{equation}
\mathcal E _m ^{<a} \to \tau _{> m} \mathcal E \to \mathcal E _m ^{\ge a}  \stackrel{+1}{\longrightarrow} \hskip 3mm \text{ in } D^b _G ( \mathfrak X ),\label{distDGm}
\end{equation}
with the following properties:
\begin{itemize}
\item[$(\ref{distDGm})_1$] $\mathcal E _m ^{< a}$ has weight $< a$, and $\mathcal E _m ^{\ge a}$ has weight $\ge a$;
\item[$(\ref{distDGm})_2$] We have $[ \tau _{> m} \mathcal E ] = [ \mathcal E _m ^{< a} ] + [ \mathcal E _m ^{\ge a} ]$;
\item[$(\ref{distDGm})_3$] We have $\tau _{\le m} \mathcal E _m ^{< a} \cong \{ 0 \} \cong \tau _{\le m} \mathcal E _m ^{\ge a}$.
\end{itemize}
We assume that a distinguished triangle of the from (\ref{distDGm}) for $m$ exists and construct a distinguished triangle of the form (\ref{distDGm}) for $m-1$. We have $\tau _{>m} \mathcal E = \{ 0 \}$ for $m \gg 0$, and hence the case $m \gg 0$ is verified by setting $\mathcal E _m ^{< a} = \{ 0 \} = \mathcal E _m ^{\ge a}$. By the existence of the weight filtration of mixed perverse sheaves (cf. \cite{BBD} 5.3.5), we deduce that the (shifted) complex $H _{\ge a} ^m := ( {}^p H ^m ( \mathcal E ) / {}^p H ^m ( \mathcal E ) _{< a} )[-m]$ has weight $\ge a$. Here, $\tau _{\ge m} \mathcal E$ is determined by the map $\tau _{> m} \mathcal E \to {}^p H ^m ( \mathcal E )[1-m]$. Its pullback to $\mathcal E_m^{< a}$ factors through ${}^p H ^m ( \mathcal E ) _{< a}[1-m]$ since
$$\mathrm{Hom} _{G} ( \mathcal E_m^{< a}, H _{\ge a} ^m [1] )_{\mathrm{pure}} = \{ 0 \}$$
by \cite{BBD} 5.1.15. Then, \cite{BBD} 1.1.11 yields the following commutative diagram of six distinguished triangles:
\begin{equation}
\xymatrix@R=10pt@C=10pt{
\mathcal E_m^{< a}[-1] \ar[d]^{+1}\ar[r]& ( \tau_{>m} \mathcal E )[-1] \ar[d]^{+1}\ar[r]& \mathcal E_m^{\ge a}[-1] \ar[d]^{+1}\ar[r]^{\hskip 7mm +1}&\\
{}^p H ^m ( \mathcal E ) _{< a} [-m]  \ar[d]\ar[r]& {}^p H ^m ( \mathcal E ) [-m] \ar[d]\ar[r]& H _{\ge a} ^m \ar[d]\ar[r]^{+1}&\\
\mathcal E_{m-1}^{< a}  \ar[d]\ar[r]& \tau_{\ge m} \mathcal E \ar[d]\ar[r] & \mathcal E_{m-1}^{\ge a} \ar[d]\ar[r]^{\hskip 7mm +1}&\\
& & &
}.\label{9-term}
\end{equation}
Here $\mathcal E_{m-1}^{< a}$ and $\mathcal E_{m-1}^{\ge a}$ are determined by the column maps. By construction, $\mathcal E_{m-1}^{< a}$ has weight $< a$ and $\mathcal E_{m-1}^{\ge a}$ has weight $\ge a$.
\begin{claim}\label{ident}
We have the following identities:
\begin{align*}
[ \mathcal E_{m-1}^{< a} ] & = [ \mathcal E_{m}^{< a} ] + [ {}^p H ^m ( \mathcal E ) _{< a} ], & \tau _{\le m-1} \mathcal E ^{< a} _{m-1} & = \{ 0 \}\\
 [ \mathcal E_{m-1}^{\ge a} ] & = [ \mathcal E_{m}^{\ge a} ] + [ H _{\ge a} ^m] & \tau _{\le m-1} \mathcal E ^{\ge a} _{m-1} & = \{ 0 \}\\
[ \tau _{\ge m} \mathcal E ] & = [ \tau _{> m} \mathcal E ] + [ {}^p H ^m ( \mathcal E )], & [ {}^p H ^m ( \mathcal E )] & = [ {}^p H ^m ( \mathcal E ) _{< a} ] + [ H _{\ge a} ^m ].
\end{align*}
\end{claim}
\begin{proof}
We have ${}^p H^{i} ( {}^p H ^m ( \mathcal E ) _{< a} [-m] ) = {}^p H ^m ( \mathcal E ) _{< a}$ ($i=m$) or $\{ 0 \}$ ($i \neq m$) by construction. Thus, the long exact sequence of perverse cohomologies of the first column of (\ref{9-term}) reads as:
\begin{align*}
0 \rightarrow & \, {}^p H ^m ( \mathcal E ) _{< a} \stackrel{\cong}{\longrightarrow} {}^p H^{m} ( \mathcal E_{m-1}^{< a} ) \rightarrow 0 \\
&  \rightarrow 0 \rightarrow {}^p H^{m+1} ( \mathcal E_{m-1}^{< a} ) \rightarrow {}^p H^{m+1} ( \mathcal E_{m}^{< a} ) \rightarrow 0 \rightarrow \cdots
\end{align*}
This implies the first two identities. By a similar analysis applied to the third column of (\ref{9-term}), we conclude the third and fourth identities. The fifth identity follows from the definition of the perverse class. The sixth identity follows from the definition of $H _{\ge a} ^m$.
\end{proof}

We return to the proof of Lemma \ref{bei}. By Claim \ref{ident} and $(\ref{distDGm})_2$ for $m$, we have
\begin{align*}
[ \mathcal E_{m-1}^{< a} ] + [ \mathcal E_{m-1}^{\ge a} ] &= [ \mathcal E_{m}^{< a} ] + [ \mathcal E_{m}^{\ge a} ] + [ {}^p H ^m ( \mathcal E ) _{< a} ] + [ H _{\ge a} ^m]\\
&= [ \mathcal E_{m}^{< a} ] + [ \mathcal E_{m}^{\ge a} ] + [ {}^p H ^m ( \mathcal E )]\\
&= [ \tau _{> m} \mathcal E] + [ {}^p H ^m ( \mathcal E )] = [ \tau _{\ge m} \mathcal E].
\end{align*}

This verifies that the bottom distinguished triangle of (\ref{9-term}) satisfies $(\ref{distDGm})_2$ for $m-1$. In addition, it also satisfies $(\ref{distDGm})_3$ for $m-1$ by Claim \ref{ident}. Therefore, we obtain a distinguished triangle of type (\ref{distDGm}) for $m-1$. By Lemma \ref{bei}, we know ${}^p H^m ( \mathcal E ) = \{ 0 \}$ for $m \ll 0$. Therefore, we take $m \ll 0$ so that $\mathcal E \cong \tau _{> m} \mathcal E$ and set $F _{< a} \mathcal E := \mathcal E _m ^{< a}$ and $F _{\ge a} \mathcal E := \mathcal E _m ^{\ge a}$ to conclude the result.
\end{proof}

\begin{corollary}\label{transTR}
Keep the setting of Lemma \ref{weight-trunc}. For each $a \le b$, we have
$$F_{< a}( F _{<b} \mathcal E ) \cong F _{< a} \mathcal E, \text{ and } F_{\ge b} ( F _{\ge a} \mathcal E ) \cong F _{\ge b} \mathcal E.$$
\end{corollary}

\begin{proof}
We use the setting of the proof of Lemma \ref{weight-trunc}. For each $m \in \mathbb Z$, we have natural inclusions
$$F _{< (a+m)} H ^m ( \mathcal E ) = F _{< (a+m)} F_{<(b+m)} H ^m ( \mathcal E ) \subset F_{<(b+m)} H ^m ( \mathcal E ) \subset H ^m ( \mathcal E )$$
in the category of ($G$-equivariant mixed) perverse sheaves. Then, the pullback of $\tau _{> m} \mathcal E \rightarrow {}^p H ^m ( \mathcal E ) [1-m]$ to $\mathcal E ^{< a}_m$ also factors through $( F_{<(b+m)} H ^m ( \mathcal E ) ) [1-m]$, which induces a map $\mathcal E ^{< a} _{m-1} \rightarrow \mathcal E ^{<b}_{m-1}$ in case the natural map $\mathcal E ^{<a} _m \rightarrow \tau _{>m} \mathcal E$ factors through $\mathcal E ^{<b}_{m}$. It further induces an isomorphism $\mathcal E ^{< a} _{m-1} \cong F_{<a} \mathcal E ^{<b}_{m-1}$ provided if $\mathcal E ^{< a} _{m} \cong F_{<a} \mathcal E ^{<b}_{m}$. Hence, we deduce $\mathcal E ^{< a} _m \cong F_{<a} \mathcal E ^{<b}_m$ by a downward induction on $m$ (since $\tau _{> m} \mathcal E = \{ 0 \}$ for $m \gg 0$). This implies $F_{< a}( F _{<b} \mathcal E ) \cong F _{< a} \mathcal E$. The assertion $F_{\ge b} ( F _{\ge a} \mathcal E ) \cong F _{\ge b} \mathcal E$ follows by applying the construction of $F_{< b}$ to the distinguished triangle $F _{< a} \mathcal E \rightarrow \mathcal E \rightarrow F _{\ge a} \mathcal E \stackrel{+1}{\rightarrow}$.
\end{proof}

\begin{lemma}\label{std-filt}
Let $\mathfrak X$ be a variety with an action of an affine algebraic group $G$. Let $\mathcal E \in D ^b _G ( \mathfrak X )$ be a complex with a mixed structure. Then, there exists a collection of objects $\{ F _{\ge a} \mathcal E \} _{a \in \mathbb Z}$ of $D_G ^b ( \mathfrak X )$ with the following properties:
\begin{enumerate}
\item Each $F _{\ge a} \mathcal E$ is of weight $\ge a$;
\item $F _{\ge a} \mathcal E = \{ 0 \}$ for $a \gg 0$, and $F _{\ge a} \mathcal E \cong \mathcal E$ for $a \ll 0$;
\item For each $a \in \mathbb Z$, we have a distinguished triangle
\begin{equation}
\longrightarrow \mathrm{gr} _a \, \mathcal E \to F _{\ge a} \mathcal E \to F _{\ge (a+1)} \mathcal E \stackrel{+1}{\longrightarrow} \hskip 3mm \text{ in } D^b _G ( \mathfrak X ),\label{distP}
\end{equation}
where $\mathrm{gr} _a \, \mathcal E$ is a pure complex of weight $a$;
\item We have $[F _{\ge a} \mathcal E] = \sum _{b \ge a} [\mathrm{gr} _b \, \mathcal E]$ and $[\mathcal E] = \sum _a [\mathrm{gr} _a \, \mathcal E]$;
\item Each $\mathrm{gr} _a \, \mathcal E$ is a direct sum of shifted simple $G$-equivariant perverse sheaves.
\end{enumerate}
\end{lemma}

\begin{proof}
For $a \ll 0$, we have $F _{\ge a} \mathcal E \cong \mathcal E$. For each $a \in \mathbb Z$,
$$F _{< (a+1)} ( F _{\ge a} \mathcal E ) \to F _{\ge a} \mathcal E \to F _{\ge (a+1)} \mathcal E \stackrel{+1}{\longrightarrow}$$
is a distinguished triangle by Corollary \ref{transTR}. By construction, each $F _{< (a+1)} ( F _{\ge a} \mathcal E )$ must be pure of weight $a$ and is non-zero only for a finitely many $a \in \mathbb Z$. By setting $\mathrm{gr} _a \, \mathcal E := F _{< (a+1)} ( F _{\ge a} \mathcal E )$, we deduce the first four assertions.

The fifth assertion is automatic for a(n arbitrary) pure complex in $D^b _G ( \mathfrak X )$ by \cite{BBD} 5.3.9 i) and 5.4.5.
\end{proof}

\begin{corollary}\label{r-std-filt}
We employ the notation and setting of section one. In particular, the $G$-action on $\mathfrak X$ has finitely many orbits. For each $\lambda = (\underline{\lambda}, \xi) \in \Lambda$, we arrange $\mathsf{IC} _{\lambda}$ and $\underline{\xi} \, [\dim \mathbb O_{\lambda}]$ to be pure of weight zero $($cf. $(\ref{ICdef}))$. There exists a collection of objects $\{ F _{\ge a} \mathbb C _{\lambda} \} _{a \in \mathbb Z}$ of $D_G ^b ( \mathfrak X )$ with the following properties:
\begin{enumerate}
\item We have $F _{\ge 0} \mathbb C _{\lambda} = \mathsf{IC} _{\lambda}$ and $F _{\ge a} \mathbb C _{\lambda} = \mathbb C _{\lambda}$ for $a \ll 0$;
\item For each $a < 0$, we have a distinguished triangle
$$\longrightarrow \mathrm{gr} _a \, \mathbb C _{\lambda} \to F _{\ge a} \mathbb C _{\lambda} \to F _{\ge a+1} \mathbb C _{\lambda} \stackrel{+1}{\longrightarrow} \hskip 3mm \text{ in } D^b _G ( \mathfrak X ),$$
where $\mathrm{gr} _a \, \mathbb C _{\lambda}$ is a direct sum of $\{ \mathsf{IC} _{\mu}[i] \} _{\mu \prec \lambda, i \in \mathbb Z}$;
\item For each $a \in \mathbb Z$, $F _{\ge a} \mathbb C _{\lambda}$ has weight $\ge a$ and $\mathrm{gr} _a \, \mathbb C _{\lambda}$ is pure of weight $a$.
\end{enumerate}
\end{corollary}
\begin{proof}
We apply the construction of Lemma \ref{std-filt} for $\mathcal E := \mathbb C _{\lambda}$. The latter part of the first assertion is automatic and the third assertion is Lemma \ref{std-filt} 5).

The operation $( j_{\lambda} ) _!$ does not increase weights (cf. \cite{BBD} 5.1.14). In particular, we have $F _{> 0} \mathbb C _{\lambda} = \{ 0 \}$. Hence, $F _{\ge 0} \mathbb C _{\lambda}$ is pure of weight zero, and we have a canonical map $\psi : \mathbb C _{\lambda} \rightarrow F _{\ge 0} \mathbb C _{\lambda}$ (from Lemma \ref{weight-trunc}). By \cite{BBD} 5.4.5, we have $F _{\ge 0} \mathbb C _{\lambda} \cong \bigoplus _m {}^p H ^m$, where we set ${}^p H ^m := {}^p H ^m ( F _{\ge 0} \mathbb C _{\lambda} )[-m]$. Each ${}^p H ^m[m]$ further decomposes into a direct sum of simple perverse sheaves by \cite{BBD} 5.3.9 i). By Lemma \ref{std-filt} 4) and the construction of the minimal extension $\mathsf{IC} _{\lambda}$ of $\underline{\xi} \, [\dim \mathbb O_{\lambda}]$, we deduce that ${}^p H ^m$ contains a direct summand supported outside of $\overline{\mathbb O _{\lambda}} \setminus \mathbb O _{\lambda}$ only if $m = 0$. Moreover, such a direct summand is unique and is isomorphic to $\mathsf{IC} _{\lambda}$. It follows that
\begin{equation}
\mathrm{Hom} _G ( \mathbb C _{\lambda}, {}^p H ^m ) \cong \mathrm{Hom} _G ( \underline{\xi} \, [\dim \mathbb O_{\lambda}], j_{\lambda} ^! {}^p H ^m ) = \begin{cases} \mathbb C & (m=0)\\ \{0\} &(\text{otherwise})\end{cases}.\label{vanCH}
\end{equation}
\begin{claim}\label{nz}
The map $\psi$ is non-zero after projected to a direct factor of $F _{\ge 0} \mathbb C _{\lambda}$.
\end{claim}
\begin{proof}
Consider the long exact sequence of perverse cohomologies associated to the distinguished triangles $( F_{<0} \mathbb C _{\lambda}, \mathbb C _{\lambda}, F_{\ge 0}  \mathbb C _{\lambda} ):$
$$\cdots \rightarrow {} ^p H ^i ( \mathbb C _{\lambda} ) \stackrel{{}^p H ^i ( \psi )}{\longrightarrow }{} ^p H ^i ( F _{\ge 0} \mathbb C _{\lambda} ) \rightarrow {} ^p H ^{i+1} ( F _{< 0} \mathbb C _{\lambda} ) \rightarrow {} ^p H ^{i+1} ( \mathbb C _{\lambda} ) \rightarrow \cdots.$$
This is an exact sequence in the category of $G$-equivariant perverse sheaves. We have $[ \mathbb C _{\lambda} ] = [ F _{\ge 0} \mathbb C _{\lambda} ] + [ F _{< 0} \mathbb C _{\lambda} ]$ by Lemma \ref{weight-trunc} 3). These imply that the connecting homomorphism $\delta _i :{} ^p H ^i ( F _{\ge 0} \mathbb C _{\lambda} ) \rightarrow {} ^p H ^{i+1} ( F _{< 0} \mathbb C _{\lambda} )$ must be zero for every $i \in \mathbb Z$.

If a projection $\phi _{\mathcal E}$ of ${}^p H ^i ( F _{\ge 0} \mathbb C _{\lambda} ) \subset F _{\ge 0} \mathbb C _{\lambda}$ ($i \in \mathbb Z$) to its (pure) direct factor $\mathcal E$ satisfies $\phi _{\mathcal E} \circ \psi = 0$, then we have necessarily $\{ 0 \} \neq \delta _i ( \mathcal E ) \subset {}^p H ^{i+1} ( F _{< 0} \mathbb C _{\lambda} )$. This is a contradiction and we conclude that $\phi _{\mathcal E} \circ \psi \neq 0$ for every direct factor $\mathcal E$ as required.
\end{proof}

We return to the proof Corollary \ref{r-std-filt}. Thanks to (\ref{vanCH}) and Claim \ref{nz}, $F_{\ge 0} \mathbb C _{\lambda}$ has a unique direct factor that is isomorphic to $\mathsf{IC} _{\lambda}$. Hence, we conclude $F _{\ge 0} \mathbb C _{\lambda} = \mathsf{IC} _{\lambda}$, that is the first part of the first assertion (cf. \cite{K4} Claims A, B). This, together with Lemma \ref{std-filt} 4), 5), implies that $\mathrm{gr} _a \, \mathbb C _{\lambda}$ ($a < 0$) is a direct sum of shifted perverse sheaves supported on $\overline{\mathbb O _{\lambda}} \setminus \mathbb O _{\lambda}$. This proves the second assertion, which finishes the proof.
\end{proof}

In the below (in this section), we assume the setting of section one. In particular, we set $A := A_{(G,\mathfrak X)}$.

\begin{theorem}[Standard complexes]\label{transfer}
In the setting of Lemma \ref{std-filt}, we also assume $(\spadesuit)$ in Condition $\ref{spadesuit}$ and $(\clubsuit)_1$ in Condition $\ref{clubsuit}$. We normalize $\mathcal L$ to be pure of weight zero. Then, we have a complex of graded projective $A$-modules
$$(Q ( \mathcal E ), d) : \cdots \rightarrow Q ( \mathcal E )_{a-1} \stackrel{d_{a-1}}{\rightarrow} Q ( \mathcal E ) _a \stackrel{d_a}{\longrightarrow} Q ( \mathcal E ) _{a+1} \stackrel{d_{a+1}}{\longrightarrow} \cdots$$
with the following properties:
\begin{enumerate}
\item We have $Q ( \mathcal E ) _a \cong \mathrm{Ext} ^\bullet _G ( \mathrm{gr} _a \, \mathcal E, \mathcal L )$ as a graded $A$-module with pure weight $- a$;
\item Each differential has degree one and respects the weight structure;
\item We have an isomorphism of graded $A$-modules:
\begin{equation}
\mathrm{Ext} ^\bullet _G ( \mathcal E, \mathcal L ) \cong H ^{\bullet} ( Q ( \mathcal E ), d ).\label{isomH}
\end{equation}
\end{enumerate}
\end{theorem}

\begin{proof}
For each $a \in \mathbb Z$, the sheaf $\mathrm{gr} _a \, \mathcal E$ is a direct sum of shifted simple ($G$-equivariant) perverse sheaves by Lemma \ref{std-filt} 5). In particular, $\mathrm{Ext}_G^{\bullet} ( \mathrm{gr} _a \, \mathcal E, \mathcal L )$ is a projective $A$-module. Since $A$ is pure of weight zero (see $(\clubsuit)_1$), each $P_{\lambda}$ is pure of some weight (depending on the weight of $\mathsf{IC} _{\lambda}$). Therefore, $\mathrm{Ext}_G^{\bullet} ( \mathrm{gr} _a \, \mathcal E, \mathcal L )$ is pure of weight $-a$ since $\mathrm{gr} _a \, \mathcal E$ is pure of weight $a$ and $\mathcal L$ is pure of weight zero.

We define a graded $A$-module complex
$$( Q _{\ge a} ( \mathcal E ), d ) : 0 \rightarrow Q ( \mathcal E ) _a \stackrel{d_a}{\longrightarrow} Q ( \mathcal E ) _{a+1} \stackrel{d_{a+1}}{\longrightarrow} Q ( \mathcal E ) _{a+2} \rightarrow \cdots$$
with degree one differentials and $Q ( \mathcal E ) _b := \mathrm{Ext} ^\bullet _G ( \mathrm{gr} _b \, \mathcal E, \mathcal L )$ ($b \ge a$) by induction on $a \in \mathbb Z$.

There exists $a_0 \in \mathbb Z$ so that $\mathrm{gr} _a \, \mathcal E \cong \{ 0 \}$ for $a \ge a_0$ by Lemma \ref{bei}, and so we define $( Q _{\ge a} ( \mathcal E ), d ) := ( \{ 0 \}, d)$ for $a \ge a_0$. We assume $( Q _{\ge (a+1)} ( \mathcal E ), d )$ is already defined and the following condition
\begin{itemize}
\item[$(\star)_b$] We have an isomorphism $H ^{\bullet} ( Q _{\ge b} ( \mathcal E ), d ) \cong \mathrm{Ext} _G ^{\bullet} ( F_{\ge b} \mathcal E, \mathcal L )$  as graded $A$-modules;
\end{itemize}
for each $b > a$. Note that $(\star)_a$ for $a \ll 0$ is equivalent to the assertion since $F_{\ge a} \mathcal E \cong \mathcal E$ and $( Q _{\ge a} ( \mathcal E ), d ) = ( Q ( \mathcal E ), d )$ for $a \ll 0$.

We only need to construct a differential $d_a : Q ( \mathcal E ) _a \rightarrow Q ( \mathcal E )_{a+1}$ with degree one in the complex $( Q _{\ge a} ( \mathcal E ), d )$ which satisfies $(\star)_a$ to proceed the induction.

The functor $\mathrm{Ext}_G^{\bullet} ( \bullet, \mathcal L)$ sends the distinguished triangle (\ref{distP}) to the distinguished triangle
\begin{equation}
\to \mathrm{Ext} _G ^{\bullet} ( F_{>a} \mathcal E, \mathcal L ) \to \mathrm{Ext} _G ^{\bullet} ( F _{\ge a} \mathcal E, \mathcal L ) \to \mathrm{Ext} _G ^{\bullet} ( \mathrm{gr} _a \, \mathcal E, \mathcal L ) \stackrel{+1}{\longrightarrow}. \label{distwt2}
\end{equation}
This is a distinguished triangle of (complexes of) $\mathbb C$-vector spaces. All the maps in (\ref{distwt2}) can be regarded as compositions on the first factor of $\mathrm{Ext} ^{\bullet} _G$. The $A$-module action on each term on (\ref{distwt2}) is induced from the composition on the second factor of $\mathrm{Ext} ^{\bullet} _G$. Therefore, all the maps in (\ref{distwt2}) must commute with the $A$-actions by the associativity of compositions (i.e. all the maps are graded $A$-module homomorphisms). By the induction hypothesis, we deduce that $H ^{\bullet} ( Q _{>a} ( \mathcal E), d ) \cong \mathrm{Ext} _G ^{\bullet} ( F_{>a} \mathcal E, \mathcal L )$ has weight $< -a$.

We denote the connecting map (the map with degree one) of (\ref{distwt2}) by $\delta$. The graded $A$-module $\ker \, \delta$ has pure weight $-a$, and $\mathrm{coker} \, \delta [-1]$ has weight $< -a$. Since $A$ is pure of weight zero (see $(\clubsuit)_1$), we conclude
\begin{align*}
 0 \to \ker \, \delta \rightarrow  &  Q ( \mathcal E ) _a \stackrel{\delta}{\longrightarrow} \mathrm{Ext} _G ^{\bullet} ( F_{>a} \mathcal E, \mathcal L ) \left< - 1 \right> \rightarrow \mathrm{coker} \, \delta \rightarrow 0 \hskip 3mm \text{(exact),} \hskip 2mm \text{ and}\\
& \mathrm{Ext} _G ^{\bullet} ( F_{\ge a} \mathcal E, \mathcal L ) \cong \ker \, \delta \oplus \mathrm{coker} \, \delta\left< 1 \right> \hskip 3mm \text{ as graded $A$-modules.}
\end{align*}
By $(\star)_{a+1}$, the graded $A$-module $\ker \, d_{a+1}$ is a direct factor of $H ^{\bullet} ( Q _{> a} ( \mathcal E ), d )$, and it is the weight $- (a+1)$-part of $H ^{\bullet} ( Q _{> a} ( \mathcal E ), d )$. Since $Q ( \mathcal E ) _a$ have weight $-a$ and $\delta$ lowers the weight by one, we conclude
$$\delta ( Q ( \mathcal E ) _a ) \subset \ker \, d_{a+1} \subset H ^{\bullet} ( Q _{> a} ( \mathcal E ), d ) \cong \mathrm{Ext} _G ^{\bullet} ( F_{>a} \mathcal E, \mathcal L ).$$
It follows that the map $\delta$ lifts to a graded map $\delta' : Q ( \mathcal E ) _a \rightarrow Q ( \mathcal E ) _{a+1}\left< -1 \right>$. If we set $d \mid _{Q ( \mathcal E ) _a} := \delta'$, then we have
\begin{align*}
&H ^{b} ( Q _{\ge a} ( \mathcal E ), d ) \cong H ^{b} ( Q _{\ge a+1} ( \mathcal E ), d ) \hskip 3mm \text{ for every } \hskip 3mm b \ge a+2;\\
&H ^{a+1} ( Q _{\ge a} ( \mathcal E ), d ) \cong \ker d_{a+1} / \delta ( Q ( \mathcal E ) _a ) \subset \mathrm{coker} \, \delta \left< 1 \right>;\\
&\bigoplus _{b > a}H ^{b} ( Q _{\ge a} ( \mathcal E ), d ) \cong \mathrm{coker} \, \delta \left< 1 \right>, \hskip 2mm \text{ and } \hskip 2mm H ^{a} ( Q _{\ge a} ( \mathcal E ), d ) \cong \ker \, \delta.
\end{align*}

Therefore, the complex $( Q _{\ge a} ( \mathcal E ), d )$ satisfies our requirement, and hence the induction proceeds as required.
\end{proof}

\begin{proposition}\label{P_to_K_finite}
Assume $(\spadesuit)$ and $(\clubsuit)_1$. Fix $\lambda = ( \underline{\lambda}, \xi ) \in \Lambda$. Let us normalize $\underline{\xi} \, [\dim \mathbb O_{\lambda}]$ to be weight zero. If each $\mathsf{IC} _{\gamma}$ $($whose weight is normalized to be zero$)$ is pointwise pure of weight zero along $\mathbb O_{\lambda}$, then $(Q ( \mathbb C _{\lambda} ), d)$ gives a $($finite length$)$ graded projective $A$-resolution of $\widetilde{K} _{\lambda}$.
\end{proposition}

\begin{proof}
By Theorem \ref{transfer}, we have $\widetilde{K} _{\lambda} \cong \bigoplus _{a \in \mathbb Z} H ^{a} ( Q ( \mathbb C _{\lambda} ), d )$ as graded $A$-modules. In addition, the direct factor $H ^{a} ( Q ( \mathbb C _{\lambda} ), d ) \subset \widetilde{K} _{\lambda}$ has weight $-a$.

By assumption, the graded $A$-module $K _{\mu}$ is pure of weight zero for every $\mu \sim \lambda$. Hence, so is $\widetilde{K} _{\lambda}$ by (\ref{leray}) since $G_{\lambda}$ is affine (note that $\{ 0 \} \neq L _{\lambda} = \mathrm{Hom} _G ( \mathbb C _{\lambda}, L _{\lambda} \boxtimes \mathsf{IC} _{\lambda} ) \subset \widetilde{K} _{\lambda} ^0$ belongs to the weight zero part). As a consequence, $H ^{\bullet} ( Q ( \mathbb C _{\lambda} ), d )$ is pure of weight zero, meaning that $H ^{i} ( Q ( \mathbb C _{\lambda} ), d ) = \{ 0 \}$ for every $i \neq 0$. By Corollary \ref{r-std-filt} 2), we have $Q ( \mathbb C _{\lambda} )_{a} = \{ 0 \}$ for $a > 0$. Therefore, the standard complex of graded $A$-modules
$$\cdots \longrightarrow Q ( \mathbb C _{\lambda} )_{-3} \stackrel{d_{-3}}{\longrightarrow} Q ( \mathbb C _{\lambda} ) _{-2} \stackrel{d_{-2}}{\longrightarrow} Q ( \mathbb C _{\lambda} ) _{-1} \stackrel{d_{-1}}{\longrightarrow} Q ( \mathbb C _{\lambda} ) _0 \rightarrow \widetilde{K} _{\lambda} \rightarrow 0$$
gives a projective resolution of $\widetilde{K} _{\lambda}$. It is of finite length by Lemma \ref{std-filt} 2).
\end{proof}

\begin{corollary}\label{kf}
Assume $(\spadesuit)$ and $(\clubsuit)$. For each $\lambda \in \Lambda$, we have:
\begin{enumerate}
\item $\widetilde{K} _{\lambda} \in A \mathchar`-\mathsf{gmod}^{pf}$;
\item $\widetilde{K} _{\lambda}$ is a quotient of $P_{\lambda}$;
\item The complex $(Q ( \mathbb C _{\lambda} ), d)$ is a minimal projective resolution of $\widetilde{K} _{\lambda}$;
\item We have $Q ( \mathbb C _{\lambda} ) _0 \cong P_{\lambda}$, and $\bigoplus _{a < 0} Q ( \mathbb C _{\lambda} ) _a$ is a direct sum of one copy of $P_{\lambda}$ and finitely many copies of $\{ P_{\mu} \left< k \right> \} _{\mu \prec \lambda, k \in \mathbb Z}$.
\end{enumerate}
\end{corollary}

\begin{proof}
The first assertion follows from Proposition \ref{P_to_K_finite} and $(\clubsuit)_2$ in Condition \ref{clubsuit}. By Corollary \ref{r-std-filt} 1), we deduce that $\widetilde{K} _{\lambda}$ is a quotient of $P_{\lambda}$, that is the second assertion.

The third and the fourth assertions except for the minimality of the resolution follows by (additionally) using Theorem \ref{transfer} and Corollary \ref{r-std-filt} 2).

As we have a finite length finitely generated projective resolution, there exists a minimal projective resolution of $\widetilde{K} _{\lambda}$. The complex $(Q ( \mathbb C _{\lambda} ), d)$ is not a minimal resolution of $\widetilde{K} _{\lambda}$ if and only if there is a graded indecomposable projective $A$-direct factor $P \subset Q ( \mathbb C _{\lambda} ) _a$ ($a \le 0$) so that the quotient map $\psi : P \rightarrow \mathsf{hd} \, P$ represents zero in $\mathrm{ext} ^{-a} _A ( \widetilde{K} _{\lambda}, \mathsf{hd} \, P )$. For this, we need the following condition $(\dagger)_a$ for some $a \le 0$:
\begin{itemize}
\item[$(\dagger)_a$] $\mathrm{Im} \, d _{a} \subset Q ( \mathbb C _{\lambda} ) _{a+1}$ gives a non-zero module in $\mathsf{hd} \, Q ( \mathbb C _{\lambda} ) _{a+1}$. In other words, there exist a pair of indecomposable direct summands $P \subset Q ( \mathbb C _{\lambda} ) _{a+1}$ and $Q \subset Q ( \mathbb C _{\lambda} ) _{a}$ so that $P = d _a ( Q )$.
\end{itemize}

We prove that such direct summands are inexistent to deduce the minimality.

The differential $d_a$ of $(Q ( \mathbb C _{\lambda} ), d)$ is induced by a morphism $F _{> a} \mathbb C _{\lambda} \rightarrow \mathrm{gr} _a \, \mathbb C _{\lambda} [1]$ arising from the distinguished triangle $(\mathsf{gr} _{a} \, \mathbb C _{\lambda}, F _{\ge a} \mathbb C _{\lambda}, F _{> a} \mathbb C _{\lambda})$. In fact, we only need its pullback $\psi$ to $\mathrm{gr} _{a+1} \, \mathbb C _{\lambda}$ though the distinguished triangle $(\mathsf{gr} _{a+1} \, \mathbb C _{\lambda}, F _{> a} \mathbb C _{\lambda}, F _{> a+1} \mathbb C _{\lambda})$. Here the both of $\mathsf{gr} _{a} \, \mathbb C _{\lambda}$ and $\mathsf{gr} _{a+1} \, \mathbb C _{\lambda}$ are direct sums of shifted $G$-equivariant perverse sheaves by Lemma \ref{std-filt} 5). Taking account into this, $d_a$ ($= \mathrm{Ext} ^{\bullet} _G (\psi, \mathcal L)$) induces an isomorphism between indecomposable projective $A$-direct factors of $Q ( \mathbb C _{\lambda} ) _{a}$ and $Q ( \mathbb C _{\lambda} ) _{a+1}$ only if $\psi$ induces an isomorphism between direct factors of $\mathsf{gr} _{a} \, \mathbb C _{\lambda}$ and $\mathsf{gr} _{a+1} \, \mathbb C _{\lambda}[1]$ that are (shifted) irreducible perverse sheaves.

In addition, $\psi$ splits into direct sums of
$$\psi _j^i : {}^p H ^i ( \mathsf{gr} _{a+1} \, \mathbb C _{\lambda} ) [-i] \longrightarrow {}^p H ^{i+j} ( \mathsf{gr} _{a} \, \mathbb C _{\lambda} ) [1-i-j] \hskip 3mm \text{ and } \hskip 3mm \psi _j := \bigoplus _i \psi _j^i.$$
The morphism $\psi _j$ in $D _G ^b ( \mathfrak X )$ is a morphism of perverse sheaves if and only if $j = 1$ (and it is zero for $j > 1$). Therefore, the condition $(\dagger)_a$ implies $\psi _1 \neq 0$.

Consider the long exact sequence of perverse cohomologies associated to the distinguished triangle $(\mathsf{gr} _{a} \, \mathbb C _{\lambda}, F _{\ge a} \mathbb C _{\lambda}, F _{> a} \mathbb C _{\lambda})$:
$$\cdots \rightarrow {} ^p H ^{i} ( F _{\ge a} \mathbb C _{\lambda} ) \rightarrow {} ^p H ^{i} ( F _{> a} \mathbb C _{\lambda} ) \stackrel{\delta_{i}}{\longrightarrow} {} ^p H ^{i+1} ( \mathsf{gr} _{a} \, \mathbb C _{\lambda} ) \rightarrow {} ^p H ^{i+1} ( F _{\ge a} \mathbb C _{\lambda} ) \rightarrow \cdots.$$
Here all the connecting maps $\delta_i$ must be zero by Lemma \ref{std-filt} 4).

The condition $\psi_1^{i} \neq 0$ implies that the composition map
$${} ^p H ^{i} ( \mathsf{gr} _{a+1} \, \mathbb C _{\lambda} ) \longrightarrow {} ^p H ^{i} ( F _{> a} \mathbb C _{\lambda} ) \stackrel{\delta_{i-1}}{\longrightarrow} {} ^p H ^{i+1} ( \mathsf{gr} _{a} \, \mathbb C _{\lambda} )$$
is non-zero, which cannot happen. Therefore, we deduce $\psi _1 = 0$, and hence the condition $(\dagger)_a$ cannot hold.

This proves that the complex $(Q ( \mathbb C _{\lambda} ), d)$ is a minimal resolution of $\widetilde{K} _{\lambda}$, which completes the proof of Corollary \ref{kf}.
\end{proof}

\begin{proof}[Proof of Theorem \ref{KS}]
For $\mu \not\succsim \lambda = (\underline{\lambda},\xi)$, we have $j_{\lambda} ^! \mathsf{IC} _{\mu} = \{ 0 \}$ by the support condition. It follows that
\begin{align}
[\widetilde{K} _{\lambda} : L _{\mu} \left< i \right> ] _0 & = \dim \mathrm{Ext} ^i _{D^b_G (\mathfrak X)} ( \mathbb C _{\lambda}, \mathsf{IC} _{\mu} ) = \dim \mathrm{Ext} ^{i-\dim \mathbb O _{\lambda}} _{D^b_G (\mathbb O _{\lambda})} ( \underline{\xi}, j_{\lambda} ^! \mathsf{IC} _{\mu} ) = 0,\label{eq-mult}\\
[ K _{\lambda} : L _{\mu} \left< i \right> ] _0 & = \dim H ^{i+\dim \mathbb O _{\lambda}} i_{\lambda} ^! \mathsf{IC} _{\mu} = 0\label{mult}
\end{align}
for every $i \in \mathbb Z$.

For each $\lambda = (\underline{\lambda}, \xi), \mu = (\underline{\lambda}, \zeta) \in \Lambda$ (so that $\lambda \sim \mu$), we have
\begin{equation}
[ K _{\lambda} : L _{\mu} \left< i \right> ] _0 = \dim \mathrm{Hom} _{C_{\lambda}} ( \xi, H ^{i+\dim \mathbb O _{\lambda}} i_{\lambda} ^! \mathsf{IC} _{\mu} ) = \begin{cases} \delta _{\lambda,\mu} & (i=0) \\ 0 & (i \neq 0) \end{cases}.\label{BMone}
\end{equation}
This proves Theorem \ref{KS} 1). In addition, we have
\begin{equation}
[ \widetilde{K} _{\lambda} : L _{\mu} \left< i \right> ] _0 = \dim \mathrm{Ext} ^{i-\dim \mathbb O _{\lambda}} _{D^b_G (\mathbb O _{\lambda})} ( \underline{\xi}, j_{\lambda} ^! \mathsf{IC} _{\mu} ) = \dim \mathrm{Hom} _{C_{\lambda}} (\xi, \zeta \otimes _{\mathbb C} H ^i _{G _{\lambda} ^{\circ}} ( \{ \mathrm{pt} \} ) ).\label{BMeq}
\end{equation}
This proves the latter half of Theorem \ref{KS} 4).

Corollary \ref{kf} and (\ref{eq-mult}) implies
$$\mathrm{ext} ^\bullet _A ( \widetilde{K} _{\lambda}, \widetilde{K} _{\mu} ) = \{ 0 \} \hskip 1mm \text{ for each } \mu \not\precsim \lambda,$$
that is the first half of Theorem \ref{KS} 2). 

Similarly, Corollary \ref{kf}, (\ref{mult}), and (\ref{BMone}) implies
\begin{center}
\begin{minipage}{.45\textwidth}
\begin{align}\label{orthKS4}
\mathrm{ext} ^\bullet _A ( \widetilde{K} _{\lambda}, L _{\mu} ) & \neq \{ 0 \}
\end{align}
\end{minipage}
\hfill
\begin{minipage}{.45\textwidth}
\begin{align}\label{orthKS5}
\hskip -10mm \text{ or } \hskip 5mm \mathrm{ext} ^\bullet _A ( \widetilde{K} _{\lambda}, K _{\mu} ) &\neq \{ 0 \}
\end{align}
\end{minipage}
\end{center}
only if $\mu \preceq \lambda$. The equation (\ref{orthKS5}) is the latter half of Theorem \ref{KS} 2).

The equations (\ref{orthKS4}) and (\ref{eq-mult}) imply
\begin{equation}
\mathrm{hom} _A ( \widetilde{K}_{\lambda}, L _{\mu} ) = \{ 0 \} \hskip 3mm \text{ for every } \lambda \neq \mu \in \Lambda. \label{orthKS3}
\end{equation}
Thanks to Corollary \ref{kf} 4), we have
\begin{equation}
\sum _{i \ge 0} \dim \mathrm{ext}^i _{A} ( \widetilde{K}_{\lambda}, L _{\lambda} ) \le 1. \label{dimest}
\end{equation}

By Corollary \ref{kf} 2), we have a quotient map $P_{\lambda} \rightarrow \widetilde{K}_{\lambda}$. Taking (\ref{eq-mult}) into account, it induces a surjective morphism
$$\psi _{\lambda} : P _{\lambda} / \Bigl(  \sum _{\mu \prec \lambda} A e_{\mu} P _{\lambda} \Bigr) \longrightarrow \widetilde{K} _{\lambda}.$$
Let $\ker := \ker \psi _{\lambda}$. Each simple quotient of $\ker$ is of the form $L_{\gamma} \left< i \right>$ with $\gamma \not\prec \lambda$ and $i > 0$. If $\ker \neq \{ 0 \}$, then the Yoneda interpretation of $\mathrm{ext}^1$ implies
$$\mathrm{ext} ^1 _{A} ( \widetilde{K} _{\lambda}, L _{\gamma}) \neq \{ 0 \},$$
which contradicts with (\ref{orthKS4}) and (\ref{dimest}). Therefore, we conclude that $\ker = \{ 0 \}$. This proves Theorem \ref{KS} 3).

It remains to prove the former part of Theorem \ref{KS} 4). Fix a $\sim$-equivalence class $\mathcal O $ in $\Lambda$. Set $C := C_{\lambda}$ and $H ^{\bullet} _{G_{\lambda} ^{\circ}} := H ^{\bullet} _{G_{\lambda} ^{\circ}} ( \{ \mathrm{pt} \} )$ for a(ny) choice $\lambda \in \mathcal O$.

For every (\'etale) morphism $j : V \longrightarrow \mathfrak X$ onto a locally closed subset and every $\overline{\mathbb Q} _{\ell}$-sheaf $\mathcal F$ on $\mathfrak X$, we have the induced map:
$$\vartheta : \mathcal Hom _{\mathfrak X} ( \mathcal F, \mathcal F ) \rightarrow \mathcal Hom _{\mathfrak X} ( \mathcal F, j_* j^* \mathcal F ) \cong j_* \mathcal Hom _{V} ( j^* \mathcal F, j^* \mathcal F ).$$
By adjunction, this map is determined by
$$j^* \mathcal Hom _{\mathfrak X} ( \mathcal F, \mathcal F ) \longrightarrow \mathcal Hom _{V} ( j^* \mathcal F, j^* \mathcal F ),$$
that is a map of sheaves of algebras over $V$. By replacing $\mathcal F$ with a complex $\mathcal F^{\bullet}$, we obtain a map
$$j^* \mathcal Hom _{\mathfrak X} ( \mathcal F^{\bullet}, \mathcal F^{\bullet} ) \longrightarrow \mathcal Hom _{V} ( j^* \mathcal F^{\bullet}, j^* \mathcal F^{\bullet} )$$
of sheaves of differential graded algebras, whose cohomology yields a map of sheaves of graded algebras. Hence, the map $\vartheta$ defines a map of sheaves of (differential) graded algebras if we replace $\mathcal F$ by a complex $\mathcal F^{\bullet}$. It follows that
$$\mathrm{Ext} ^{\bullet} _G ( \mathcal L, \mathcal L ) \rightarrow \mathrm{Ext} ^{\bullet} _G ( \mathcal L, ( j_{\mu})  _* j _{\mu}^* \mathcal L ) \cong \mathrm{Ext} ^{\bullet} _G ( j_{\mu} ^* \mathcal L, j_{\mu} ^* \mathcal L )$$
is a map of algebras. Since we have $\mathcal L \cong \mathbb D \mathcal L$, the Verdier duality induces an algebra map
\begin{equation}
\mathrm{Ext} ^{\bullet} _G ( \mathcal L, \mathcal L ) \rightarrow \mathrm{Ext} ^{\bullet} _G ( ( j_{\mu} )_! j _{\mu}^! \mathcal L,  \mathcal L ) \cong \mathrm{Ext} ^{\bullet} _G ( j_{\mu} ^! \mathcal L, j_{\mu} ^! \mathcal L ).\label{!-alg}
\end{equation}

Consider a graded algebra
\begin{equation}
A _{\mathcal O} := H _{G_{\lambda} ^{\circ}} ^{\bullet} \otimes _{\mathbb C} \bigoplus _{\lambda = (\underline{\lambda}, \xi), \mu = ( \underline{\lambda}, \zeta ) \in \mathcal O} \mathrm{hom} _{\mathbb C} ( \xi \boxtimes K_{\lambda}, \zeta \boxtimes K_{\mu} )\label{AO}
\end{equation}
with the diagonal $C$-action. This algebra has a unique self-dual graded simple module $\oplus _{\lambda = (\underline{\lambda}, \xi) \in \mathcal O} \xi \boxtimes K_{\lambda}$ (that can be split-off if we take the $C$-action into account). Note that $\widetilde{K}_{\lambda}$ is a direct summand of the RHS of (\ref{!-alg}). It follows that the $A$-action on $\widetilde{K}_{\lambda}$ factors through the algebra $\mathrm{Ext} ^{\bullet} _G ( j_{\mu} ^! \mathcal L, j_{\mu} ^! \mathcal L )$, that is isomorphic to $( A _{\mathcal O} )^{C} \subset A _{\mathcal O}$.

For each $k \ge 0$, the ideal $H _{G_{\lambda} ^{\circ}} ^{\ge k} \subset H _{G_{\lambda} ^{\circ}} ^{\bullet}$ generates a two-sided ideal $J^k \subset A _{\mathcal O}$. Here $J^1$ is the graded Jacobson radical of $A_{\mathcal O}$. Consider the graded $A _{\mathcal O}$-module
\begin{equation}
\widetilde{K} := H _{G_{\lambda} ^{\circ}} ^{\bullet} \otimes _{\mathbb C} \bigoplus _{\mu = (\underline{\lambda}, \zeta) \in \mathcal O} \zeta \boxtimes K_{\mu}.\label{Lst}
\end{equation}
We have a natural identification $\widetilde{K} = \bigoplus _{\lambda = ( \underline{\lambda}, \xi ) \in \mathcal O} \xi \boxtimes \widetilde{K} _{\lambda}$ through the algebra map $A \rightarrow ( A_{\mathcal O} ) ^C$. For each $k$,
$$J ^k \widetilde{K} / J ^{k+1} \widetilde{K} \cong H _{G_{\lambda} ^{\circ}} ^{k} \otimes _{\mathbb C} \bigoplus _{\mu = (\underline{\lambda}, \zeta) \in \mathcal O} \zeta \boxtimes K_{\lambda}$$
is a semisimple graded $A _{\mathcal O}$-module. This induces a filtration of $\widetilde{K} _{\lambda}$ whose associated graded is a direct sum of the grading shifts of $K_{\mu}$ with $\mu \in \mathcal O$, which completes the proof of Theorem \ref{KS}.
\end{proof}

\section{Applications of Theorem \ref{KS}}
Keep the setting of section one. We assume the conditions $(\spadesuit)$ in Condition \ref{spadesuit}, $(\clubsuit)$ in Condition \ref{clubsuit}, and work over an algebraic closure of a finite field unless stated otherwise.

\begin{lemma}\label{fd}
For each $\lambda \in \Lambda$, we have $\dim K_{\lambda} < \infty$.
\end{lemma}

\begin{proof}
The assertion is equivalent to $\dim H ^{\bullet} i _{\lambda} ^! \mathcal L < \infty$, which follows from the fact that $\mathcal L$ is a constructible sheaf.
\end{proof}

\begin{corollary}\label{fc}
For each $\lambda \in \Lambda$, the $A$-module $K_{\lambda}$ has a finite composition series. \hfill $\Box$
\end{corollary}

\begin{proposition}\label{fgd-k}
We assume $(\spadesuit)$ and $(\clubsuit)$. For each $\lambda \in \Lambda$, $K_{\lambda}$ admits a finite resolution by the grading shifts of $\{\widetilde{K}_{\gamma}\} _{\gamma \sim \lambda}$. In addition, we have
$$\mathrm{ext} _A ^{\bullet} ( K_{\lambda}, K_{\mu} ) = \{ 0 \} = \mathrm{ext} _A ^{\bullet} ( K_{\lambda}, L_{\mu} ) \hskip 1mm \text{ for each } \mu \not\precsim \lambda.$$
\end{proposition}

\begin{proof}
Fix a $\sim$-equivalence class $\mathcal O$ in $\Lambda$, and we borrow the notation (namely $A_{\mathcal O}, C, \widetilde{K}$, and $H ^{\bullet} _{G _{\lambda} ^{\circ}}$) and setting from the proof of the former half of Theorem \ref{KS} 4). Here we set $A' _{\mathcal O} := ( A _{\mathcal O} ) ^C$ (cf. (\ref{AO})) for simplicity during this proof.

For each $\lambda \in \mathcal O$, $\widetilde{K} _{\lambda}$ is a projective object in the full subcategory $A \mathchar`-\mathsf{gmod} _{\mathcal O}$ of $A \mathchar`-\mathsf{gmod}$ consisting of objects $M$ such that $[ M : L_{\gamma} ] = 0$ if $\gamma \not\succsim \lambda \in \mathcal O$ by Theorem \ref{KS} 3).

In particular, we have
$$\mathsf{gdim} \, \mathrm{hom} _A ( \widetilde{K} _{\lambda}, \widetilde{K} _{\mu} ) = \mathsf{gdim} \, \mathrm{Hom} _{C} ( \zeta, \xi \otimes _{\mathbb C} H ^{\bullet} _{G _{\lambda} ^{\circ}} )$$
for each $\lambda = ( \underline{\lambda}, \xi ), \mu = ( \underline{\lambda}, \zeta ) \in \mathcal O$.

Consider the graded $A' _{\mathcal O}$-module
$$M_{\lambda} := \mathrm{hom} _C ( \xi, H_{G _{\lambda} ^{\circ}} ^{\bullet} \otimes _{\mathbb C} \bigoplus _{\mu = (\underline{\lambda}, \zeta) \in \mathcal O} \zeta \boxtimes K_{\mu} )$$
for each $\lambda = ( \underline{\lambda}, \xi ) \in \mathcal O$, that is the $\xi$-isotypic component of $\widetilde{K}$ (cf. (\ref{AO}) and (\ref{Lst})).

The factorization of the $A$-action on $\widetilde{K} _{\lambda}$ ($\lambda \in \mathcal O$) through $A' _{\mathcal O}$ identifies $M_{\lambda}$ with $\widetilde{K} _{\lambda}$. This also identifies (grading shifts of) $K_{\lambda}$ (that are subquotients of $\widetilde{K}_{\lambda}$) with a simple graded $A' _{\mathcal O}$-module.

For each $\lambda \in \mathcal O$, we deduce that $M_{\lambda}$ is a projective $A _{\mathcal O}'$-module with a simple head $K_{\lambda}$ by replacing $L_{\lambda}$ with $K_{\lambda}$ and $(G, \mathfrak X)$ with $( G_{\lambda}, \{ \mathrm{pt} \} )$ in the discussion of section one. It follows that
$$\mathsf{gdim} \, \mathrm{hom} _{A' _{\mathcal O}} ( M _{\lambda}, M _{\mu} ) = \mathsf{gdim} \, \mathrm{hom} _{C} ( \zeta, \xi \otimes _{\mathbb C} H ^{\bullet} _{G _{\lambda} ^{\circ}} )$$
for each $\lambda = ( \underline{\lambda}, \xi ), \mu = ( \underline{\lambda}, \zeta ) \in \mathcal O$.

The map $\psi \in \mathrm{end} _{A' _{\mathcal O}} ( M _{\lambda}, M _{\mu} )$ is uniquely determined by choosing the image of $K_{\lambda}$, and we have a canonical quotient $K _{\lambda} \rightarrow L_{\lambda}$ (or its kernel) singled out by its restriction to the image of $A \to A' _{\mathcal O}$. Thus, we have a natural inclusion
$$\mathrm{hom} _{A' _{\mathcal O}} ( M _{\lambda}, M _{\mu} ) \hookrightarrow \mathrm{hom} _A ( \widetilde{K} _{\lambda}, \widetilde{K} _{\mu} ),$$
that is in fact an isomorphism by the comparison of graded dimensions.

This shows that
\begin{equation}
\mathrm{end} _{A' _{\mathcal O}} ( \bigoplus _{\lambda \in \mathcal O} M _{\lambda} ) \cong \mathrm{end} _{A} ( \bigoplus _{\lambda \in \mathcal O} \widetilde{K} _{\lambda} )\label{trfb}
\end{equation}
as graded algebras. Here the LHS must be Morita equivalent to $A' _{\mathcal O}$.

We set $R_C := \mathbb C C \ltimes H ^{\bullet} _{G _{\lambda} ^{\circ}}$, and regard it as a graded algebra by setting $\deg \, C = 0$ and $\deg \, H ^{k} _{G _{\lambda} ^{\circ}} = k$ for each $k \ge 0$. Each $\xi \in \mathsf{Irr} \, C$ gives rise to an indecomposable graded $R_C$-module $P_{\xi}^C := H ^{\bullet} _{G _{\lambda} ^{\circ}} \otimes _{\mathbb C} \xi$, and we have $R_C \cong \bigoplus _{\xi \in \mathsf{Irr} \, C} P^C _{\xi} \boxtimes \xi$ as graded left $R_C$-modules (cf. \cite{K4} 1.2).

If we further replace each $M _{\lambda}$ for $\lambda = ( \underline{\lambda}, \xi ) \in \Lambda$ with $M_{\lambda} \boxtimes \xi$ in (\ref{trfb}), then we find an isomorphism
$$
R_C = \mathrm{end} _{R_C} ( \bigoplus _{\lambda = ( \underline{\lambda}, \xi ) \in \mathcal O} P^C _{\xi} \boxtimes \xi ) \stackrel{\psi}{\longrightarrow} \mathrm{end} _{A' _{\mathcal O}} ( \bigoplus _{\lambda = ( \underline{\lambda}, \xi ) \in \mathcal O} M _{\lambda} \boxtimes \xi )
.$$
Here we have $\psi = \oplus _{\kappa \in \mathsf{Irr} \, C} \psi _{\kappa}$, where $\psi _{\kappa}$ sends $\zeta^{\vee} \boxtimes \zeta \ltimes \theta \subset \mathbb C C \ltimes H^{\bullet} _{G _{\lambda} ^{\circ}}$ to
$$
\mathrm{hom} _{\mathbb C} ( K_{\mu} \boxtimes \zeta, K _{\mu} \otimes _{\mathbb C} \mathrm{hom} _C ( \kappa, \zeta \otimes _{\mathbb C} \theta ) \boxtimes \kappa ) \subset \mathrm{hom} _{\mathbb C} ( M_{\mu} \boxtimes \zeta, M_{\gamma} \boxtimes \kappa ),\label{exinc}
$$
for each $\mu = (\underline{\lambda}, \zeta ), \gamma = (\underline{\lambda}, \kappa ) \in \mathcal O$ and an irreducible $C$-submodule $\theta \subset H^{\bullet}_{G _{\lambda} ^{\circ}}$. This shows that $R_C$ is Morita equivalent to $A' _{\mathcal O}$.

Here (\ref{trfb}) exhibits that every graded $A$-module morphism between grading shifts of $\{ \widetilde{K} _{\lambda} \} _{\lambda \in \mathcal O}$ is in fact a pullback of a graded $A' _{\mathcal O}$-module morphism. In particular, its kernel and cokernel admits a decreasing separable $\{ K_{\lambda}\left< j \right>\} _{\lambda \in \mathcal O, j \in \mathbb Z}$-filtration. Hence, so is the homology of a graded complex whose terms are direct sums of grading shifts of $\{ \widetilde{K} _{\lambda}  \} _{\lambda \in \mathcal O}$.

By \cite{MR} 7.5.6, the global dimension of the algebra $R_C$ is the same as that of $H ^{\bullet} _{G _{\lambda} ^{\circ}}$, that is a polynomial ring. In particular, every simple module of $R_C$ admits a finite length projective resolution.

As a consequence, each $K_{\lambda}$ admits a finite length graded projective resolution as a graded $A' _{\mathcal O}$-module, and hence also admits a finite length graded $\{\widetilde{K} _{\lambda}\left< j \right>\} _{\lambda \in \mathcal O, j \in \mathbb Z}$-resolution as a graded $A$-module. This is the first assertion.

For each $\lambda \in \mathcal O$, we replace $K_{\lambda}$ by its finite resolution whose terms are grading shifts of $\{\widetilde{K} _{\gamma}\} _{\gamma \in \mathcal O}$ to compute $\mathrm{ext} _A ^{\bullet} ( K_{\lambda}, K_{\mu} )$ and $\mathrm{ext} _A ^{\bullet} ( K_{\lambda}, L_{\mu} )$ for $\mu \not\precsim \lambda$ via double complexes. Since we know
$$\mathrm{ext} _A ^{\bullet} ( \widetilde{K}_{\gamma}, K_{\mu} )= \{ 0 \} =\mathrm{ext} _A ^{\bullet} ( \widetilde{K}_{\gamma}, L_{\mu} )$$
for $\gamma \sim \lambda$ by Theorem \ref{KS} 2) and (\ref{orthKS4}) in the proof of Theorem \ref{KS}, the double complexes must be entirely zero. This shows the second assertion.
\end{proof}

\begin{corollary}\label{fgd-simple}
For each $\lambda \in \Lambda$, the $A$-module $L_{\lambda}$ admits a $($graded$)$ projective resolution of finite length.
\end{corollary}

\begin{proof}
Combining Proposition \ref{fgd-k} and Corollary \ref{kf} (through a double complex), we deduce that each $K_{\lambda}$ admits a finite length finitely generated graded projective resolution. By Corollary \ref{fc} and Theorem \ref{KS} 1), we see that each $L_{\lambda}$ is constructed by a finitely many successive short exact sequences from $\{K_{\mu} \left< i \right> \} _{\mu \in \Lambda, i \in \mathbb Z}$. Thus, we conclude
$$\sum _{\mu \in \Lambda} \dim \, \mathrm{ext} ^{\bullet} _A ( L_{\lambda}, L _{\mu} ) < \infty.$$
Therefore, a minimal projective resolution of $L_{\lambda}$ (that exists by Corollary \ref{mphd}) contains finitely many indecomposable projective $A$-modules as its direct summands (counted with multiplicities), which proves the assertion.
\end{proof}

\begin{theorem}\label{absKas} Assume the conditions $(\spadesuit)$ and $(\clubsuit)$. Then, the algebra $A$ has finite global dimension. In particular, we have $A \mathchar`-\mathsf{gmod} ^{pf} \stackrel{\cong}{\longrightarrow} A \mathchar`-\mathsf{gmod}$.
\end{theorem}

\begin{proof}
The graded algebra $A$ is (left and right) Noether, and is graded Morita equivalent to $B$. The graded algebra $B$ is non-negatively graded and $B^0$ is (canonically isomorphic to) the semi-simple quotient of the graded Jacobson radical $B^{> 0}$ of $B$. Therefore, we apply Li's result \cite{Li} (cf. \cite{NO} D.VII) to Corollary \ref{fgd-simple} to conclude the result.
\end{proof}

\begin{remark}
Applying the arguments of section two by an induction with respect to the closure relation (from open orbits), we can use the condition of Theorem \ref{critpure} a) instead $(\clubsuit)_2$ to deduce Theorem \ref{absKas}.
\end{remark}

For each $M \in A \mathchar`-\mathsf{gmod}$, we define its graded character as:
$$\mathsf{gch} \, M := \sum _{\lambda \in \Lambda} [ M : L _{\lambda} ] [ L _{\lambda} ] \in \bigoplus _{\lambda \in \Lambda} \mathbb Q (\!( t )\!) [ L _{\lambda} ].$$

\begin{lemma}\label{3b}
Each of the collections $\{ \mathsf{gch} \, K _{\lambda} \} _{\lambda \in \Lambda}$ and $\{ \mathsf{gch} \, \widetilde{K} _{\lambda} \} _{\lambda \in \Lambda}$ is a $\mathbb Q (\!( t )\!)$-basis of $\bigoplus _{\lambda \in \Lambda} \mathbb Q (\!( t )\!) [ L _{\lambda} ]$.
\end{lemma}

\begin{proof}
This is a rephrasement of Corollary \ref{utK}.
\end{proof}

For each $M \in A \mathchar`-\mathsf{gmod}$, we have collections of elements $[ M : P _{\lambda} ]$, $[ M : K _{\lambda} ]$, and $[ M : \widetilde{K} _{\lambda} ]$ in $\mathbb Q (\!( t )\!)$ so that
$$\mathsf{gch} M = \sum _{\lambda \in \Lambda} [ M : P _{\lambda} ] \, \mathsf{gch} \, P _{\lambda} = \sum _{\lambda \in \Lambda} [ M : K _{\lambda} ] \, \mathsf{gch} \, K _{\lambda} = \sum _{\lambda \in \Lambda} [ M : \widetilde{K} _{\lambda} ] \, \mathsf{gch} \, \widetilde{K} _{\lambda}.$$

Thanks to Theorem \ref{absKas}, every module in $A \mathchar`-\mathsf{gmod}$ can be a (first) variable of the graded Euler-Poincar\'e pairing (\ref{defnEP}), and $[ M : P _{\lambda} ] \in \mathbb Z [t^{\pm 1}]$.

\begin{lemma}\label{dualmult}
If $M \in A \mathchar`-\mathsf{gmod}$ is finite-dimensional, then its graded dual $M^*$ belongs to $A \mathchar`-\mathsf{gmod}$. In addition, we have
$$[M : L_{\mu}] = \overline{[M^* : L_{\mu ^{\vee}}]}.$$
\end{lemma}

\begin{proof}
See Lemma \ref{symm} and \cite{K4} Lemma 2.5.
\end{proof}

\begin{proposition}\label{absLS}
Assume the conditions $(\spadesuit)$ and $(\clubsuit)$. We have
$$\mathrm{ext}^{i}_A ( \widetilde{K} _{\lambda}, K _{\mu ^{\vee}} ^* ) \cong \begin{cases} \mathbb C & (\lambda = \mu, i = 0) \\ \{ 0 \} & (otherwise) \end{cases}, \text{ and } \left< \widetilde{K} _{\lambda}, K _{\mu ^{\vee}} ^* \right> _{\mathsf{gEP}} = \begin{cases} 1 & (\lambda = \mu) \\ 0 & (\lambda \neq \mu) \end{cases}.$$
\end{proposition}

\begin{remark}\label{absLSrem}
Proposition \ref{absLS} can be seen as the ``dual picture" of Mirollo-Vilonen and Beilinson-Ginzburg-Soergel \cite{MV, BGS}, or the equivariant picture of Chriss-Ginzburg \cite{CG} \S 8.7.
\end{remark}

\begin{proof}[Proof of Proposition \ref{absLS}]
Since the second assertion follows from the first one, we prove only the first assertion.

By Corollary \ref{kf} 4), Lemma \ref{dualmult}, and Theorem \ref{KS} 1), we deduce the case $\mu \not\precsim \lambda$. In this case, we further deduce
\begin{equation}
\mathrm{ext}^{\bullet}_A ( K _{\lambda}, K _{\mu ^{\vee}} ^* ) \cong \{ 0 \} \hskip 3mm \text{ if } \hskip 3mm \lambda \not\sim \mu\label{KBH}
\end{equation}
by Proposition \ref{fgd-k}.

Thus, we assume $\mu \prec \lambda$ in the below. Consider the derived functors of
$$\mathrm{hom} _A ( M, N ) \cong \mathrm{hom} _A ( N^*, M^* ).$$
Since $*$ is an exact functor (from $A\mathchar`-\mathsf{gmod}$ to its dual category) and $\mathrm{ext} _A^*$ is an universal $\delta$-functor on them, we conclude that both of them define mutually isomorphic derived functors. In particular, (\ref{KBH}) also holds for $\mu \prec \lambda$.

Since the grading of $A$ is bounded from the below, so are the gradings of $K_{\gamma}$ for every $\gamma \in \Lambda$ (say $A ^i = \{ 0 \}$ and $K_{\gamma} ^{i} = \{ 0 \}$ for every $\gamma \in \Lambda$ and $i < c$). Hence, we have $( K _{\mu ^{\vee}} ^* ) ^i = \{ 0 \}$ for every $i > - c$. Let $\ell$ be the global dimension of $A$, that is finite by Theorem \ref{absKas}. For each $j \in \mathbb Z$, there exists a member $K_j^{\perp} \subset \widetilde{K} _{\lambda}$ of a decreasing separable $A$-module filtration in Theorem \ref{KS} 4) such that {\bf a)} $( K_j^{\perp} ) ^i = \{ 0 \}$ for $i \le - j - ( \ell + 2 ) c$, and {\bf b)} $\widetilde{K} _{\lambda} / K_j^{\perp}$ is a finite successive self-extension of (grading shifts) of $K _{\gamma}$ with $\gamma \sim \lambda$. Then, the minimal projective resolution of $K_j^{\perp}$ is concentrated in degree $i > - j - c$ by Corollary \ref{mphd}. In particular, (\ref{KBH}) implies
$$\mathrm{ext} _{A} ^{\bullet} ( K_j^{\perp}, K _{\mu ^{\vee}} ^* )^j = \{ 0 \} = \mathrm{ext} _{R_{\beta}} ^{\bullet} ( \widetilde{K} _{\lambda} / K_j^{\perp}, K _{\mu ^{\vee}} ^*).$$
This yields $\mathrm{ext} _{A} ^{\bullet} ( \widetilde{K} _{\lambda}, K _{\mu ^{\vee}} ^* )^j = \{ 0 \}$ (for each $j$) as required.
\end{proof}

\begin{corollary}\label{absLS2}
Keep the setting of Proposition \ref{absLS}. We have
$$\mathrm{ext}^{\bullet}_A ( K _{\lambda}, K _{\mu} ^* ) = \{ 0 \} \hskip 3mm \text{ and } \hskip 3mm \left< K _{\lambda}, K _{\mu} ^* \right> _{\mathsf{gEP}} = 0 \hskip 5mm (\lambda \not\sim \mu).$$
\end{corollary}

\begin{proof}
See (\ref{KBH}) and the third paragraph of the proof of Proposition \ref{absLS}.
\end{proof}

For each $X, Y \in \{ P, \widetilde{K}, K, L \}$, we define a $\mathbb Q (\!(t)\!)$-valued $\# \Lambda$-square matrix as $[X:Y] = ( [X_{\lambda} : Y_{\mu}] ) _{\lambda,\mu \in \Lambda}$. In particular, we have $[P:\widetilde{K}] = ( [P_{\lambda} : \widetilde{K}_{\mu}] ) _{\lambda,\mu \in \Lambda}$ for $X = P, Y = L$. Similarly, we have $[P:\widetilde{K}],[\widetilde{K}:K]$, etc...

\begin{corollary}[Brauer-Humphreys type reciprocity]\label{absBH} Assume the conditions $(\spadesuit)$ and $(\clubsuit)$. Then, we have
$$[P _{\lambda} : \widetilde{K} _{\mu} ] = [K _{\mu} : L _{\lambda}] \hskip 3mm \text{ for each } \hskip 3mm \lambda, \mu \in \Lambda.$$
In addition, we have the matrix identity
\begin{equation}
[P:L] = [P:\widetilde{K}][\widetilde{K}:K][K:L] = {}^{\mathtt{t}} [K:L][\widetilde{K}:K][K:L].\label{BHrec}
\end{equation}
\end{corollary}

\begin{proof}
We have
\begin{align*}
\delta_{\lambda, \mu} & = \left< P_{\lambda}, L _{\mu} \right> _{\mathsf{gEP}} = \left< P_{\lambda}, L _{\mu^{\vee}} ^* \right> _{\mathsf{gEP}} = \sum _{\gamma, \delta} \overline{[P_{\lambda} : \widetilde{K} _{\gamma}]} \overline{[L_{\mu} : K _{\delta}]} \left< \widetilde{K}_{\gamma}, K _{\delta ^{\vee}} ^* \right> _{\mathsf{gEP}}\\
& = \sum _{\gamma} \overline{[P_{\lambda} : \widetilde{K} _{\gamma}]} \overline{[L_{\mu} : K _{\gamma}]}
\end{align*}
by Proposition \ref{absLS}. By applying the bar involution, this shows
$$( [P_{\lambda} : \widetilde{K} _{\gamma}] ) ( [K_{\delta} : L _{\mu}] )^{-1} = ( \delta_{\lambda, \mu}),$$
which is equivalent to the first assertion. The first equality of (\ref{BHrec}) is the definition. The second equality of (\ref{BHrec}) follows by $[P:\widetilde{K}] = {}^{\mathtt{t}} [K:L]$, that is the first assertion.
\end{proof}

\begin{corollary}[Cartan determinant formula]\label{cartan}
Assume that $C_{\lambda} = \{ 1 \}$ for every $\lambda \in \Lambda$. We have
$$\det \, [P:L] = \prod _{\lambda \in \Lambda} \mathsf{gdim} \, H ^{\bullet} _{G _{\lambda}} ( \{ \mathrm{pt} \}).$$
\end{corollary}

\begin{remark}\label{Rcartan}
{\bf 1)} It is straight-forward to formulate an analogue of Corollary \ref{cartan} without the extra assumption $C_{\lambda} = \{ 1 \}$ for every $\lambda \in \Lambda$. We choose the current formulation for the sake of simplicity; {\bf 2)} Corollary \ref{cartan} asserts $\det \, [P:L] \neq 0$. Hence the knowledge of $\Lambda$, $\prec$, and $[P:L]$ are enough to determine the matrix $( \mathsf{gdim} \, H ^{\bullet+\dim \mathbb O_{\lambda}} i_{\lambda}^! \mathsf{IC} _{\mu} )_{\lambda,\mu} = ( [K_{\lambda}:L_{\mu}] )_{\lambda,\mu}$ by Theorem \ref{KS} 1), 4) and (\ref{BHrec}) (that is a version of the Lusztig-Shoji algorithm; cf. \cite{K4} and Remark \ref{orSh} {\bf 1)}).
\end{remark}

\begin{proof}[Proof of Corollary \ref{cartan}]
We have $\det \, [P:L] = ( \det \, [K : L] )^2 \det \, [\widetilde{K}:K]$ by (\ref{BHrec}). By Theorem \ref{KS} 1), we have $\det \, [K : L] = 1$. Hence, the result follows from Theorem \ref{KS} 4).
\end{proof}

\section{An inheritance property of $(\spadesuit)$ and $(\clubsuit)$}
Keep the setting of the previous section. In particular, we assume the conditions $(\spadesuit)$ in Condition \ref{spadesuit}, and $(\clubsuit)$ in Condition \ref{clubsuit} (for the pair $(G,\mathfrak X)$), and work over an algebraic closure of a finite field unless stated otherwise.

For $\underline{\lambda} \in \underline{\Lambda}$, we set $e _{\underline{\lambda}} := \sum _{\mu = (\underline{\lambda}, \xi) \in \Lambda} e _{\mu}$. We denote by $\prec$ the partial order on $\underline{\Lambda}$ induced from $\prec$ on $\Lambda$ by simplicity.

The following proof of Theorem \ref{wfilt} (modulo Theorem \ref{absBH}) is rather well-understood (see e.g. Donkin \cite{Do}).

\begin{theorem}\label{wfilt}
Assume the conditions $(\spadesuit)$ and $(\clubsuit)$. Then, each $P_{\lambda}$ admits a decreasing $A$-module filtration so that its associated graded is a finite direct sum of $A$-modules of the form $\widetilde{K} _{\mu} \left< i \right>$ with $\mu \preceq \lambda$ and $i \ge 0$.
\end{theorem}

\begin{proof}
Let us introduce a total order $<$ on $\underline{\Lambda}$ which refines $\prec$. We name elements of $\underline{\Lambda}$ as $\{ \mathcal O _1, \mathcal O _2, \ldots \}$ so that $\mathcal O _i < \mathcal O _j$ if $i < j$. We set $e_i := \sum _{j < i} e_{\mathcal O_j}$. We also identify $\mathcal O_j$ with its preimage in $\Lambda$ by abuse of notation.

For each $1 \le i, j \le \# \Lambda$ and $\lambda \in \mathcal O _i$, we put $P_{\lambda} ^{(j)} := P _{\lambda} / A e_{j} P_{\lambda}$. We have $P_{\lambda} ^{(i)} \cong \widetilde{K} _{\lambda}$ by Theorem \ref{KS} 1) and 3). Since $P_{\lambda} ^{(j)}$ is the quotient of $P_{\lambda}$ obtained by annihilating $L _{\mu}$ with $\mu \in \mathcal O _k$ for $k < j$, the Yoneda interpretation of $\mathrm{ext}^1$ yields
\begin{equation}
\mathrm{ext} ^1 _A ( P_{\lambda} ^{(j)}, L _{\mu} ) = \{ 0 \} \hskip 5mm \text{ for every } \mu \in \mathcal O_k \text{ with } j \le k.\label{extone}
\end{equation}
Applying long exact sequences repeatedly, we deduce
\begin{equation}
\mathrm{ext} ^1 _A ( P_{\lambda} ^{(j)}, K ^* _{\mu} ) = \{ 0 \} \hskip 5mm \text{ for every } \mu \in \mathcal O_k \text{ with } j \le k.\label{K*extone}
\end{equation}
since we have $[K ^* _{\mu} : L_{\nu} ] \neq 0$ only if $\nu \in \mathcal O_{l}$ with $k \le l$ by Theorem \ref{KS} 1) and Lemma \ref{dualmult}. We have a short exact sequence
\begin{equation}
0 \to \ker _{\lambda}^j \rightarrow P_{\lambda} ^{(j)} \rightarrow P_{\lambda} ^{(j+1)} \rightarrow 0\label{P-seq}
\end{equation}
for each $1 \le i,j \le \# \Lambda$ and $\lambda \in \mathcal O_i$. By a comparison of multiplicities, we deduce
\begin{equation}
0 = [P_{\lambda} ^{(j)} : L _{\mu}] = [ \ker _{\lambda}^j : L _{\mu}]  \hskip 5mm \text{ if } \mu \in \mathcal O_k \text{ with } j > k.\label{kermult}
\end{equation}
This, together with Theorem \ref{KS} 3) and the construction, implies that each $\ker _{\lambda}^j$ is a quotient of a direct sum of grading shifts of $\widetilde{K} _{\mu}$ with $\mu \in \mathcal O_j$. 

To prove the claim, it suffices to show that $\ker _{\lambda}^j$ is a direct sum of grading shifts of $\widetilde{K} _{\mu}$ with $\mu \in \mathcal O_j$ for each $j \le i$ by a downward induction on $j$. The case $j=i$ is clear as $P_{\lambda} ^{(i)} = \ker _{\lambda}^i = \widetilde{K} _{\lambda}$. We assume that $\ker _{\lambda}^{j'}$ is a direct sum of grading shifts of $\widetilde{K} _{\mu}$ with $\mu \in \mathcal O_{j'}$ for each $j < j' \le i$ to deduce that $\ker _{\lambda}^j$ is a direct sum of grading shifts of $\widetilde{K} _{\mu}$ with $\mu \in \mathcal O_{j}$. For each $j \le k$ and $\mu \in \mathcal O_k$, (\ref{P-seq}) yields an exact sequence
\begin{equation}
\mathrm{ext} ^1 _A ( P_{\lambda} ^{(j+1)}, K ^* _{\mu} ) \to \mathrm{ext}^1 _A ( P_{\lambda} ^{(j)}, K ^* _{\mu} ) \longrightarrow \mathrm{ext} ^1 _A ( \ker _{\lambda}^j, K ^* _{\mu} ) \to \mathrm{ext} ^2 _A ( P_{\lambda} ^{(j+1)}, K ^* _{\mu} ).\label{4term}
\end{equation}
By induction hypothesis and Proposition \ref{absLS}, the most LHS/RHS of (\ref{4term}) are $\{ 0 \}$. Hence, applying (\ref{K*extone}) yields
\begin{equation}
\mathrm{ext} ^1 _A ( \ker _{\lambda}^j, K ^* _{\mu} ) = \{ 0 \} \hskip 5mm \text{ for every } \mu \in \mathcal O_k \text{ with } j \le k.\label{Kextone}
\end{equation}

Let $K$ be the minimal direct sum of (grading shifts of) $\{ \widetilde{K} _{\mu} \}_{\mu \in \mathcal O _j}$ which surjects onto $\ker _{\lambda}^j$. Let $K' := \ker \, ( K \to \ker _{\lambda} ^j )$. By the Noetherian hypothesis of $A$, we deduce that $\ker _{\lambda}^j$ is finitely generated. Thus, $K$ and $K'$ are also finitely generated. For each $j \le k$ and $\mu \in \mathcal O_k$, we apply $\mathrm{ext} ^{\bullet} _A ( \bullet, K^* _{\mu})$ to deduce
$$
0 \to \mathrm{hom} _A ( \ker _{\lambda}^j, K^* _{\mu} ) \stackrel{g}{\rightarrow} \mathrm{hom} _A ( K, K^* _{\mu} ) \to \mathrm{hom} _A ( K', K^* _{\mu} ) \to \mathrm{ext} ^1 _A ( \ker _{\lambda}^j, K^* _{\mu} ).$$
Since the image of a non-zero map $\widetilde{K} _{\gamma} \to K _{\gamma^{\vee}} ^*$ is $L_{\gamma}$ for each $\gamma \in \Lambda$ (cf. Proposition \ref{absLS}), the map $g$ is surjective by the construction of $K$. Hence, (\ref{Kextone}) implies
$$\mathrm{hom} _A ( K', K^* _{\mu} ) = \{ 0 \} \hskip 5mm \text{ for every } \mu \in \mathcal O_k\text{ with } j \le k.$$
We have $[K:L _{\nu}] \neq \{ 0 \}$ only if $\nu \in \mathcal O_k$ with $j \le k$ by Theorem \ref{KS} 1). Hence, the same is true for $[ K':L _{\nu} ]$. It follows that $K' = \{ 0 \}$ by Corollary \ref{hd}. Therefore, $K \cong \ker _{\lambda}^j$ is isomorphic to a finite direct sum of (grading shifts of) $\widetilde{K} _{\mu}$ with $\mu \in \mathcal O_j$. This proceeds the induction and we conclude the result.
\end{proof}

\begin{corollary}\label{etr}
Let $\lambda = (\underline{\lambda},\xi) \in \Lambda$ be a label of a closed $G$-orbit of $\mathfrak X$. Then, for each $\lambda \not\sim \gamma \in \Lambda$, we have
$$\mathrm{ext} ^{*}_{A} ( A, L_{\gamma} ) \cong \mathrm{ext} ^{*} _{A} ( A / ( A e _{\underline{\lambda}} A ), L_{\gamma} ).$$
\end{corollary}

\begin{proof}
By the construction in the proof of Theorem \ref{wfilt}, we know that $A e _{\underline{\lambda}} A$ is isomorphic to a direct sum of grading shifts of $\widetilde{K} _{\mu} \cong P _{\mu}$ with $\mu \sim \lambda$. It follows that
$$0 \rightarrow A e_{\underline{\lambda}} A \rightarrow A \rightarrow A / ( A e_{\underline{\lambda}} A ) \rightarrow 0 \hskip 5mm \text{(exact)}$$
is a projective resolution with $\mathrm{hom} _A  ( A e_{\underline{\lambda}} A, L _{\gamma} ) = \{ 0 \}$, which proves the assertion.
\end{proof}

Let $j : \mathfrak Y \hookrightarrow \mathfrak X$ be the inclusion of an open $G$-stable subvariety. We form a graded algebra
$$A_{(G, \mathfrak Y)} := \mathrm{Ext} ^{\bullet} _{G} ( j ^* \mathcal L, j ^* \mathcal L ).$$
For each $\lambda \in \Lambda$ so that $\mathbb O_{\lambda} \subset \mathfrak Y$, the module $\widetilde{K} _{\lambda}$ is in common between $A_{(G, \mathfrak X)}$ and $A_{(G,\mathfrak Y)}$ by construction.

\begin{proposition}\label{Kquotients}
Let $i : \mathbb O_{\lambda} \hookrightarrow \mathfrak X$ be the inclusion of a closed $G$-orbit $($with $\lambda = ( \underline{\lambda}, \xi ) \in \Lambda)$, and let $j : \mathfrak Y \hookrightarrow \mathfrak X$ be its complement. Then, we have an algebra map $\varpi : A_{(G, \mathfrak X)} \longrightarrow A _{(G, \mathfrak Y)}$ which induces an isomorphism
$$A_{(G, \mathfrak X)} / ( A_{(G, \mathfrak X)} e_{\underline{\lambda}} A_{(G, \mathfrak X)} ) \stackrel{\cong}{\longrightarrow} A _{(G, \mathfrak Y)}.$$
In addition, $\ker \, \varpi$ is a direct sum of $\{ \widetilde{K} _{\mu} \left< k \right> \} _{\mu \sim \lambda, k \in \mathbb Z}$.
\end{proposition}

\begin{proof}[Proof of Proposition \ref{Kquotients}]
We set $A' := A_{(G, \mathfrak Y)}$ and $A := A _{(G, \mathfrak X)}$. The algebra map $\varpi : A \to A'$ is the restriction map
$$\mathrm{Ext} ^* _{G} ( \mathcal L, \mathcal L ) \longrightarrow \mathrm{Ext} ^* _{G} ( j ^* \mathcal L, j ^* \mathcal L ) \cong \mathrm{Ext} ^* _{G} ( j ^! \mathcal L, j ^! \mathcal L ),$$
where we used $j ^* \cong j ^!$ for an open embedding $j$. It follows that this map, viewed as a left $A$-module, is given by
$$\mathrm{Ext} ^* _{G} ( \mathcal L, \mathcal L ) \longrightarrow  \mathrm{Ext} ^* _{G} ( j _!  j ^! \mathcal L, \mathcal L ).$$

We have a distinguished triangle
$$\to j _!  j ^! \mathcal L \to \mathcal L \to i _* i ^* \mathcal L \stackrel{+1}{\longrightarrow}.$$
It yields the following short exact sequence of graded $A$-modules:
$$0 \longrightarrow \mathrm{coker} \, \varpi \left< 1 \right> \rightarrow K \stackrel{f}{\longrightarrow} A \stackrel{\varpi}{\longrightarrow} A' \rightarrow \mathrm{coker} \, \varpi \rightarrow 0.$$
By the purity assumption $(\clubsuit)_2$, the sheaf $i _* i ^* \mathcal L$ is a direct sum of $\{ \mathbb C _{\mu} [k] \} _{\mu \sim \lambda, k \in \mathbb Z}$. It follows that $K$ is a direct sum of grading shifts of $\{ \widetilde{K} _{\mu} \} _{\mu \sim \lambda}$. Therefore, the second assertion modulo the surjectivity of $\varpi$ follows. We have an exact sequence:
\begin{equation}
0 \longrightarrow \mathrm{coker} \, \varpi \left< 1 \right> \rightarrow K \stackrel{f}{\longrightarrow} A \longrightarrow \mathrm{Im} \, \varpi \rightarrow 0.\label{ImgP}
\end{equation}
By Theorem \ref{KS} 3), we have $\widetilde{K} _{\mu} = P_{\mu}$ for $\mu \sim \lambda$. Regarding $(\ref{ImgP})$ as the first two terms of a projective resolution, we deduce
\begin{align}
& \mathrm{ext} ^1_A ( \mathrm{Im} \, \varpi, L _{\gamma} ) = \{ 0 \} \hskip 5mm \text{ for every } \gamma \not\sim \lambda, \hskip 3mm \text{ and } \label{Imgext}\\
& \mathrm{ext} ^{i}_A ( \mathrm{coker} \, \varpi, L _{\gamma} ) \left< - 1 \right> \cong \mathrm{ext} ^{i+2}_A ( \mathrm{Im} \, \varpi, L _{\gamma} ) \hskip 2mm \text{ for every } \hskip 1mm i \in \mathbb Z _{\ge 0} \hskip 1mm \text{ and } \hskip 1mm \gamma \in \Lambda.\label{connKC}
\end{align}
In addition, we have $[A' : L _{\mu}] _{A} = 0$ for $\mu \sim \lambda$ since $L_{\mu}$ is the multiplicity space of $\mathsf{IC} _{\mu}$ in $\mathcal L$ and $j ^! \mathsf{IC} _{\mu} = \{ 0 \}$. Moreover, we have
$$[\mathrm{Im} \, \varpi : L _{\mu}] _A \le [A' : L _{\mu}] _A = 0, \text{ and } [\mathrm{coker} \, \varpi : L _{\mu}] _{A} \le [A' : L _{\mu}] _{A} = 0$$
for each $\mu \sim \lambda$. This, together with (\ref{Imgext}) and the shape of $K$, implies that $\mathrm{Im} \, \varpi \cong A / ( A e _{\underline{\lambda}} A )$. The surjectivity of $\varpi$ is equivalent to $\mathrm{coker} \, \varpi = \{0\}$, and it is further equivalent to
\begin{equation*}
\mathrm{ext} ^2_A ( \mathrm{Im} \, \varpi, L _{\gamma} ) = \mathrm{ext} ^2_A ( A / ( A e _{\underline{\lambda}} A ), L _{\gamma} ) = \{ 0 \} \text{ for every } \gamma \not\sim \lambda
\end{equation*}
by (\ref{connKC}). This follows from Corollary \ref{etr} as desired.
\end{proof}

\begin{corollary}\label{localKsc}
Let $j : \mathfrak Y \hookrightarrow \mathfrak X$ be the inclusion of an open $G$-stable subvariety. Then, $\mathfrak Y$ satisfies the conditions $(\spadesuit)$ and $(\clubsuit)$ if $\mathfrak X$ does.
\end{corollary}

\begin{proof}
The conditions $(\spadesuit)$ and $(\clubsuit)_2$ become weaker by the restriction to any $G$-invariant locally closed subset.

Let $\mathbb O_{\lambda}$ be a closed $G$-orbit of $\mathfrak X$. We set $\mathbb O_{\lambda} = \mathfrak X \backslash \mathfrak Y$. By Proposition \ref{Kquotients}, we have a natural short exact sequence as graded left $A_{(G, \mathfrak X)}$-modules
\begin{equation}
0 \to K \to A_{(G, \mathfrak X )} \to A _{(G, \mathfrak Y)} \to 0.\label{se-sp}
\end{equation}
This particularly shows the condition $(\clubsuit)_1$ for $\mathfrak X$ implies that of $\mathfrak Y$ when $\mathfrak X \backslash \mathfrak Y$ is a single $G$-orbit. Hence, we deduce the result by induction.
\end{proof}

\begin{remark}
Thanks to Corollary \ref{localKsc}, analogues results of Corollary \ref{etr} and Proposition \ref{Kquotients} also hold for an arbitrary inclusion of closed $G$-stable subsets.
\end{remark}

\section{A proof of Shoji's conjecture for type $\mathsf B$}\label{Sconj}
In this section, we work with varieties over $\mathbb C$ unless stated otherwise.

Let $G = \mathop{Sp} ( 2n, \mathbb C )$ be a symplectic group, and let $V_1$ be its vector representation. We set $V_2 := \wedge ^2 V_1$ and $\mathbb V := V_1 \oplus V_2$. We fix a maximal torus $T \subset G$ and a Borel subgroup $B \subset G$ so that $T \subset B$. Let $W := N_G ( T )/ T$. Let $X ^* ( T )$ be the character lattice of $T$, which we may identify with the cocharacter lattice via a $W$-invariant perfect pairing. We fix a basis $\epsilon_1, \epsilon_2, \ldots, \epsilon_n$ of $X^* ( T )$ so that the set $R$ of coroots of $G$ and the set $R^+$ of positive coroots with respect to $B$ are presented as:
$$R := \{ \pm \epsilon _i \pm \epsilon _j \} _{i < j} \cup \{ \pm \epsilon _i \} _{i=1}^{n} \supset \{ \epsilon _i \pm \epsilon _j \} _{i < j} \cup \{ \epsilon _i \} _{i=1}^{n} =: R^+.$$

For each $\beta \in X ^* ( T )$, let $\mathbb V [\beta]$ be the weight $\beta$-part of $\mathbb V$. Note that we have $\dim \mathbb V [\beta] \le 1$ if $\beta \neq 0$. We set $\mathbb V ^+ := \bigoplus _{\beta \in R^+} \mathbb V [ \beta ] \subset \mathbb V$, that is a $B$-submodule. We consider a $G$-equivariant vector bundle $F := G \times ^B \mathbb V^+$ and a composition map
$$\mu : F \hookrightarrow G \times ^B \mathbb V \cong G / B \times \mathbb V \stackrel{\mathsf{pr}_2}{\longrightarrow} \mathbb V,$$
which is again $G$-equivariant. We denote the image of $\mu$ by $\mathfrak N$ and denote the resulting map $F \to \mathfrak N$ again by $\mu$.

For each $x \in \mathfrak N ( \mathbb C )$, the composition map $\mu^{-1} ( x ) \hookrightarrow F \to G/B$ is injective. It induces a map
$$\imath_x : H_{\bullet} ( \mu^{-1} ( x ), \mathbb C) \longrightarrow H_{\bullet} ( G / B, \mathbb C )$$
between the Borel-Moore homologies (see \cite{CG} \S 2 for example).

We put $\mathbb G = G \times ( \mathbb C ^{\times} ) ^2$. An element $a = ( s, q_1,q_2) \in \mathbb G$ acts on $v_1 \oplus v_2 \in \mathbb V$ as $a.v = q_1 s v_1 \oplus q_2 s v_2$.

\begin{theorem}[\cite{K1, K2} and Lusztig-Spaltenstein \cite{LS}]\label{eS} We have:
\begin{enumerate}
\item The variety $\mathfrak N$ has finitely many $G$-orbits, and a $G$-orbit is also a $\mathbb G$-orbit;
\item For each $x \in \mathfrak N ( \mathbb C )$, the groups $\mathsf{Stab} _G ( x )$ and $\mathsf{Stab} _{\mathbb G} ( x )$ are connected;
\item The map $\mu$ is strictly semi-small, and hence $\mu _* \mathbb C [ \dim F ]$ is a direct sum of $G$-equivariant simple perverse sheaves;
\item The sheaf $\mathcal L := \mu _* \mathbb C [ \dim F ]$ contains all $G$-equivariant simple perverse sheaves on $\mathfrak N$ as its direct summands;
\item We have
$$A : = \mathrm{Ext} ^{\bullet}_{D^b_{G}(\mathfrak N)} ( \mathcal L, \mathcal L ) \cong \mathbb C W \ltimes \mathbb C [\mathfrak t],$$
where $\mathbb C W$ is a group algebra of $W$ sitting in degree $0$ and $\mathbb C [\mathfrak t]$ is a polynomial algebra generated by $\mathfrak t^*$ in degree $2$;
\item The odd-part of the Borel-Moore homology $H_{odd} ( \mu ^{-1} ( x ), \mathbb C )$ vanishes;
\item Let $d_x = 2 \dim \mu^{-1} ( x )$. Then, $\imath_x H_{d_x} ( \mu^{-1} ( x ), \mathbb C)$ is an irreducible $W$-submodule of $H_{\bullet} ( G/ B, \mathbb C)$ that we denote by $\mathsf L_x$;
\item Let $H_{- \bullet} ( G / B, \mathbb C) \subset \mathbb C [ \mathfrak t^* ]$ be the harmonic polynomial realization $(\cite{CG}$ \S $6$ with the convention $\deg \, \mathfrak t = -2$; here $\mathfrak t^*$ acts on $H_{- \bullet} ( G / B, \mathbb C)$ via derivatives with degree two$)$. Then, the module $\mathsf L_x$ is the unique maximal degree realization $($of an irreducible $W$-module$)$;
\item There exists a reflection subgroup $W_x \subset W$ and a polynomial $p_x \in \mathsf L _x$ so that $\mathbb C p_x$ is the maximal degree realization of a one-dimensional representation of $W_x$ inside $\mathbb C [ \mathfrak t^* ]$.
\end{enumerate}
\end{theorem}

\begin{proof}
The first two assertions for $G$ are in \cite{K1} 1.14. The coincidence of $G$-orbits and $\mathbb G$-orbits follows by the form of a representative of each orbit {\it loc. cit.} 1.13. The second assertion for $\mathbb G$ can be deduced from {\it loc. cit.} 4.10 and its proof. The third assertion follows by \cite{K1} 1.2 and \cite{K2} 3.4, the fourth is \cite{K1} 8.3, the fifth is \cite{K1} 8.1, and the sixth is \cite{K1} 6.2. The seventh assertion follows from \cite{K2} 7.6, 10.7 and \cite{CG} 6.5.2 (or rather its proof). The eighth and the ninth assertions follow from \cite{LS} (cf. \cite{K2} 10.5). Here if we denote the bipartition parametrizing the $G$-orbit $G.x$ by $(\mu,\nu)$ by \cite{K2} 5.1, then we have $p_x = \mathsf{D} ( \mu, \nu )$ in {\it loc. cit.} 10.4. In addition, $W_x$ is the product of smaller Weyl groups of type $\mathsf{BC}$ whose sizes are the entries of ${} ^{\mathtt t} \mu$ and ${} ^{\mathtt t} \nu$.
\end{proof}

Since $\mathfrak t^* \subset A$ in Theorem \ref{eS} 5) is the space of (virtual) hyperplane sections (cf. \cite{K1} \S 8), we deduce that each $\iota_x$ is an $A$-module map.

The $G$-variety $\mathfrak N$ is a specialization of a flat $\mathbb Z$-scheme $\mathfrak N _{\mathbb Z}$ defined by the same defining equations as $\mathfrak N$ (namely $G$-invariant polynomials on $\mathbb V$ with integer coefficients and constant terms zero; cf. \cite{K2} \S 2) to $\mathbb C$. Hence, for an algebraic closure $\Bbbk$ of a field with $p$-element, we have a variety $\mathfrak N _{\Bbbk}$ ($:= \mathfrak N_{\mathbb Z} \otimes _{\mathbb Z} \Bbbk$) over $\Bbbk$ with an action of $G _{\Bbbk}$ ($:= \mathop{Sp} ( 2n, \Bbbk )$).

\begin{proposition}\label{eNc}
For $p \gg 0$, the conditions $(\spadesuit)$ in Condition $\ref{spadesuit}$ and $(\clubsuit)$ in Condition $\ref{clubsuit}$ holds for $\mathfrak N _{\Bbbk}$. In addition, we have $A \cong A _{(G _{\Bbbk}, \mathfrak N _{\Bbbk})}$ in the notation of section one.\\
Moreover, the algebra $A$ and its standard/dual standard modules $K _{\lambda} / \widetilde{K} _{\lambda}$ defined from $\mathfrak N$ $($over $\mathbb C)$ satisfies the conclusions of Theorem \ref{KS}, Proposition \ref{fgd-k}, and Corollaries \ref{absLS2} and \ref{absBH}.
\end{proposition}

\begin{proof}[Proof of Proposition \ref{eNc}]
As in \cite{BBD} \S 6 and \cite{BL}, we transplant the notion of (equivariant) derived category of mixed constructible/perverse sheaves from the sufficiently large characteristic case to the characteristic zero case by (the scalar extension of) \cite{BBD} 6.1.9 (note that our sheaves have finite monodromy by its $G$-equivariancy and $\# ( G \backslash \mathfrak N ) < \infty$, and such constructible $\overline{\mathbb Q} _{\ell}$-sheaves are in common over $\Bbbk$ and over $\mathbb C$; cf. \cite{BBD} 6.1.2A").

The condition that $(\spadesuit)$ and $(\clubsuit)$ holds for sufficiently large characteristic is equivalent to that to be hold in characteristic zero. In addition, the latter part of the assertion follows by the identification of their $\mathrm{Hom}$-spaces. Therefore, it suffices to verify the conditions $(\spadesuit)$ and $(\clubsuit)$ in characteristic zero.

Theorem \ref{eS} 1) and 2) asserts $(\spadesuit)$, 3) and 4) imply that the algebra $A$ (in Theorem \ref{eS} 5)) is the one we described in section one. The existence of an $\alpha$-partition (cf. \cite{DLP} 1.3) of the variety $F \times _{\mathfrak N} F$ into $G$-equivariant vector bundles on $G / B$ (corresponding to equi-dimensional irreducible components; all of them are $\dim \, F = 2 \dim G / B$ dimensional) can be read off from \cite{K1} 1.5 and 2.3. It follows that the $G$-equivariant Borel-Moore homology
$$H_{\bullet} ^G ( F \times _{\mathfrak N} F ) := H ^{- \bullet} ( BG, p_* p^! \mathbb C )$$
(where $p : EG \times _G ( F \times _{\mathfrak N} F ) \rightarrow BG$) is generated by its top term by the $\mathbb C [ \mathfrak t ]$-action (arising from $H ^{\bullet} _G ( G / B, \mathbb C )$ with the convention $\deg \, \mathfrak t^* = -2$) by \cite{CG} \S 5. By the isomorphism $A \cong H_{\bullet} ^G ( F \times _{\mathfrak N} F )$ as an algebra (cf. \cite{CG} \S 8), we conclude that $A$ is generated by $\mathfrak t^*$ and $\mathrm{Hom} _G ( \mathcal L, \mathcal L )$. Here the both of $\mathfrak t^*$ and $\mathrm{Hom} _G ( \mathcal L, \mathcal L )$ are weight zero (actually we count this on $\mathfrak N _{\Bbbk}$ for $p \gg 0$ to verify it on $\mathfrak N$), and hence $(\clubsuit)_1$.

We show $(\clubsuit)_2$ by Theorem \ref{critpure}. Since $(\spadesuit)$ and $(\clubsuit)_1$ also holds for $G$ replaced with $\mathbb G$, we can use the $\mathbb G$-action instead of the $G$-action. By Theorem \ref{eS} 6), the condition $a)'$ of Theorem \ref{critpure} follows.

We verify Theorem \ref{critpure} $b)$. Let $\mu \prec \lambda \in \Lambda$ (we borrow the notation from section one as we have $(\spadesuit)$). There exists a semisimple element $A _{\mu} \in ( \mathrm{Lie} \, \mathbb G ) ( \mathbb R )$ which gives $a _{\mu} := \exp A _\mu \in \mathbb G$ so that $\mathbb O _{\mu} ^{a_{\mu}} \subset \mathbb V ^{a_{\mu}}$ is open dense\footnote{By \cite{K1} \S 1.3, each $x_{\mu}$ is presented by a signed partition $\mathbf J = \{ J_1, J_2,\ldots \}$ dividing $[1,n]$ and a foot function $\delta_1 : [-n,n] \setminus \{ 0 \} \to \{ 0, 1 \}$ corresponding to a strict normal form {\it loc. cit.} 1.14. Fix $A_{\mu} := (S,r_1,r_2) \in ( \mathrm{Lie} \, T ) ( \mathbb R ) \oplus \mathbb R ^{\oplus 2}$ so that $a_{\mu} x_{\mu} = x_{\mu}$, $r_1, r_2 > 0, 2 r_1 \not\in \mathbb Z r_2$, and
$$\{ \pm \epsilon_i (S) \} _{i \in J_k^+} \cap \left\{ ( \{ \pm \epsilon_i (S) \} _{i \in J_{k'}^+} + \mathbb Z r_2) \cup \frac{1}{2} \mathbb Z r_2 \right\} = \emptyset \hskip 2mm \text{ for each } \hskip 2mm k,k' \hskip 2mm \text{ so that } \hskip 2mm \delta _1 (J_k) = \{ 0 \}.$$
Let $\mathbf J^0 := \{ J _k \mid \delta _1 ( J _k ) = \{ 0 \} \}$. Then, the centralizer $G^{\mu}$ of $A_{\mu}$ in $G$ is naturally contained in $\mathop{Sp} ( 2 n_0 ) \times \prod _{J_k \in \mathbf J^0} \mathop{GL} ( |J_k| )$, where $n_0:= n - \sum _{J_k \in \mathbf J^0} |J_k|$. Moreover, $\mathfrak N^{a_{\mu}}$ is contained in the product of nilpotent cones of $\mathop{GL} ( |J_k| )$ and smaller $\mathfrak N$ for $\mathop{Sp} ( 2n_0 )$ (say $\mathfrak N_{n_0}$), and hence we have a product decomposition of the situation. Here $x_{\mu}$ is decomposed into the product of regular nilpotent elements of $\mathop{GL} ( |J_k| )$ and an element $x_{\mu}' \in \mathfrak N_{n_0}$ (cf. {\it loc. cit.} 1.11). Thus, each $\mathop{GL} ( |J_k| )$-part defines a dense open subset.

Hence, we only need to consider the case of strict marked partition with $\delta_1 ( J_k ) \neq \{ 0 \}$ for all $J_k$ (i.e. the case $n = n_0$). Decompose $x_{\mu} = x_{\mu}^1 \oplus x_{\mu}^2 \in V_1 \oplus V_2$. The condition $2 r_1 \not\in \mathbb Z r_2$ guarantees that $G^{\mu} x_{\mu}^2 \subset V_2 ^{a_{\mu}}$ is open dense. Now $G^{\mu} x_{\mu}$ gives an open dense subset of a rank $\dim V_1 ^{a_{\mu}}$ vector bundle over $G^{\mu} x_{\mu}^2$, and hence $G^{\mu} x_{\mu} \subset \mathfrak N ^{a_{\mu}}$ is open dense.}. Then, we have $\mathbb O _{\lambda} ^{a _{\mu}} = \emptyset$. It follows that $A_\mu$ is $\mathbb G$-conjugate to an element of $\mathrm{Lie} \, \mathsf{Stab} _{\mathbb G} ( x _{\mu} )$, but is not $\mathbb G$-conjugate to an element of $\mathrm{Lie} \, \mathsf{Stab} _{\mathbb G} ( x _{\lambda} )$. Passing to the adjoint quotients, we have maps
$$\mathrm{Lie} \, \mathsf{Stab} _{\mathbb G} ( x _{\mu} ) /\!\!/ \mathsf{Stab} _{\mathbb G} ( x _{\mu} ) \stackrel{\vartheta_{\mu}}{\longrightarrow} \mathrm{Lie} \, \mathbb G /\!\!/ \mathbb G \stackrel{\vartheta_{\lambda}}{\longleftarrow} \mathrm{Lie} \, \mathsf{Stab} _{\mathbb G} ( x _{\lambda} ) /\!\!/ \mathsf{Stab} _{\mathbb G} ( x _{\lambda} ).$$
The point $A_{\mu} /\!\!/ \mathbb G$ lies in the image of $\vartheta _{\mu}$, but not lie on the image of $\vartheta _{\lambda}$. In particular, evaluation at $A_\mu$ defines proper ideals of $H^{\bullet} _{\mathbb G} ( \{ \mathrm{pt} \} )$ and $H^{\bullet}_{\mathbb G} ( \mathbb{O} _{\mu} )$, but it does not define a proper ideal of $H^{\bullet}_{\mathbb G} ( \mathbb{O} _{\lambda} )$. Therefore, the natural map $\psi _{\mu} : H^{\bullet}_{\mathbb G} ( \{ \mathrm{pt} \} ) \rightarrow H^{\bullet}_{\mathbb G} ( \mathbb{O} _{\mu} )$ does not factor through $\mathrm{Im} \, \psi _{\lambda}$, which verifies the condition of Theorem \ref{critpure} $b)$ as required.
\end{proof}

\begin{corollary}\label{espure}
For each $x \in \mathfrak N ( \mathbb C )$, the variety $\mu^{-1} (x)$ has a pure homology. \hfill $\Box$
\end{corollary}

Thanks to Proposition \ref{eNc}, we can apply the machinery and notation of section one. In particular, we have $\underline{\Lambda} = G \backslash \mathfrak N$ and its closure relation $\prec$. We set $A_n := A$ and replace $x$ with $\lambda \in \Lambda$ such that $x$ is $G$-conjugate to $x_{\lambda} \in \mathbb O_{\lambda}$ in the below.

\begin{lemma}\label{As} Under the above setting, the followings hold:
\begin{enumerate}
\item We have $C _{\lambda} = \{1\}$ for every $\lambda \in \Lambda$. In particular, we have $\Lambda = \underline{\Lambda}$;
\item Each of the graded simple $A$-module $L_{\lambda}$ is concentrated in degree zero;
\item For each $\lambda \in \Lambda$, the $A$-module $L_{\lambda}$ is identified with the irreducible $W$-module $\mathsf{L} _{\lambda} \left< d_{\lambda} \right>$ through the embedding $\mathbb C W = A ^0 \subset A$.
\end{enumerate}
\end{lemma}

\begin{proof}
The first assertion follows by Theorem \ref{eS} 2). The second assertion follows by Theorem \ref{eS} 3). The third assertion follows by the presentation of $A$ in Theorem \ref{eS} 5), and conventions in Theorem \ref{eS} 7), 8).
\end{proof}

Let $\mathsf{triv}$ and $\mathsf{sgn}$ be the trivial and sign representations of $W$, respectively. We define $\mathsf{Ssgn}$ to be the unique one-dimensional representation of $W$ so that $\mathsf{Ssgn} \not\cong \mathsf{sgn}$ as $W$-representations and $\mathsf{Ssgn} \cong \mathsf{sgn}$ as $\mathfrak S_n$-representations.

\begin{proposition}[\cite{K4} 10.7, cf. sign twisted from \cite{K1} 8.3]\label{eorb}
Let $\lambda _0, \lambda_1 \in \Lambda$ be the labels corresponding to $W$-representations $\mathsf{sgn}$ and $\mathsf{Ssgn}$, respectively. We have $\mathbb O _{\lambda _0} = \{0\}$, and $\mathbb O _{\lambda _1}$ is the open dense $G$-orbit in $V_1 \subset \mathfrak N$ $($here $\mathsf{triv}$ corresponds to the open dense $G$-orbit of $\mathfrak N)$. \hfill $\Box$
\end{proposition}

For each $\lambda \in \Lambda$, we have
\begin{align*}
K_{\lambda} & = H ^{\bullet} i_{\lambda} ^! \mathcal L[ \dim \mathbb O _{\lambda}] \cong \bigoplus _{i \in \mathbb Z} H ^{i} i_{\lambda} ^! \mathcal L[ \dim \mathbb O _{\lambda}]\\
& = \bigoplus _{i \in \mathbb Z} H ^{i} i_{\lambda} ^! \mu _* \underline{\mathbb C} [ \dim F + \dim \mathbb O _{\lambda}] \cong \bigoplus _{i \in \mathbb Z} H ^{i} ( \mu^{-1} ( x _{\lambda} ), \omega [ - d _{\lambda} ] ),
\end{align*}
where $\omega$ is the dualizing complex of $\mu^{-1} ( x _{\lambda} )$ and the last equation follows from the base change (and strict semi-smallness for the grading). This implies that $K_{\lambda} \cong H_{- \bullet} ( \mu^{-1} ( x _{\lambda} ), \mathbb C ) \left< d_{\lambda} \right>$ as graded $A$-modules.

Recall that each $K_{\lambda}$ is a graded $A$-module presented as:
\begin{equation}
K_{\lambda} = P _{\lambda} / \sum _{f \in \mathrm{hom} _{A} ( P_{\mu}, P _{\lambda}) ^{> 0}, \mu \preceq \lambda} \mathrm{Im} \, f\label{defPtr}
\end{equation}
by Theorem \ref{KS} 3) and 4). The main result in this section is:

\begin{theorem}\label{Shoji}
For each distinct $\lambda, \mu \in \Lambda$, we have
\begin{equation}
\mathrm{ext} ^{\bullet} _A ( K _{\lambda}, K _{\mu} ^* ) = \{ 0 \},  \hskip 3mm \text{ and } \hskip 3mm \left< K _{\lambda}, K _{\mu} ^* \right> = 0.\label{eLSorth}
\end{equation}
In addition, each $\imath _{\lambda}$ is an inclusion of $A$-modules.
\end{theorem}

Before proving Theorem \ref{Shoji}, let us summarize its consequences. For each $\lambda \in \Lambda$, we define an ideal
$$I_{\lambda} := \{ f \in \mathbb C [ \mathfrak t ] \mid f \mathsf{L}_{\lambda} = \{ 0 \} \}.$$

\begin{corollary}[Shoji's conjecture \cite{Sh4} 3.13]\label{SC}
For each $\lambda \in \Lambda$, the module $K_{\lambda} \left< - d_{\lambda} \right>$ is identified with the span of derivatives of $\mathsf{L} _{\lambda}$ in the polynomial ring $\mathbb C [ \mathfrak t^* ]$. In addition, we have $K_{\lambda}^* \left< d_{\lambda} \right> \cong \mathbb C [\mathfrak t] / I_{\lambda}$ as graded $A$-modules.
\end{corollary}

\begin{proof}[Proof of Corollary \ref{SC} modulo Theorem \ref{Shoji}]
The module $K_{\lambda}$ is spanned by $L_{\lambda}$ by (\ref{defPtr}). Since $A \cong \mathbb C [ \mathfrak t ] \otimes _{\mathbb C} \mathbb C W$ and $L_{\lambda}$ is $W$-stable, the image of $\imath _{\lambda}$ is exactly the span of derivatives of $\mathsf{L}_{\lambda} = \imath _{\lambda} L _{\lambda} \left< - d _{\lambda} \right>$ in the space of (harmonic) polynomials. Therefore, the latter half of Theorem \ref{Shoji} implies the former part of the assertion. The latter part of the assertion is a standard consequence of the former part of the assertion (see e.g. \cite{Ens} B.1).
\end{proof}

For each $\lambda, \mu \in \Lambda$, we define $K _{\mu, \lambda} (t) \in \mathbb Q (t)$ by
$$[K _{\lambda} ^*] = \sum _{\mu \in \Lambda} \overline{[K_{\lambda} : L_{\mu}]} [L_{\mu}] = \sum _{\mu \in \Lambda} t ^{- d_{\lambda}} K _{\mu, \lambda} (t^{2}) [L_{\mu}] \in K ( A \mathchar`-\mathsf{gmod} ).$$

\begin{corollary}[Numerical part of Shoji's conjecture \cite{Sh4} 3.13]\label{SS}
For each $\lambda, \mu \in \Lambda$, $K _{\mu, \lambda} (t)$ is a polynomial with non-negative integer coefficients. In addition, $\{ K _{\mu, \lambda} (t) \} _{\lambda, \mu \in \Lambda}$ are the modified Kostka polynomials of type $\mathsf{B}$ attached to the limit symbols.
\end{corollary}

\begin{remark}\label{orSh}
{\bf 1)} Corollary \ref{SC}, Theorem 1.8 1), 4), and Corollary \ref{absBH} (cf. Proposition \ref{eNc}) implies that the graded $W$-character $t^{d_{\lambda}} [K_{\lambda}^*]$ (cf. Lemma \ref{As} 3)) of the ring $\mathbb C [\mathfrak t] /I_{\lambda}$ is computable from
\begin{equation}
K_{\mu,\lambda} ( t ) = \begin{cases} 0 & (\lambda \not\preceq \mu) \\ t^{d_{\lambda}/2} & (\lambda = \mu) \end{cases}, \hskip 3mm \text{ and } \hskip 3mm  [P:L] = {}^{\mathtt{t}} [K:L] [\widetilde{K} : K] [K:L],\label{exLS}
\end{equation}
where $[P:L] = ( [P_{\lambda} : L_{\mu}] ) _{\lambda,\mu \in \Lambda}$, $[K:L] = ( t ^{d_{\lambda}} K _{\mu, \lambda} (t^{-2}) ) _{\lambda,\mu \in \Lambda}$ are square matrices, and $[\widetilde{K} : K]$ is a diagonal matrix in Corollary \ref{absBH}. In fact, solving the equation (\ref{exLS}) determines $[K:L]$ and $[\widetilde{K} : K]$ simultaneously from $[P:L]$ in a unique fashion (the {\it Lusztig-Shoji algorithm}). We have
$$[P_{\lambda} : L_{\mu}] = \mathsf{gdim} \, \mathrm{hom} _W ( L_{\mu}, L_{\lambda} \otimes _{\mathbb C} \mathbb C [ \mathfrak t ]),$$
that is commonly called the {\it fake degree} in the literature (with $t$ replaced with $t^{1/2}$; cf. \cite{K4} \S 2). This is the original form of Shoji's conjecture \cite{Sh4} 3.13.\\
{\bf 2)} Corollary \ref{SS} is independently obtained by Shoji-Sorlin \cite{SS} (cf. Achar-Henderson \cite{AH} for a closely related result). Also, \cite{SS} contains another proofs of Corollaries \ref{espure} and \ref{parity}.
\end{remark}

\begin{proof}[Proof of Corollary \ref{SS} modulo Theorem \ref{Shoji}]
Each $K _{\mu, \lambda} (t)$ is a polynomial by Theorem \ref{eS} 6) and $K_{\lambda} ^i = H_{- i + d_{\lambda}} ( \mu^{-1} ( x _{\lambda} ) ) = \{ 0 \}$ for $i > d_{\lambda}$. It has non-negative integer coefficients as it counts the graded dimension of a module.

By \cite{K4} 2.17 and (\ref{eLSorth}), we deduce that $t ^{d_{\lambda} / 2} K _{\mu, \lambda} (t^{-1})$ is the Kostka polynomial attached to the preorder $\prec$ on $\Lambda = \mathsf{Irr} \, W$ in the sense of Shoji \cite{Sh4} 1.3. Hence, in order to identify $\{ K _{\mu, \lambda} (t) \} _{\lambda, \mu \in \Lambda}$ with the Kostka polynomials attached to the limit symbol, we need to identify the preorder relations on $\mathsf{Irr} \, W$ arising from the exotic Springer correspondence and the limit symbols, and to identify the dimensions of the exotic Springer fibers and the $a$-function arising from the latter context. These identifications follow from Achar-Henderson \cite{AH} 6.3 with \cite{Sh4} 1.4, and the fact that the $b$-function in \cite{AH} \S 5, our $d_{\lambda} / 2$, and $n ( {\bm \beta} ) = a ( {\bm \Lambda} ( {\bm \beta} ) )$ in \cite{Sh4} \S 3 are in common.
\end{proof}

\begin{proof}[Proof of Theorem \ref{Shoji}]
We apply Corollary \ref{absLS2} to deduce the first two equation.

We prove the remaining assertion. Let $M _{\lambda} := \imath _{\lambda} K _{\lambda} \left< - d _{\lambda} \right>$. Since $K_{\lambda}$ has simple head as a graded $A$-module, the $A$-module $M_{\lambda}$ is obtained by the $A$-module saturation of $\mathsf L _{\lambda}$. By (\ref{defPtr}) and Theorem \ref{KS} 1), the injectivity of $\imath _{\lambda}$ is equivalent to
\begin{equation}
\mathrm{ext} ^1 _A ( M _{\lambda}, L _{\mu} ) = \{ 0 \} \hskip 3mm \text{ for every } \hskip 3mm \lambda \prec \mu. \label{B-vanish}
\end{equation}
If $W_{\lambda} = W$ (recall that $W_{\lambda} = W_{x_{\lambda}}$ in Theorem \ref{eS} 9) by convention), then we have $L _{\lambda} \cong \mathsf{sgn}$ or $\mathsf{Ssgn}$ by \cite{K4} Fact 4.1. We prove the assertion (\ref{B-vanish}) in the both cases. These cases correspond to {\bf a)} $x_{\lambda} = 0$, and {\bf b)} $\mathbb O_{\lambda} = ( V_1 \setminus \{ 0 \} ) \oplus \{ 0 \} \subset \mathbb V$, respectively. For the case {\bf a)}, the assertion $M_{\lambda} = K _{\lambda} \left< - d_{\lambda} \right>$ is standard since the both sides are $\left( \mathbb C [ \mathfrak t ] / \left< \mathbb C [ \mathfrak t ] ^W _+ \right> \right)^*$ and $H _{- \bullet} ( G/B )$, respectively (cf. \cite{CG} \S 6.4). For the case {\bf b)}, the ($G$-equivariant) composition map $F \rightarrow \mathfrak N \rightarrow V_1$ is surjective, and is regular along $x_{\lambda}$. It follows that $\mu ^{-1} ( x _{\lambda} )$ is a $( d_{\lambda} / 2 )$-dimensional smooth projective variety. In particular, we have the Poincar\'e duality $K_{\lambda} ^* \cong K _{\lambda} \left< - d _{\lambda}\right>$, which intertwines the action of the hyperplane sections. Since $\dim \, \mathsf{L} _{\lambda} = 1$, the module $K _{\lambda}$ has simple head as a $\mathbb C [\mathfrak t]$-module. Therefore, we conclude that $K_{\lambda} \left< - d_{\lambda} \right>$ must have simple socle of degree $0$. By construction, we also have $M_{\lambda} ^0 \neq \{ 0 \}$. This forces $K _{\lambda} \left< - d _{\lambda}\right> \cong M _{\lambda}$.

In the below, we examine the case $W _{\lambda} \neq W$. If we have $\lambda \prec \mu$, then $d_{\lambda} > d _{\mu}$ (see Theorem \ref{eS} 3)) and Theorem \ref{eS} 9) implies $\mathbb C p _{\lambda} \not \subset L _{\mu}$ as $W_{\lambda}$-modules.

Let $A _{\lambda} := \mathbb C W_{\lambda} \ltimes \mathbb C [\mathfrak t] \subset A$. We have $A _{\lambda} \neq A$ by $W _{\lambda} \neq W$. We define $M_{\lambda} ^{\downarrow}$ as the $A _{\lambda}$-submodule of $M_{\lambda}$ generated by $\mathbb C p_{\lambda}$. We have
\begin{equation}
\mathbb C W M _{\lambda} ^{\downarrow} = M _{\lambda} \label{W-sat}
\end{equation}
by $\mathbb C W A _{\lambda} = A$. Since the both of $\mathbb C W$ and $\mathbb C W_{\lambda}$ are semi-simple algebras, a non-trivial element of $\mathrm{ext} ^1 _A ( M _{\lambda}, L _{\mu} )$ induces a non-trivial extension as $\mathbb C [\mathfrak t ]$-modules, and hence that as $A _{\lambda}$-modules. In particular, the non-vanishing of the LHS of (\ref{B-vanish}) implies
$$\mathrm{ext} ^1 _{A _{\lambda}} ( M _{\lambda} ^{\downarrow}, L _{\mu} ) \neq \{ 0 \} \hskip 3mm \text{ for some } \hskip 3mm \lambda \prec \mu$$
from (\ref{W-sat}). If we decompose $W_{\lambda} \cong W_{n_1} \times W _{n_2} \times \cdots \times W _{n_m}$ by the smaller type $\mathsf{BC}$ Weyl groups, then we have $A _{\lambda} \cong \boxtimes _{i=1} ^m A _{n_i}$ with $\sum _{i=1}^m n_i = n$. For each $1 \le i \le m$, we have a standard module $K_{\lambda_i}$ of $A_{n_i}$ corresponding to $L_{\lambda_i} \cong \mathsf{sgn}$ or $\mathsf{Ssgn}$ (as $W _{n_i}$-modules) so that $\boxtimes _{i = 1} ^m K_{\lambda_i}$ is generated by a $W _{\lambda}$-module isomorphic to $\mathbb C p _{\lambda}$.

Now we use induction on $n$ and assume that the assertion holds for strictly smaller ranks. Then, we deduce $M _{\lambda} ^{\downarrow} \cong \boxtimes _{i = 1} ^m K_{\lambda_i}$ as a graded $A _{\lambda}$-module (up to grading shifts). By Proposition \ref{fgd-k} (cf. Corollary \ref{kf}), each direct summand of the minimal projective resolution of $M _{\lambda} ^{\downarrow}$ as a graded $A _{\lambda}$-module is of the form $\boxtimes _{i = 1} ^m P _{\gamma _i}$ with $\gamma _i \preceq \lambda _i$ for every $i = 1,\ldots, m$ up to a grading shift. Since the head of $P _{\gamma_i}$ is $L_{\gamma _i}$ as a $\mathbb C W _{n_i}$-module, we conclude that
\begin{equation}
\mathrm{hom} _{W_{\lambda}} ( \boxtimes _{i = 1} ^m L _{\gamma _i}, L _{\mu} ) \neq \{ 0 \}\label{nv}
\end{equation}
for some $\lambda \prec \mu$ if (\ref{B-vanish}) fails. By Theorem \ref{eS} 9), the $W _{\lambda}$-module $\boxtimes _{i = 1} ^m L _{\gamma _i}$ is realized inside $\mathbb C [ \mathfrak t ^* ]$ only at degree
$$- \sum _{i=1} ^m d _{\gamma _i} \le - \sum _{i=1} ^m d _{\lambda _i} = - d_{\lambda} < - d _{\mu}.$$
This shows that (\ref{nv}) cannot happen. Therefore, we conclude $\mathrm{ext} ^1 _A ( M _{\lambda}, L _{\mu} ) = \{ 0 \}$ also in the case $W _{\lambda} \neq W$. This proceeds the induction and finishes the proof. 
\end{proof}

\begin{corollary}[Achar-Henderson \cite{AH} Conjecture 6.4 (1)]\label{parity}
For each $\lambda, \mu \in \Lambda$, we have $[K_{\lambda} : L _{\mu}] = t^k Q ( t^4 )$ for some $k \ge 0$ and $Q \in \mathbb N [t]$.
\end{corollary}

\begin{proof}
By \cite{K4} Fact 4.1 (1) and (6), if $L_{\lambda}$ corresponds to a bi-partition $( \lambda ^{(0)}, \lambda^{(1)} )$, then each $L_{\mu} \subset ( \mathfrak t ^* ) ^{\otimes l} \otimes _{\mathbb C} L_{\lambda}$ $(l \ge 0)$ corresponds to a bi-partition $( \mu ^{(0)}, \mu ^{(1)} )$ with $|\mu^{(0)}| - |\lambda^{(0)}| \equiv l \mod 2$. Since $P_{\lambda} \cong \mathbb C [\mathfrak t] \otimes L _{\lambda}$ as graded $W$-modules (cf. \cite{K4} \S 1), we deduce $[ P_{\lambda} : L _\mu ] = t ^{k'} Q' ( t ^4 )$ for some $k' \ge 0$ and $Q' \in \mathbb N [\![t]\!]$. As $K_{\lambda}$ is a quotient of $P_{\lambda}$, we conclude the result.
\end{proof}

\section*{{\normalsize Appendix A: Coinvariants and the Lieb-McGuire systems}}
\renewcommand{\thesection}{\Alph{section}}
\setcounter{section}{1}
\setcounter{theorem}{0}

We work in the setting of section five. We set $\alpha ^{\vee} _i := \epsilon _i - \epsilon_{i+1}$ for each $1 \le i \le n$ (with $\epsilon _{n+1} := 0$). Let $s_1,\ldots,s_n \in W$ be the corresponding simple reflections. We fix two parameters $m, r \in \mathbb R$. Let $\mathcal H_{r,m}$ be the algebra that contains the group ring $\mathbb C W$ and the polynomial ring $\mathbb C [ \mathfrak t ] = \mathbb C [\epsilon _1,\ldots, \epsilon_n]$ with the following properties:
\begin{enumerate}
\item $\mathcal H_{r,m} \cong \mathbb C W \otimes _{\mathbb C} \mathbb C [ \mathfrak t ]$ as vector spaces;
\item For each $1 \le i \le n$ and each $f \in \mathbb C [ \mathfrak t ]$, we have
$$s_i \cdot f - {}^{s_i} f \cdot s_i = \begin{cases} r \frac{f-{}^{s_i}f}{\alpha^{\vee}_i} & (i \neq n)\\ m r \frac{f-{}^{s_n}f}{\alpha^{\vee}_n} & (i = n) \end{cases},$$
where $f \mapsto {}^w f$ ($w \in W$) is the natural $w$-action on $\mathbb C [ \mathfrak t ]$.
\end{enumerate}

Let $\mathsf e$ be the idempotent of $\mathbb C W$ corresponding to the trivial representation. We have an isomorphism $\mathbb C [ \mathfrak t ] \ni f \mapsto f \mathsf e \in \mathcal H _{r,m} \mathsf e$, by which we regard $\mathbb C [ \mathfrak t ]$ as a representation of $\mathcal H _{r,m}$.

We define an anti-isomorphism $\dagger : \mathcal H _{r,m} \rightarrow \mathcal H _{-r,m}$ as:
$$s _i ^{\dagger} := s_i, \text{ and } \epsilon _j ^{\dagger} = - \epsilon _j \text{ for every } 1 \le i,j \le n.$$

The center of $\mathcal H_{r,m}$ is isomorphic to $\mathbb C [ \mathfrak t ] ^W$ via the natural inclusion (see \cite{Lu-aff}), and hence each $W$-orbit $W \gamma$ (that we may simply denote by $\gamma$) of $\mathfrak t$ defines a central character $\mathbb C [ \mathfrak t ] ^W \to \mathbb C$ ($=: \mathbb C _{\gamma}$). Let $\mathsf{Irr} _{\gamma} \, \mathcal H _{r,m}$ be the set of isomorphism classes of irreducible $\mathcal H _{r,m}$-modules with a central character $\gamma$.

\begin{theorem}[Heckman-Opdam \cite{HO}]\label{HO}
Let $\mathcal R$ be the space of $C ^{\infty}$-functions with respect to $\xi_1,\ldots,\xi_n$ for which $\epsilon _i$ acts by $\frac{\partial}{\partial \xi _i}$. We set $\mathbb C [ \mathfrak t ^* ] := \mathbb C [\xi_1,\ldots, \xi_n] \subset \mathcal R$.
\begin{enumerate}
\item There exists a $\mathcal H _{-r,m}$-action on $\mathcal R$ so that the pairing
\begin{equation}\label{ndpair}
\mathbb C [ \mathfrak t ] \times \mathcal R \ni ( P, f ) \mapsto (Pf) (0) \in \mathbb C
\end{equation}
interchanges the $\mathcal H_{r,m}$-module structures on $\mathbb C [ \mathfrak t ]$ with the $\mathcal H_{- r,m}$-module structures of $\mathbb C [ \mathfrak t^* ]$ and $\mathcal R$ as:
$$( T P, f ) = ( P, T ^{\dagger} f) \text{ for every } T \in \mathcal H _{r,m}, P \in \mathbb C [ \mathfrak t ], \text{ and } f \in \mathcal R;$$
\item For each $\gamma \in \mathfrak t$, there exists a unique non-zero function $\phi _{\gamma} \in \mathcal R $ up to scalar such that: 
$$
\mathbb C [ \mathfrak t ] ^W \phi _{\gamma} \cong \mathbb C _{\gamma} \text{ as } \mathbb C [ \mathfrak t ] ^W \text{-modules; and } W \text{ acts on } \mathbb C \phi_{\gamma} \text{ by } \mathsf{triv};$$
\item If we set $M_{\gamma, m} := \mathcal H _{-r,m} \phi _{\gamma}$, then it is irreducible as a $\mathcal H _{-r,m}$-module.
\end{enumerate}
\end{theorem}

The algebra $\mathcal H_{r,m}$ specializes to $A$ (in section five) by setting $r = 0$.

We consider the following condition $(\star)$ on $a = (s, \vec{q}) \in \mathbb G = G \times ( \mathbb C ^{\times} )^2$:
\begin{itemize}
\item[$(\star)_0$] $\vec{q} = ( q^m, q )$, where $q = e^r \in \mathbb R _{> 1}$ and $m \in \mathbb R$;
\item[$(\star)_1$] $s = \exp ( \gamma )$ with $\gamma \in \mathfrak t ( \mathbb R )$.
\end{itemize}

\begin{theorem}[Standard modules, \cite{K1} \S 7, \S 9]\label{eDLst}
Assume that $a \in \mathbb G$ satisfies $(\star)$. Let $v = v_1 \oplus v_2 \in \mathbb V$ such that $a . v = v$. We have a $\mathcal H_{-r,m}$-module
$$M _{(a,v)} := H _{\bullet} ( \mu^{-1} ( v ) ^a ),$$
that is isomorphic to $K_{\lambda}$ for some $\lambda \in \Lambda$ as $W$-modules.
\end{theorem}

\begin{theorem}[eDL correspondence, \cite{K1}, \S 10]\label{eDL}
Assume that $a \in \mathbb G$ satisfies $(\star)$. Then, we have a one-to-one correspondence
\begin{align*}
\mathsf{Irr} _{\gamma} \mathcal H _{-r,m} \ni L _{(a,v)} \longleftrightarrow v \in Z _{G} (s) \backslash \mathbb V^{a}.
\end{align*}
Moreover, $L _{(a,v)}$ is a quotient of $M_{(a,v)}$ as a $\mathcal H _{-r,m}$-module.
\end{theorem}

\begin{lemma}\label{lower}
Let $L \subset \mathbb C [ \mathfrak t^* ]$ be a homogeneous $W$-submodule isomorphic to $L _{\lambda}$ for some $\lambda \in \Lambda$. Then, we have
$$K_{\lambda} \left< - d _{\lambda} \right> \subset A L \hskip 4mm \text{inside} \hskip 4mm \mathbb C [ \mathfrak t ^* ].$$
\end{lemma}

\begin{proof}
By Theorem \ref{eS} 9), $\mathbb C p_{\lambda} \subset K _{\lambda} \left< - d_{\lambda} \right> \subset \mathbb C [ \mathfrak t^* ]$ is a one-dimensional representation of $W_{\lambda}$ that is its maximal degree realization. Let $\mathbb C [ \mathfrak t ^* ] ^{W _{\lambda}}$ denote the $W _{\lambda}$-invariant part of $\mathbb C [ \mathfrak t ^* ]$.

There exists $q \in L$ for which $\mathbb C W _{\lambda} q \cong \mathbb C W _{\lambda} p _{\lambda}$ as (one-dimensional) $W _{\lambda}$-modules. It follows that we have a factorization
$$q = p _{\lambda} \cdot r \hskip 3mm \text{ with } \hskip 3mm r \in \mathbb C [ \mathfrak t ^* ] ^{W _{\lambda}}.$$
There exists a homogeneous element $Q^+ \in \mathbb C [ \mathfrak t ]$ such that $Q ^+ q = 1 \in \mathbb C [ \mathfrak t ]$ by the non-degeneracy of the pairing (\ref{ndpair}) restricted to $\mathbb C [\mathfrak t^*] \subset \mathcal R$. Here we can rearrange $Q^+$ if necessary to assume that $\mathbb C Q ^+$ is isomorphic to $\mathbb C p _{\lambda}$ as a $W _{\lambda}$-module. It follows that $Q ^+$ admits a factorization $Q ^+ = P Q$, where $P \in \mathbb C [ \mathfrak t ]$ is the minimal degree realization of $\mathbb C p _{\lambda}$ inside non-negatively graded ring $\mathbb C [ \mathfrak t ]$, and $Q$ is $W _{\lambda}$-invariant. By the comparison of degrees, we conclude that $0 \neq Q q \in \mathbb C p _{\lambda}$, which implies the result.
\end{proof}

\begin{theorem}\label{vanish}
Let $\gamma \in \mathfrak t$. Then, there exists $\lambda \in \Lambda$ so that the vanishing order of $M _{m, \gamma}$ along $0$ induces a graded $W$-module structure equal to $K _{\lambda} \left< - d_{\lambda} \right>$.
\end{theorem}

\begin{remark}
In Theorem \ref{HO}, the function $\phi _{\gamma}$ gives rise to a unique (up to scalars) solution of the Lieb-McGuire system so that $\mathbb C [\mathfrak h] ^W$ acts by $\gamma$ along the set of regular points of $\mathfrak h$ (see \cite{HO} for detail). In this sense, the vanishing order filtration of $M_{m,\gamma}$ measures the structure of the Taylor series of $\phi _{\gamma}$.
\end{remark}

\begin{proof}[Proof of Theorem \ref{vanish}]
For each $f \in \mathcal R$, let $\mathsf{lt} \, f$ denote the maximal degree non-zero homogeneous component of the Taylor expansion of $f$ along $0$ (remember that our degree counting on $\mathcal R$ is nonpositive). Then, the vanishing-order filtration of $M_{m, \gamma}$ is transformed into the graded structure of the module
$$\mathrm{gr} \, M _{m,\gamma} := \{ \mathsf{lt} \, f \mid f \in M _{m,\gamma}\}.$$
Since $M_{m,\gamma}$ is $\mathbb C [\mathfrak t]$-stable, so is the space $\mathrm{gr} \, M _{m, \gamma}$. Here \cite{HO} formula (2.1) implies that $s_i \in W$ acts on $\mathbb C [ \mathfrak t ^* ]$ by letting $s_i$ acts by the natural homogeneous action and adds some lower order terms. It follows that $\mathrm{gr} \, M _{m,\gamma} \subset \mathbb C [ \mathfrak t ^* ]$ is an $A$-submodule. Here $M _{m, \gamma}$ contains $\mathsf{triv}$ as a $W$-module. It implies that the irreducible module $M_{m,\gamma}$ is isomorphic to some standard module in the sense of Theorem \ref{eDLst} (cf. \cite{CK} 1.20). Hence, Theorem \ref{eDL} asserts that $\mathrm{gr} \, M _{m, \gamma} \cong K _{\lambda}$ as a $W$-module for some $\lambda \in \Lambda$ (note that Theorem \ref{KS} 1) determines such $\lambda$ uniquely only from the $W$-module structure of $\mathrm{gr} \, M _{m, \gamma}$). By Theorem \ref{eS} 8), we deduce that the image of the natural composition map
$$\varphi : L _{\lambda} \hookrightarrow \mathrm{gr} \, M _{m,\gamma} \subset \mathbb C [ \mathfrak t ^* ]$$
lands in degree $\le - d _{\lambda}$. Since $\dim \, \mathrm{Hom}_W ( L _{\lambda}, M_{m, \gamma} ) = 1$, we conclude that $\mathrm{Im} \, \varphi$ is homogeneous. If $\mathrm{Im} \, \varphi = \mathsf L _{\lambda}$ (i.e. the degree is $- d _{\lambda}$ by Theorem \ref{eS} 8)), then we have $\mathrm{gr} \, M _{m, \gamma} = K _{\lambda} \left< - d_{\lambda} \right>$ by the inclusion as $A$-modules and the comparison of dimensions.

If $\mathrm{Im} \, \varphi$ is of degree $< - d _{\lambda}$, then Lemma \ref{lower} implies that $\dim \, \mathrm{Hom}_W ( L _{\lambda}, M_{m, \gamma} ) \ge 2$. Hence, this case cannot occur as desired.
\end{proof}

\section*{{\normalsize Appendix B: Proof of Theorem \ref{critpure}}}
\renewcommand{\thesection}{\Alph{section}}
\setcounter{section}{2}
\setcounter{theorem}{0}

This appendix is entirely devoted to the proof of Theorem \ref{critpure}. We assume the setting of section one. In particular, we assume $(\spadesuit)$ in Condition \ref{spadesuit}, $(\clubsuit)_1$ in Condition \ref{clubsuit}, and work over an algebraic closure of a finite field. For each $\lambda \in \Lambda$, we have a natural morphism $\psi_{\lambda} : H ^{\bullet} _{G} ( \{ \mathrm{pt} \} ) \rightarrow H ^{\bullet} _{G} ( \mathbb{O} _{\lambda} )$ of graded algebras. Let us restate our statement to be proved:

\begin{theorem}[Theorem \ref{critpure}]\label{acritpure}
The sheaf $\mathsf{IC} _{\lambda}$ $($normalized as weight zero$)$ is pointwise pure of weight zero for every $\lambda \in \Lambda$ if the following two conditions hold:
\begin{enumerate}
\item[$a)$] The following spectral sequence is $E_2$-degenerate for each $\lambda = ( \underline{\lambda}, \xi ) \in \Lambda$:
$$E_2 = \bigoplus _{\mu = ( \underline{\lambda}, \zeta ) \in \Lambda} \mathrm{Hom} _{C_{\lambda}} ( \xi, H _{\mathsf{Stab}_G (x_{\lambda})^{\circ}} ^{\bullet} ( \{ x_{\lambda} \} ) \otimes _{\mathbb C} \zeta) \boxtimes K _{\mu} \Rightarrow \widetilde{K} _{\lambda}. \hskip 20mm (\ref{leray})$$
\item[$b)$] We have $\ker \, \psi _{\lambda} \not\subset \ker \, \psi _{\gamma}$ for every $\gamma \prec \lambda \in \Lambda$;
\end{enumerate}
Here the condition $a)$ can be replaced by its variant:
\begin{enumerate}
\item[$a)'$] For each $\lambda, \mu \in \Lambda$, the stalk of $\mathsf{IC}_{\lambda}$ along $\mathbb O_{\mu}$ satisfies the parity vanishing.
\end{enumerate}
\end{theorem}

We first prove that $a)'$ implies $a)$.

We assume $a)'$. Since $G_{\lambda} ^{\circ}$ is an affine algebraic group, we deduce that the odd degree part of $H _{G_{\lambda} ^{\circ}} ^{\bullet} ( \{ \mathrm{pt} \} )$ is zero. It follows that we cannot have a non-zero differential among distinct terms of the spectral sequence
$$E_2 (\gamma) := H _{\mathsf{Stab}_G (x_{\lambda})^{\circ}} ^{\bullet} ( \{ x_{\lambda} \} ) \otimes _{\mathbb C} i^! _{\lambda} \mathsf{IC} _{\gamma} \Rightarrow H ^{\bullet} _{\mathsf{Stab}_G (x_{\lambda}) ^{\circ}} ( \{ x _{\lambda} \}, i^! _{\lambda} \mathsf{IC} _{\gamma} )$$
for each $\gamma \in \Lambda$. Since $E_2 = \mathrm{Hom} _{C_{\lambda}} ( \xi, \bigoplus _{\gamma \in \Lambda} L_{\gamma} \boxtimes E_2 ( \gamma ))$ as spectral sequences of vector spaces, we conclude $a)$.

Now we assume $a)$ and $b)$ to deduce that $\mathsf{IC} _{\lambda}$ (normalized as weight zero) is pointwise pure of weight zero along $\mathbb O _{\gamma}$ for each $\gamma \in \Lambda$.

\begin{lemma}\label{cc}
The conditions $(\spadesuit)$ and $(\clubsuit)_1$ are stable under the restriction to a $G$-stable closed subvariety $\mathfrak Y \subset \mathfrak X$.
\end{lemma}

\begin{proof}
The condition $(\spadesuit)$ is clear. For the condition $(\clubsuit)_1$, let us name the closed embedding $i : \mathfrak Y \hookrightarrow \mathfrak X$. For every $\mathcal E, \mathcal F \in D ^b_G (\mathfrak Y)$, we have
\begin{align*}
\mathrm{Ext} ^{\bullet} _{D ^b _G ( \mathfrak Y )} ( \mathcal E, \mathcal F ) \cong \mathrm{Ext} ^{\bullet} _{D ^b _G ( \mathfrak Y )} ( i^* i_* \mathcal E, \mathcal F ) \cong \mathrm{Ext} ^{\bullet} _{D ^b _G ( \mathfrak X )} ( i_* \mathcal E, i_* \mathcal F )
\end{align*}
by adjunction. Hence, if we consider the idempotent $e \in \mathrm{End} _{G} ( \mathcal L )$ so that $e \mathcal L = \bigoplus _{\lambda \in \Lambda, \mathbb O _{\lambda} \subset \mathfrak Y} L _{\lambda} \boxtimes \mathsf{IC} _{\lambda}$, then we have $A_{(G,\mathfrak Y)} = e A_{(G,\mathfrak X)} e$. It follows that the condition $(\clubsuit)_1$ for $\mathfrak X$ implies that of $\mathfrak Y$ as required.
\end{proof}

We assume that $\mathsf{IC} _{\lambda}$ is not pointwise pure of weight zero along $\mathbb O _{\gamma}$ to deduce contradiction (we have $\lambda \succ \gamma$ by the support condition). By rearranging the choice of $\lambda$ and $\gamma$ and replacing $\mathfrak X$ with $\overline{\mathbb O _{\lambda}}$ if necessary, we can assume:

\begin{itemize}
\item[$c)$] The orbit $\mathbb O _{\lambda}$ is open dense in $\mathfrak X$;
\item[$d)$] For every $\gamma \prec \mu \prec \lambda$, the sheaf $\mathsf{IC} _{\mu}$ is pointwise pure of weight zero along $\mathbb O _{\gamma}$.
\end{itemize}

\begin{theorem}[Theorem \ref{transfer} + Corollary \ref{r-std-filt}]\label{at}
For each $\delta \in \Lambda$, we have a complex of graded projective $A$-modules
$$(Q ( \mathbb C _{\delta} ), d) : \cdots \rightarrow Q ( \mathbb C _{\delta} ) _{-2} \stackrel{d_{-2}}{\rightarrow} Q ( \mathbb C _{\delta} ) _{-1} \stackrel{d_{-1}}{\longrightarrow} Q ( \mathbb C _{\delta} ) _{0} \longrightarrow 0$$
with the following properties:
\begin{enumerate}
\item We have $Q ( \mathbb C _{\delta} ) _a \cong \mathrm{Ext} ^\bullet _G ( \mathrm{gr} _a \, \mathbb C _{\delta}, \mathcal L )$ as a graded $A$-module with pure weight $- a$;
\item We have an isomorphism of graded $A$-modules:
$$\widetilde{K} _{\delta} \cong \mathrm{Ext} ^\bullet _G ( \mathbb C _{\delta}, \mathcal L ) \cong H ^{\bullet} ( Q ( \mathbb C _{\delta} ), d );$$
\item For each $a < 0$, the projective $A$-module $Q ( \mathbb C _{\delta} ) _{a}$ does not contain grading shifts of $P_{\epsilon}$ with $\epsilon \sim \delta$ as its direct factor.
\end{enumerate}
\end{theorem}

\begin{proof}
We explain the diffusion from Theorem \ref{transfer}. The complex $(Q ( \mathbb C _{\delta} ), d)$ vanishes in positive degree by Corollary \ref{r-std-filt} 1). The last property follows from Corollary \ref{r-std-filt} 2).
\end{proof}

For each $\mu \in \Lambda$, we have
$$K _{\gamma} \supset L_{\mu} \boxtimes i_{\gamma} ^! \mathsf{IC_{\mu}}.$$
By the purity of $H _{G_{\gamma} ^{\circ}} ^{\bullet} ( \{ \mathrm{pt} \} )$ and the condition $a)$, we deduce that $[ H ^{i} ( Q ( \mathbb C _{\gamma} ), d ) : L _{\mu}] \neq 0$ holds if and only if the vector space $i_{\gamma} ^! \mathsf{IC_{\mu}}$ contains nonzero weight $(-i)$-part. By the condition $d)$ and the definition of $K _{\gamma}$, we have
$$[H ^{i} ( Q ( \mathbb C _{\gamma} ), d ) : L _{\mu}] = 0 \hskip 5mm \text{ if } \mu \not\sim \lambda$$
for each $0 \neq i \in \mathbb Z$. In addition, there exists $i \neq 0$ so that $H ^{i} ( Q ( \mathbb C _{\gamma} ), d ) \neq \{ 0 \}$ since $i_{\gamma} ^! \mathsf{IC} _{\lambda}$ is not pure of weight zero by assumption. Henceforth we fix $i \neq 0$ with $H ^{i} ( Q ( \mathbb C _{\gamma} ), d ) \neq \{ 0 \}$. By the condition $a)$, we deduce that $H ^{i} ( Q ( \mathbb C _{\gamma} ), d )$ is a torsion-free $H _{G_{\gamma}} ^{\bullet} ( \{ \mathrm{pt} \} )$-module.

The graded $A$-module $\widetilde{K}_{\lambda}$ admits a finite length projective resolution by Theorem \ref{at} 1)--2) (cf. Proposition \ref{P_to_K_finite}). In addition, every irreducible constituent of $\widetilde{K}_{\lambda}$ is a grading shift of $L_{\mu}$ with $\mu \sim \lambda$ by the condition c). Moreover, Theorem \ref{at} 3) implies that
$$\mathrm{ext}^1 _A ( \widetilde{K}_{\lambda}, L _{\mu} ) = \{ 0 \} \hskip 5mm \text{ for every }  \mu \sim \lambda.$$

This shows that $\widetilde{K}_{\lambda}$ is an indecomposable projective cover of $L_{\lambda}$ in the category of graded $A$-modules with its composition factors in $\{ L_{\mu} \left< j \right> \} _{\mu \sim \lambda, j \in \mathbb Z}$. Therefore, the graded $A$-module $H ^{i} ( Q ( \mathbb C _{\gamma} ), d )$ is a quotient of a direct sum of grading shifts of $\widetilde{K}_{\mu}$ with $\mu \sim \lambda$. By the condition $a)$, the graded $A$-module $\widetilde{K}_{\mu}$ with $\mu \sim \lambda$ is a torsion-free $H _{G_{\lambda}} ^{\bullet} ( \{ \mathrm{pt} \} )$-module, and we have $\mathrm{End}_A ( \widetilde{K} _{\mu} ) \cong H _{G_{\lambda}} ^{\bullet} ( \{ \mathrm{pt} \} )$ as graded algebras. This induces a $H _{G_{\lambda}} ^{\bullet} ( \{ \mathrm{pt} \} )$-action on $H ^{i} ( Q ( \mathbb C _{\gamma} ), d )$, that is non-trivial (but not necessarily faithful).

We have a $H _{G} ^{\bullet} ( \{ \mathrm{pt} \} )$-action on $H ^{i} ( Q ( \mathbb C _{\gamma} ), d )$ defined through the natural map $H _G ^{\bullet} ( \{ \mathrm{pt} \} ) \rightarrow A$ which lands on the center of $A$ (see (\ref{SS1})). These three actions must be compatible through $\psi_{\lambda}$ and $\psi _{\gamma}$ since the $H _G ^{\bullet} ( \{ \mathrm{pt} \} )$-action represents the action of (some part of) the center of $A$. In particular, if $\xi \in H ^{\bullet} _G ( \{ \mathrm{pt} \} )$ belongs to $\ker \psi _{\lambda}$, then the action of $\psi _{\lambda} ( \xi )$ must be zero, and hence $\psi _{\gamma} ( \xi )$ annihilates $H ^{i} ( Q ( \mathbb C _{\gamma} ), d )$. Since $\psi _{\gamma} ( \xi )$ cannot have torsion (as seen above), we conclude $\psi _{\gamma} ( \xi ) = 0$. It follows that $\ker \, \psi _{\lambda} \subset \ker \, \psi _{\gamma}$. This violates the condition $b)$, and we deduce a contradiction. Therefore, we conclude that every $\mathsf{IC} _{\lambda}$ must be pointwise pure as required.

\end{document}